\documentclass[11pt,a4paper,reqno]{amsart}

\usepackage[applemac]{inputenc}
\usepackage[T1]{fontenc}
\usepackage{amsmath}
\usepackage{amsthm}
\usepackage{amsfonts}
\usepackage{amssymb}
\usepackage{graphicx}
\usepackage{amsbsy}
\usepackage{mathrsfs}
\usepackage{bbm}
\usepackage{xcolor}
\newtheorem{theorem}{Theorem}[section]
\newtheorem{lemma}[theorem]{Lemma}

\newtheorem{proposition}[theorem]{Proposition}
\newtheorem{corollary}[theorem]{Corollary}

\theoremstyle{remark}
\newtheorem{remark}[theorem]{\it \bf{Remark}\/}
\newenvironment{acknowledgement}{\noindent{\bf Acknowledgement.~}}{}
\numberwithin{equation}{section}

\catcode`@=11
\def\section{\@startsection{section}{1}%
  \z@{1.5\linespacing\@plus\linespacing}{.5\linespacing}%
  {\normalfont\bfseries\large\centering}}
\catcode`@=12
\newcommand{\be}{\begin{equation}}
\newcommand{\ee}{\end{equation}}
\newcommand{\bea}{\begin{eqnarray}}
\newcommand{\eea}{\end{eqnarray}}
\newcommand{\bee}{\begin{eqnarray*}}
\newcommand{\eee}{\end{eqnarray*}}

\def\RR{\mathbb{R}}

\def\fref#1{{\rm (\ref{#1})}}

\catcode`@=11
\def\supess{\mathop{\operator@font Sup\,ess}}
\catcode`@=12

\def\RR{\mathbb{R}}

\def\e{\varepsilon}

\def\bar#1{{\overline #1}}
\def\fref#1{{\rm (\ref{#1})}}

\def\R2+{\RR ^2_+}

\def\lim{\mathop{\rm lim}}

\def\sup{\mathop{\rm sup}}

\def\ba{\begin{array}}
\def\ea{\end{array}}

\newcommand\ep{\epsilon}

\def\lab{\label}

\author[C. Collot]{Charles Collot}
\address{Courant Institute of Mathematical Sciences, New York University, 251 Mercer Street, New York, NY 10003, United States of America.}
\email{cc5786@nyu.edu}
\author[P. Germain]{Pierre Germain}
\address{Courant Institute of Mathematical Sciences, New York University, 251 Mercer Street, New York, NY 10003, United States of America.}
\email{pgermain@cims.nyu.edu}

\keywords{Nonlinear Schr\"odinger equation, random initial data, kinetic wave equation, Feynman graphs}
\subjclass[2010]{primary, 35Q55 70K70 35C20, secondary 35B34 81Q30} 

\title{On the derivation of the homogeneous kinetic wave equation}

\begin{document}

\begin{abstract} 
The nonlinear Schr\"odinger equation in the weakly nonlinear regime with random Gaussian fields as initial data is considered. The problem is set on the torus in any dimension greater than two. A conjecture in statistical physics is that there exists a kinetic time scale depending on the frequency localisation of the data and on the strength of the nonlinearity, on which the expectation of the squares of moduli of Fourier modes evolve according to an effective equation: the so-called kinetic wave equation. When the kinetic time for our setup is $1$, we prove this conjecture up to an arbitrarily small polynomial loss. When the kinetic time is larger than $1$, we obtain its validity on a more restricted time scale. The key idea of the proof is the use of Feynman interaction diagrams both in the construction of an approximate solution and in the study of its nonlinear stability. We perform a truncated series expansion in the initial data, and obtain bounds in average in various function spaces for its elements. The linearised dynamics then involves a linear Schr\"odinger equation with a corresponding random potential whose operator norm in Bourgain spaces we are able to estimate on average. This gives a new approach for the analysis of nonlinear wave equations out of equilibrium, and gives hope that refinements of the method could help settle the conjecture.
\end{abstract}

\maketitle
\tableofcontents

\section{Introduction}

\subsection{Presentation of the problem}

We consider the nonlinear Schr\"odinger equation
\begin{equation} \label{NLS} \tag{NLS}
\left\{ \begin{array}{l}
i \partial_t u + \Delta u=  \lambda^2 |u|^2 u \\
u(t=0) = u_0
\end{array} \right.
\end{equation}
set on the torus: $x \in \mathbb{T}^d = \mathbb{R}^d / (2\pi \mathbb{Z}^d)$, where $d \geq 2$, and with initial data of the form
\begin{equation}
\label{data}
u_0(x) = \epsilon^{d/2} \sum_{k \in \mathbb{Z}^d} A(\epsilon k) G(k) e^{i k \cdot x},
\end{equation}
where $A \in \mathcal{C}^\infty_0(\mathbb{R}^d,\mathbb{R})$ and $(G(k))_{k \in \mathbb{Z}^d}$ are independent standard centred complex Gaussians (see Section~\ref{notations}). The normalization is such that on average $\| u_0 \|_{L^2} \sim 1$, and, by Khinchine's inequality, $\| u_0 \|_{L^p} \sim_p 1$ on average as well, for any $p < \infty$.

Heuristic derivations, to which we will come back, show that, as $\epsilon \lambda \to 0$ (weakly nonlinear regime) and $\epsilon \to 0$ (high frequency limit), scaling properly the square modulus of the Fourier coefficients of $u$, they satisfy on average (with $\mathbb E$ the expectation)
\begin{equation}
\label{quetzalresplendissant}
{\epsilon^{-d} \mathbb{E} \left| \widehat{u}( \lfloor \epsilon k \rfloor) \left( T_{kin}t' \right) \right|^2 \longrightarrow \rho(t',k), \qquad \qquad \mbox{for }T_{kin} = \frac{1}{\epsilon^2 \lambda^4}}
\end{equation}
where $\rho$ solves the kinetic wave equation
\begin{equation}
\tag{KWE} \label{KWE}
\left\{
\begin{array}{l}
{\partial_{t'} \rho(t',k) = \mathcal{C}[\rho](k) ,}\\
\rho(0,k) = |A(k)|^2,
\end{array}
\right.
\end{equation}
with the collision operator given by
\begin{equation}
\begin{split}
\label{collisionoperator}
\mathcal{C}[\rho](k) =  & \ {c_0}\int_{(\mathbb{R}^d)^3} \delta(k+\ell-m-n) \delta(|k|^2 + |\ell|^2 - |m|^2 - |n|^2)\\
& \qquad \qquad \qquad \qquad  \rho(k) \rho(\ell) \rho(m) \rho(n) \left[ \frac{1}{\rho(k)} + \frac{1}{\rho(\ell)} - \frac{1}{\rho(m)} - \frac{1}{\rho(n)} \right] \,d\ell \,dm \,dn.
\end{split}
\end{equation}
and
\be \label{def:c0}
c_0 = 2^{2-2d} \pi^{1-2d}.
\ee
The present paper is an attempt to show that the prediction~\eqref{quetzalresplendissant} is verified on non trivial time scales, even though the kinetic time scale $T_{kin}$ is for the moment out of reach.

{Note that, for a quick formal computation to derive the kinetic time, one can consider the first nonlinear correction to the linear dynamics obtained by Duhamel formula:
\begin{align*}
\hat u(k) \ &\approx \ e^{-it|k|^2}\ep^{\frac d2} A(\ep k)G(k)\\
&\qquad -ie^{-it|k|^2} \epsilon^{\frac 32 d}\lambda^2 \sum_{\ell+m+n=k}\int_0^t  e^{is(|k|^2-|\ell |^2+|m|^2-|n|^2)}ds A(\ep \ell)A(-\ep m)A(\ep n)G(\ell)\overline{G(-m)}G(n).
\end{align*}
On the nearly resonant set $||k|^2-|\ell|^2+|m|^2-|n|^2|\lesssim t^{-1}$ the time integral is of order $t$. The sum of gaussians restricted to this set is performed over $\ep^{2-2d}t^{-1}$ integer points, hence of order $\ep^{1-d}t^{-1/2}$ from square root cancellation. The first nonlinear correction is thus of order $\ep^{d/2}(\lambda^2\ep \sqrt t)$ and compares with the initial datum for $t\sim T_{kin}$. This derivation has limits, see below.}

\subsection{Relevant time scales, parameters range}
\begin{itemize}
\item $\displaystyle T_{kin} = \frac{1}{\epsilon^2 \lambda^4}$: characteristic time scale for the kinetic wave equation.
\item $\displaystyle T_{lin} = \epsilon^2$: characteristic time scale for the linear part of the equation 
\item $\displaystyle T_{nonlin} = \frac{1}{\lambda^2}$:  characteristic time-scale for the nonlinear part of the equation (given that we expect $\|u\|_{L^\infty} \sim 1$ - up to logarithmic losses).
\end{itemize}

We will consider the regime where
$$
T_{lin} \ll T_{kin} \quad \iff \quad T_{nonlin} \ll T_{kin} \quad \iff \quad T_{lin} \ll T_{nonlin} \quad \iff \quad \lambda \epsilon \ll 1,
$$
which means that the regime we are considering is weakly nonlinear. In this regime,
$$
T_{lin} \ll T_{nonlin}\ll T_{kin} .
$$
We choose a power type relation between the strength of the nonlinearity and the frequency
\be \label{regimeeplambda}
\lambda=\ep^{-\gamma} \qquad \mbox{with} \qquad 0<\gamma < \frac 12
\ee
so that:
$$
T_{nonlin}<1<T_{kin}.
$$
Finally, let us consider resonances. The resonance modulus is classically given by
$$
\Omega(k,\ell,m,n) = |k|^2 - |\ell|^2 + |m|^2 - |n|^2.
$$
Since $k,\ell,m,n \in \mathbb{Z}^d$, $\Omega$ takes integer values; and since $|k|, |\ell|, |m|, |n| \lesssim \frac{1}{\epsilon}$, it satisfies $|\Omega| \lesssim \frac{1}{\epsilon^2}$. This means that only time scales $T$ such that
\begin{equation}
\label{chardonneret}
\epsilon^2 \ll T \ll 1
\end{equation}
are susceptible to yield the kinetic wave equation in an asymptotic regime. Indeed, if $T \lesssim \epsilon^2$, resonances are hardly playing any role; while if $T \gtrsim 1$, the resonance moduli $\Omega$ cannot be equidistributed modulo $\frac{1}{T}$, which prevents from taking the discrete to continuous limit in frequency.

\subsection{Rescaling to a large torus}
In the problem formulated above, the equation is set on a torus of size $1$, and $u_0$ has size $\sim 1$ in any $L^p$ (neglecting logarithmic factors if $p=\infty$), and varies on a typical scale $\sim \epsilon$. For the reader's convenience, we show how it can be rescaled to fit the setup adopted in~\cite{FGH,BGHS0,BGHS}.

We now let $\epsilon = L^{-1}$, and rescale $u$ by setting
$$
u'(t',x') =u' \left( \frac{t}{\epsilon^2},\frac{x}{\epsilon} \right) = \epsilon^{d/2} u \left( t,x \right).
$$ 
The equation solved by $u'$ is now
\begin{equation}
\label{NLS2} \tag{NLS2}
i \partial_{t'} u' +  \Delta_{x'} u' = (\lambda')^2 |u'|^2 u' \qquad \mbox{with} \qquad \lambda' = \lambda \epsilon^{1-\frac{d}{2}}.
\end{equation}
In this new setting, the domain has size $L$, $u'$ is of size $\sim L^{\frac{d}{p} - \frac{d}{2}}$ in $L^p$, and varies on a typical scale $\sim 1$. These orders of magnitude coincide with the framework adopted in~\cite{BGHS}.

To convert results from~\eqref{NLS} to~\eqref{NLS2}, observe that the time scale $t_0$ for~\eqref{NLS} corresponds to $t'_0 = \frac{t_0}{\epsilon^2}$ for~\eqref{NLS2}. In particular, the kinetic time scale $T_{kin} = \frac{1}{\lambda^4 \epsilon^2}$ for~\eqref{NLS} corresponds to the kinetic time scale $\displaystyle T'_{kin} = \frac{1}{(\lambda')^4 \epsilon^{2d}} = \frac{L^{2d}}{(\lambda')^4}$ for~\eqref{NLS2}.

\subsection{Background} \subsubsection{Derivation of the kinetic wave equation} The kinetic wave equation was first derived by Peierls~\cite{Peierls} in the context of Quantum Mechanics, and, in a different form and independently by Hasselmann~\cite{Hasselmann1,Hasselmann2}, who worked on water waves. The theory was then revived by Zakharov and his collaborators~\cite{ZLF}, giving a very versatile framework, which applies to a number of Hamiltonian systems satisfying the basic assumptions of weak nonlinearity, high frequency (or infinite volume limit), and phase randomness. Introductions to this research field can be found in Nazarenko~\cite{Nazarenko} and Newell-Rumpf~\cite{NR}. 

As far as rigorous mathematics go, a fundamental work, building up on~\cite{EY0}, is due to Lukkarinen and Spohn~\cite{LS1}, who reach the kinetic time scale for the correlations in a system at statistical equilbrium, obtaining a linearized version of the kinetic wave equation (see also~\cite{Faou} for another derivation of the linearized system). More heuristic considerations by the same authors on this question  can be found in~\cite{Spohn2} and~\cite{LS2}. We are indebted to them for several ideas in the present work, around Feynman diagrams and their estimation. 

The question of the derivation of~\eqref{KWE} for random data out of statistical equilibrium was first tackled in~\cite{BGHS}, which is the source of a several ideas which we extend here (construction of an approximate solution and control of the remainder). Compared to this work, we introduce a number of novel ideas: use of Bourgain spaces, finer analysis of the expansion through Feynman diagrams, and estimate of the average operator norm of the linearized operator. This allows us to get arbitrarily close to the kinetic time scale.

Another appearance of~\eqref{KWE} from \eqref{NLS} is when applying a random forcing to the equation and a dissipation \cite{ZL}. Recently, Dymov and Kuksin \cite{DK1,DK2} (see also the survey \cite{DK3}) studied the truncated series expansion for this model. They showed how in a weakly nonlinear context with non-vanishing dissipation the three first terms of the expansion were agreeing with ~\eqref{KWE}, studying also higher order terms but without controlling the full solution.

\subsubsection{Derivation of related collisional kinetic models} The kinetic wave equation is to phonons, or linear waves, what the Boltzmann equation is to classical particles. The derivation of the Boltzmann equation was put on a rigorous mathematical foundation with the foundational work of Lanford~\cite{Lanford} and its more recent clarification by Gallagher-Saint-Raymond-Texier~\cite{GST}. A few articles deal with the derivation of kinetic models for quantum particles \cite{BCEP1,BCEP2, BCEP3}; this question is closely related to the derivation of the kinetic wave equation, but is harder, 
since~\eqref{NLS} can be thought of as an intermediary step between a quantum mechanical model with a large number of particles, and kinetic theory.

Another strand of research focuses on linear dispersive models with random potential, from which one can derive the linear Boltzmann equation on a short time scale~\cite{Spohn1}, and the heat equation on a longer time scale~\cite{ESY1,ESY2}. 

Finally, let us mention the possibility of deriving Hamiltonian models for NLS with deterministic data in the infinite volume, or big box, limit~\cite{FGH,BGHS0}.

\subsubsection{Nonlinear dispersive equations with random data} The main difficulty in the present paper is to control a solution to~\eqref{NLS} with random data of typical frequency $\ep^{-1}$ and amplitude $\lambda$, in the regime \fref{regimeeplambda}. But other regimes are of interest, and have been considered previously: at fixed frequency size and vanishing amplitude ($\ep=1$, $\lambda \rightarrow 0$), correlations of the solutions were studied in \cite{dST}. Replacing $A(\ep k)=c_k$ in \fref{data} by $\ep=1=\lambda$, Burq and Tzvetkov obtained the existence of solutions below the critical regularity for $c_k$ \cite{BT1,BT2} (see \cite{CO,NS} for the Schr\"odinger equation). This result exploits the fact that randomizing Fourier series yields improved $L^p$ bounds (see \cite{PZ} for the historical work on the torus and \cite{AT,BL} for other Riemannian manifolds), which we also use here. Other measures than the ones generated by pushing forward $\sum c_k G(k) e^{ik.x}$ can be considered as well, as the Gibbs measure which was showed to be invariant in low dimensions \cite{Bourgain2,Bourgain3,DengNahmodYue}, and interpolation between Gibbs and white noise \cite{OQ}. Since we are here out of equilibrium, controlling the flow requires uniform bounds as $\ep\rightarrow 0$ instead of invariant measures properties.

\subsection{Main result} 

{Our main result is the existence of a solution over the nontrivial time range $[0,1]$, and the validity of the approximation by the kinetic wave equation on this time interval. The existence of~\eqref{NLS} up to time $1$ (on the complement of an exceptional set) is non trivial, and given by the first part of the Theorem; as for the local well-posedness of~\eqref{KWE}, it has been established for instance in~\cite{GIT}. We refer to \cite{EV} for global weak solutions and other qualitative results for \eqref{KWE}.
}

{Since $T_{kin}=\ep^{-2+4\gamma} $, as $ \gamma$ approaches $1/2$, the time interval $[0,1]$ gets close, up to an arbitrarily small polynomial loss, from the kinetic time. Note that $T_{kin}$ of order $1$ means to $\lambda^4 \ep^2$ of order $1$. }

{
\begin{theorem} \label{th:main} Pick any $0<\gamma <\frac 12$ and $\eta,\mu> 0$. Then there exist $0<\ep^*\leq 1$ and $\nu>0$ depending on $\gamma$, $\mu$ and $\eta$ such that for all $0<\ep\leq \ep^*$, a set $E$ of measure $\mathbb{P}(E) > 1 -  \epsilon^{\mu}$ exists such that
\begin{itemize}
\item \emph{Existence of solution:} On $E$, the equation~\eqref{NLS} with initial datum \fref{data} admits a solution $u \in \mathcal{C}^\infty([0,1] \times \mathbb{T}^d)$.
\item \emph{Validity of the kinetic wave equation:} Furthermore, denoting $\mathcal{A} = |A|^2$, for any $\lambda^{-2}\leq t \leq \ep^\eta$:
\be \label{serinducap}
\sum_{k \in \mathbb{Z}^d} \left| \mathbb{E} \left[ \mathbbm{1}_{E}\left[ |\widehat{u}(k,t)|^2 - \epsilon^d \mathcal{A}(\epsilon k) - \epsilon^d \frac{t}{T_{kin}}  \mathcal{C}(\mathcal{A})(\epsilon k) \right] \right] \right| \lesssim \frac{\epsilon^\nu t }{T_{kin}},
\ee
where $\mathcal{C}$ is the collision operator defined in~\eqref{collisionoperator} and $c_0$ is the constant given by \fref{def:c0}
\end{itemize}
\end{theorem}
}

{
\begin{remark} \emph{Bounds on the solution} Since the regime considered in the theorem is supercritical (in the classical, deterministic sense) for $d \geq 5$, the mere existence of a smooth solution on a unit time interval is not trivial. The proof of the theorem actually gives quantitative bounds: namely, for some $s>\frac{d}{2}-1$ and $b>\frac{1}{2}$, the norm of the solution in $X^{s,b}([0,1]\times \mathbb T^d)$ (defined in \fref{id:bourgainspaces}) is $O(\epsilon^{-\mu})$. As usual with $X^{s,b}$ spaces, this should be understood as follows: there exists an extension $\widetilde{u}$ of $u$ to $\mathbb{R} \times \mathbb T^d$ such that $\| \widetilde{u} \|_{X^{s,b}} \lesssim \epsilon^{-\mu}$.
\end{remark}
}

{
\begin{remark}
\emph{On the range of parameters}:
\begin{itemize}
\item[(i)] \emph{Restriction to small times $t\lesssim 1$.} For $t$ of order $1$ or greater, the resonance moduli are not equidistributed on a scale $\frac{1}{t}$, preventing one to take the discrete to continuous limit which yields the collision integral (see below). Therefore, one needs $t\ll 1$ in \fref{serinducap} to obtain that the right hand side is of lower order compared to the size of the correction $\frac{t}{T_{kin}} \mathcal{C}(\mathcal{A})$ in the left hand side. For a similar reason of equidistribution for the dispersion relation of the Schr\"odinger equation, we cannot control the solution on a time interval $[0,T]$ with $T\gg 1$.
\item[(ii)] \emph{Restriction to large kinetic times $T_{kin}\gg 1$.} In the case $T_{kin}\lesssim 1$, the bounds we obtain for the approximate solution are powers of $\frac{1}{T_{kin}}$ due to degenerate low frequencies effects. This prevents both the convergence of our series expansion and its linear stability, so that we cannot cover this case. This explains the restriction $\gamma<\frac 12$.
\item[(iii)] \emph{Unit time interval for existence}. We need our time interval for existence to avoid equidistribution problems, and not to exceed the kinetic time. The first condition requires to consider small times $t\lesssim 1$ from (i) above, for which the second condition is always fulfilled from (ii) above. Hence the time interval $[0,1]$.
\item[(iv)] \emph{Nontrivial nonlinear effects}. The condition $\gamma>0$ gives $T_{nonlin}=\lambda^{-2}\leq 1$, so that the time interval $[0,1]$ exceeds the nonlinear time and the existence result is nontrivial. The condition $\lambda^{-2}\leq t\leq 1$ in \fref{serinducap} is due to the following. Before the nonlinear time $\lambda^{-2}$, the nonlinear effects did not kick in, so that the kinetic wave equation is irrelevant on $[0,\lambda^{-2}]$.
\item[(v)] \emph{Validity of (KWE), almost up to the kinetic time}. The above shows that the prediction~\eqref{quetzalresplendissant} is satisfied on time scales $\ll 1$; namely, the expectation of $\epsilon^{-d} |\widehat{u}_k|^2$ and the solution of~\eqref{KWE} are as close as expected. In particular, we are able to treat kinetic time scale of order $\ep^{-\mu}$ for $\mu>0$ as close as we want to $0$ (that is, for $0<\frac 12 -\gamma \ll 1$). This shows how $T_{kin}=1$ is attainable with arbitrarily small polynomial loss. 
\end{itemize}
\end{remark}
}

\begin{remark}
\emph{Relaxing the set up}: Many assumptions in the above theorem could be relaxed. The randomization through Gaussians is convenient since it allows the use of Wick's formula, but other randomizations should also be possible. The function $A$ was taken in $\mathcal{C}^\infty_0$ to simplify the proof as much as possible, but much milder hypotheses should suffice. Other dispersion relations could be considered as well, and in particular under better equidistribution properties our present analysis could be strengthened and intervals of time $[0,T]$ with $T\gg 1$ could be considered. The torus was chosen to be rational, but our proof applies verbatim to irrational tori - thanks to the work of Bourgain and Demeter~\cite{BD}, which gives Strichartz estimates for them. It should be possible to make the size of the exceptional set exponentially small in $\epsilon$. Finally, as should be expected, whether the equation is focusing or defocusing does not change anything in our argument.
\end{remark}

\subsection{Strategy of the proof and plan of the article} We present here a caricature of the proof, describing only the main ideas. \textit{In particular, we simplify formulas by omitting the less important terms.}

\bigskip

\noindent \underline{Approximation and error} The heart of the proof is to build a sufficiently good approximation of the solution. It is obtained by renormalising the phase (Wick ordering) and iterating Duhamel's formula, yielding a truncated series expansion.

\emph{Wick ordering} It can be explained formally as follows. $u_0$ given by \fref{data} is a stationary Gaussian field: the law of $u_0(x)$ as a random variable is independent of $x$. Assuming $u$ remains of the form \fref{data} one gets\footnote{Precisely, $\frac{(2\pi)^d}{2}\mathbb E \langle |u|^2u,v\rangle=\mathbb E \langle \| u\|_{L^2}^2u,v\rangle-\sum_k \mathbb E|\hat u_k|^2\mathbb E \overline{\hat v_k}\hat u_k$ in this case, and the second term of order $O(\ep^d\| v\|_{L^2})$ under \fref{data} is negligible.} that $\mathbb E\langle |u|^2u,v\rangle \approx \frac{2}{(2\pi)^d} \mathbb E \langle \| u\|^2_{L^2}u,v\rangle $ for $v$ a test Gaussian field via Wick formula \fref{wickformula}. This and mass conservation $\| u\|_{L^2}=\| u_0\|_{L^2}$ produces the first approximation
$$
iu_t+\Delta u\approx \frac{2}{(2\pi)^d}\lambda^2\| u_0 \|_{L^2}^2u
$$
for which the sole effect is a phase modulation $u(t)=e^{-it \lambda^2 \frac{2}{(2\pi)^d}\| u_0\|_{L^2}^2}e^{-it\Delta}u_0$. This in particular does not change the statistical properties of the solution. We therefore first renormalise the solution, setting $u=e^{-it \lambda^2 \frac{2}{(2\pi)^d}\| u_0\|_{L^2}^2}v$.

\emph{Approximation} Then, we define $u^0 = e^{it\Delta} u_0$ and, for $n \geq 1$,
$$
\left\{ 
\begin{array}{l}
i \partial_t u^n + \Delta u^n = \sum_{i+j+k = n-1} P (u^i, \overline{u^j}, u^k) \\
u(t=0) = u_0.
\end{array}
\right.
$$
Here $P$ is a trilinear operator which essentially corresponds to Wick ordering. The solution splits into approximate solution and error:
$$
u = e^{-it \lambda^2 \frac{2}{(2\pi)^d} \| u_0 \|_{L^2}^2} \left(u^{app} + u^{err} \right), \qquad \mbox{with} \quad u^{app} = \sum_{n=0}^N u^n.
$$
The approximate solution satisfies the equation up to an error $\mathcal{E}$
$$
i\partial_t u^{app} + \Delta u^{app} = |u^{app}|^2 u^{app} + \mathcal{E}.
$$
and the equation satisfied by $u^{err}$ is then (forgetting below the terms due to Wick ordering)
$$
i\partial_t u^{err} + \Delta u^{err} = \underbrace{2 |u^{app}|^2 u^{err} + (u^{app})^2 \overline{u^{err}}}_{\displaystyle \mathcal{L}(u^{err})} + \underbrace{2 |u^{err}|^2 u^{app} + (u^{err})^2 \overline{u^{app}}}_{\displaystyle \mathcal{B}(u^{err})} + \underbrace{|u^{err}|^2 u^{err}}_{\displaystyle \mathcal{T}(u^{err})} + \mathcal{E},
$$
where $\mathcal{L}$, $\mathcal{B}$, and $\mathcal{T}$ stand for the linear, bilinear, and trilinear terms respectively.

There remains to apply a fixed point argument in a (slightly modified) Bourgain space $X^{s,b}$ in order to show that $u^{err}$ is sufficiently small. In order to carry out this plan, we need bounds on $u^{app}$, $\mathcal{L}$, $\mathcal{B}$ and $\mathcal{T}$. Obtaining these bounds is the aim of the remainder of the paper.

The above is explained in greater detail and with full formulas in Section~\ref{sectionproofmaintheorem}.

\bigskip

\noindent \underline{Feynman diagrams} A formula for $u^1$ is easy to write: in physical space
$$
u^1 = i\lambda^2  \int_0^t e^{i(t-s) \Delta} P (e^{it\Delta} u^0,\overline{e^{it\Delta} u^0},e^{it\Delta} u^0)\,ds,
$$
and in frequency
$$
\widehat{u^1}(k) = - \frac{ i\lambda^2 \epsilon^{\frac{3d}{2}}}{(2\pi)^d} e^{-it|k|^2} \sum_{k = k_{0,1} + k_{0,2} + k_{0,3}} \int_0^t e^{is\Omega} G(k_{0,1}) \overline{G(-k_{0,2})}  G(k_{0,3}) A(\epsilon k_{0,1} ) \overline{A(-\epsilon k_{0,2} )}  A(\epsilon k_{0,3} ) \chi \,ds,
$$
where $\chi$ is the symbol of $P$, and $\Omega = |k|^2 - |k_{0,1}|^2 + |k_{0,1}|^2 - |k_{0,1}|^2$. However, writing the formula giving $u^3$ or $u^4$ is already extremely lengthy. This motivates the introduction of Feynman diagrams, which provide an intuitive and analytically efficient way of representing iterates. For instance, $u^1$ is represented by

\vspace*{0.5cm}
\begin{center}
\includegraphics[width=11cm]{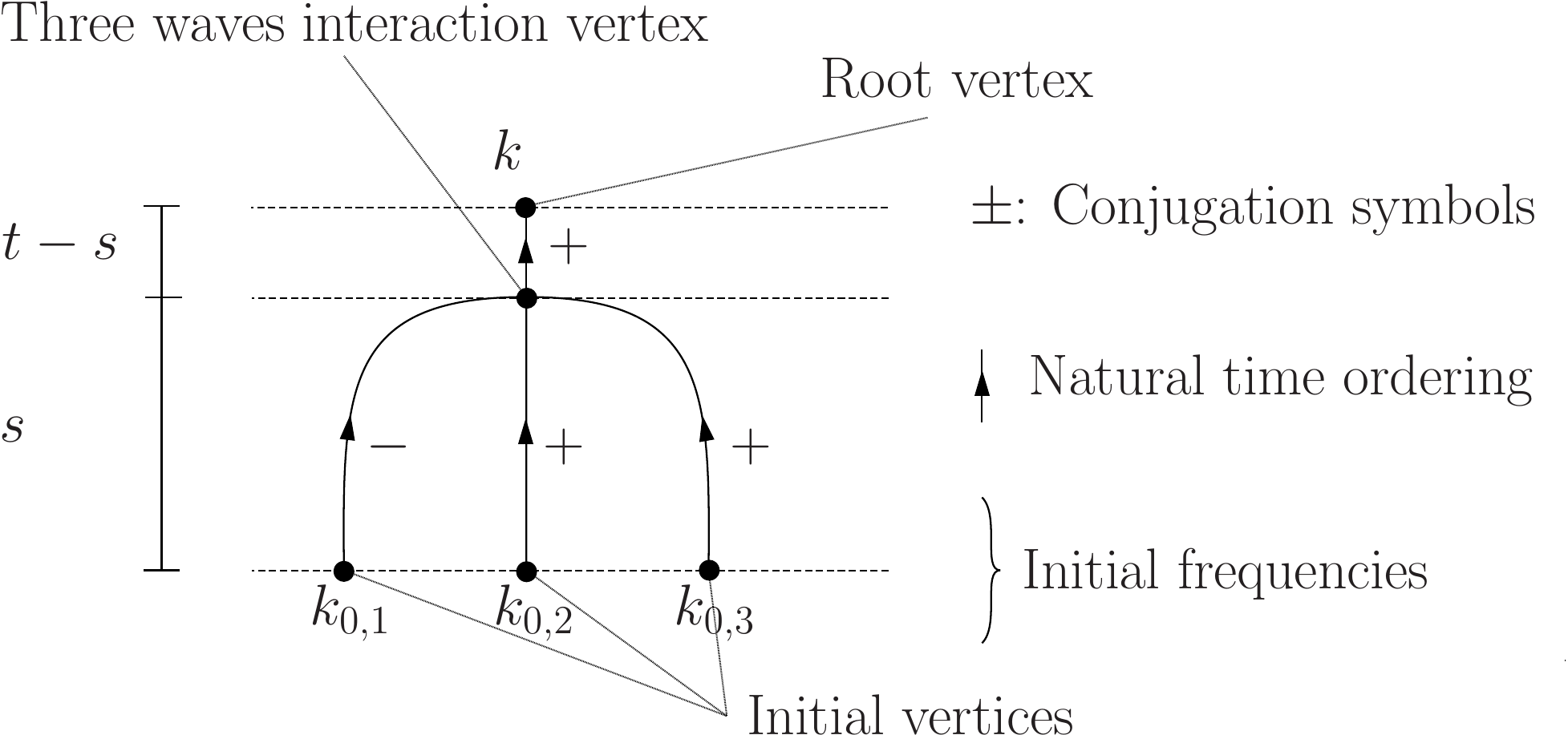}
\end{center}
\vspace*{0.5cm}

We refer to Section~\ref{sectionfeynman} for a full presentation of Feynman diagrams.

\bigskip

\noindent \underline{Bound on the iterates}
 In order to understand the approximate solution, we are led to estimating
$$
\mathbb{E}(\| u^n \|_{L^2}^2) \quad \mbox{and more generally} \quad \mathbb{E}(\| u^n \|_{L^p}^p), \mathbb{E}(\| u^n \|_{X^{s,b}}^2)
$$
(expectations are always taken with respect to the randomization of the data; an application of the Bienaym\'e-Chebyshev inequality allows to transfer these bounds to most data, up to a small loss, and excluding an exceptional set).
Applying Wick's formula~\eqref{wickformula} given the choice of the data~\eqref{data}, it appears that taking the expectation forces initial frequencies to be pairwise equal. At the level of Feynman diagrams, this can be represented intuitively as follows

\vspace*{0.5cm}
\begin{center}
\includegraphics[width=17cm]{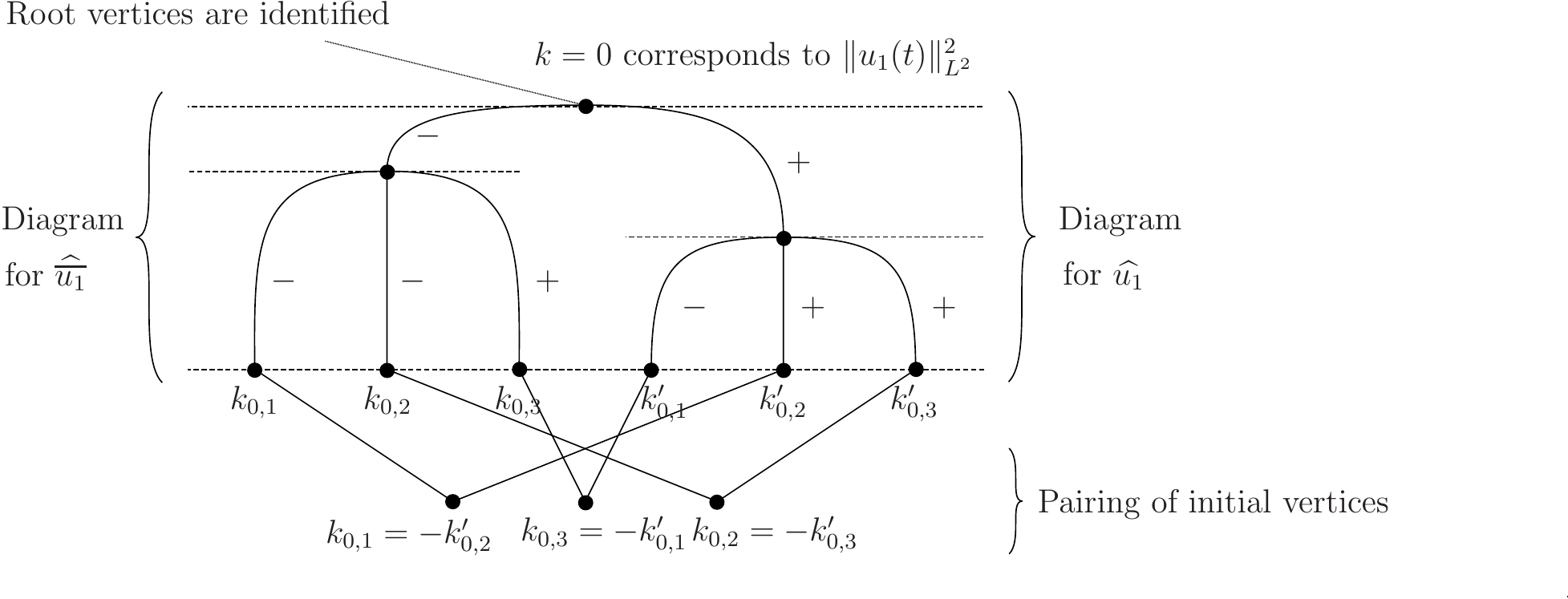}
\end{center}
\vspace*{0.5cm}

Using this representation, and some graph combinatorics, we can show that $\mathbb{E}(\| u^n \|_{L^2}^2)$ is bounded by a factor times $\left( \frac{1+t}{T_{kin}} \right)^n$. Though the factor depends on $n$, this is indicative of the fact that the series $u^n$ converges on the right time-scale. This is achieved in Section~\ref{sectionapproximate}. 

\bigskip

\noindent \underline{Bound on the error $\mathcal{E}$} As for the error $\mathcal{E}$, it is bounded by a constant times $t^{1/2}\left(\frac{1}{T_{kin}}\right)^{\frac N2}$ (for $t\geq 1$ such a bound would be $\left(\frac{t^2}{T_{kin}}\right)^{\frac N2}$ due to equipartition problems).

Since $T_{kin}$ is given by a negative power of $\epsilon$, this bound corresponds to an arbitrarily large power of $\epsilon$ by choosing $N$ sufficiently large; but note that the implicit constant depends on $N$.

\bigskip

\noindent \underline{Bound on the linear term $\mathcal{L}$} Turning to the linear term $\mathcal{L}$, we need to show that $\int_0^t e^{i(t-s) \Delta} \mathcal{L} \,ds$ has operator norm $\ll 1$ in $X^{s,b}$, so that it can be absorbed in a Neumann series type argument. By the theory of $X^{s,b}$ spaces,
$$
\left\|  \int_0^t e^{i(t-s) \Delta} \mathcal{L}\, ds \right\|_{X^{s,b} \to X^{s,b}} \lesssim \| \mathcal{L} \|_{X^{s,b} \to X^{s,b-1}}, 
$$
so it suffices to focus on the operator norm of $\mathcal{L}$ from $X^{s,b}$ to $X^{s,b-1}$. Using that $v^{app}$ is supported in Fourier on a ball of radius $C \epsilon^{-1}$, it suffices to bound $\mathcal{L}$ when it is furthermore localized on cubes of comparable size.

The idea is now to view $\mathcal{L}$ as a random operator, and rely on a classical trick in random matrix theory. First, consider $\mathcal{L}^* \mathcal{L}$, which is positive, and thus has its operator norm bounded by its trace. Since this operator is self-adjoint, we can estimate its norm by raising it to a high power, and taking the trace. Finally, taking in addition the expectation and using H\"older's inequality, we obtain that
$$
\mathbb{E} \| \mathcal{L} \| \lesssim \left[ \mathbb{E} \operatorname{Tr}  (\mathcal{L}^* \mathcal{L})^N  \right]^{\frac{1}{2N}}.
$$
The key is now that $\mathbb{E} \operatorname{Tr}  (\mathcal{L}^* \mathcal{L})^N$ can be represented through Feynman graphs, and estimated using the same tools as for the quantities we already discussed. This estimate is performed in Section~\ref{sectionlinear}.

\bigskip

\noindent \underline{Bound on the nonlinear terms $\mathcal{B}$ and $\mathcal{T}$} It is achieved by using the classical nonlinear theory of $X^{s,b}$ spaces in Section~\ref{sectionnonlinear}.

\subsection{Obstacles to reaching the kinetic time scale} The present work misses the kinetic time scale by an arbitrarily small power of $\lambda^{-1}$ or $\epsilon$, in the event where $T_{kin}$ is close to 1; it is therefore natural to ask what remains to be done to fully justify the kinetic wave equation, and whether this restriction on $T_{kin}$ can be lifted. We believe that a number of very significant difficulties need to be overcome. First, one should fully classify paired graphs, extract dominant pairings, and bound subdominant ones; this should be within reach by following the ideas in~\cite{LS2}. Second, the number theoretical analysis performed here is very crude, and only allows for a time scale $T_{kin}$ very large with respect to, but close to $1$; a much deeper analysis is required in order to prove or disprove the equidistribution of resonances on various scales. Third, small powers of $\lambda^{-1}$ or $\epsilon$ are lost in many steps  of the argument: for instance, when using the Bienaym\'e-Tchebycheff equality to exclude an exceptional set, or when bounding the operator norm of the linearized operator through the trace of the iterated $TT^*$;  bypassing this loss seems to require completely new ideas. Finally, the nonlinear analysis performed here could break down at the kinetic time scale: as is usual in the study of nonlinear dispersive problems in Bourgain $X^{s,b}$ spaces, reaching the scaling formally requires choosing $b = 1/2$, which is in principle forbidden.

At a more fundamental level, most results on nonlinear dispersive problems with random initial data (see~\cite{Bourgain2,Bourgain3,BT2,LS2}) make use, one way or another, of the invariant measure to reach long time scales. This is obviously not possible here since the problem is out of statistical equillibrium.

\bigskip \bigskip

\acknowledgement The independent work of Deng and Hani~\cite{DengHani} was made public at the same time as the present article. The main results of both papers are close, and the proofs share several ideas. Deng and Hani's viewpoint is informed by the local well-posedness theory for nonlinear dispersive problems with random data (and in particular its recent extensions~\cite{DengNahmodYue}). These authors developed an original approach to the combinatorics of Feynman diagrams. They also considered times $> 1$ (in the scaling considered here), at which point the distinction between rational and irrational tori becomes crucial.

The second author is grateful to S. Meyerson for an illuminating conversation on the circle method.

\section{Notations} \label{notations}

\subsection{Time range} Throughout the rest of this paper, we assume that we study a solution over a time interval $[0,T]$ with:
\begin{equation}
\label{epsilont}
\epsilon^c \leq T \lesssim 1.
\end{equation}

\subsection{Inequalities}
For any two quantities $X$ and $Y$, we denote $X \lesssim Y$ if there exists a universal constant $C$ such that $X \leq C Y$. {We write $X\sim Y$ if $X\lesssim Y$ and $Y\lesssim X$.} We denote $X \lesssim_Z Y$ if the constant $C$ is allowed to depend on a further quantity $Z$. To keep notations under control, we do not systematically record the dependence of all the constants; in particular, we always omit the dependence on obvious quantities, such as the dimension $d$, or the size of the support of $A$ through which the data is defined.

Most estimates are valid up to subpolynomial factors in $\epsilon$. This is recorded by a small constant $\kappa>0$, for instance we denote
$$
X \lesssim_\kappa \epsilon^{-\kappa} Y.
$$
For ease of notation, we allow the value of $\kappa$ to change from one line to the next (provided it can always be taken arbitrarily small), and we sometimes denote $\lesssim$ instead of $\lesssim_\kappa$.

Due to the assumption~\eqref{epsilont}, subpolynomial factors in $T$ give subpolynomial factors in $\epsilon$.

\subsection{Fourier transform}
Given a function $f(x)$ on the torus $\mathbb{T}^d = \mathbb{R}^d / (2\pi \mathbb{Z}^d)$, the Fourier transform in space is denoted by
$$
\mathcal{F}(f)_k = \widehat{f}(k) = \frac{1}{(2\pi)^{d/2}} \int_{\mathbb{T}^d} f(x) e^{- i k \cdot x} \,dx, \qquad f(x) =  \frac{1}{(2\pi)^{d/2}} \sum_{k\in \mathbb{Z}^d} \widehat{f}(k) e^{ i k \cdot x}.
$$
Given a function $g(t,x)$ on $\mathbb{R} \times \mathbb{T}^d$, the Fourier transform in space-time is
\begin{align*}
& \widetilde{g}(\tau,k) =  \frac{1}{(2\pi)^{\frac{d+1}{2}}} \int_{\mathbb{R}} \int_{\mathbb{T}^d} g(t,x) e^{- i (t \tau +  k \cdot x)} \,dx\,dt, \\
& g(t,x) =   \frac{1}{(2\pi)^{\frac{d+1}{2}}} \sum_{k \in \mathbb{Z}^d} \int_{\mathbb{R}} \widetilde{g}(\tau,k) e^{i (t \tau +  k \cdot x)} \,d\tau,
\end{align*}
The Fourier multiplier $m(D)$ acts only in the space variable, through the definition
$$
m(D) f = \mathcal{F}^{-1} [m(k) \widehat{f}(k)].
$$
For $\epsilon>0$, $N\in 2^\mathbb{N}$,  and $n\in \mathbb{Z}^d$ ,we now define $C_{\epsilon,N}^n$ to be the cuboid of side $\frac{N}{\epsilon}$ and center $\frac{N}{\epsilon}n$. The characteristic function of this cube is denoted $\mathbf{1}_{C_{\epsilon,N}^n}$, and enables us to define the projection operators
$$
P_{\epsilon,1} = {\mathbf{1}}_{C_{\epsilon,2}^0}(D) \qquad \mbox{and if $N \geq 2$,} \qquad P_{\epsilon,N} = {\mathbf{1}}_{C_{\epsilon,2N}^0}(D) -  {\mathbf{1}}_{C_{\epsilon,N}^0}(D)
$$
as well as
$$
Q_{\epsilon,N}^n = \mathbf{1}_{C^n_{\epsilon,N}}(D).
$$
These operators are bounded on $L^p$ spaces, $1<p<\infty$ and provide decompositions of the identity:
$$
\sum_{N \in 2^\mathbb{N}} P_{\epsilon,N} = \operatorname{Id}, \qquad \sum_{n \in \mathbb{Z}^d} Q_{\epsilon,N}^n = \operatorname{Id}.
$$

\subsection{Lebesgue spaces}
The scalar product on $L^2$ is defined by
$$
\langle f \,,\,g\rangle = \int \overline{f} g\,dx.
$$
Space-time Lebesgue spaces on $[0,T] \times \mathbb{T}^d$ are given by
$$
\| u \|_{L^p_T L^q} = \| u \|_{L^p([0,T],L^q(\mathbb{T}^d))} = \left\| \| u(t,x) \|_{L^q(\mathbb{T}^d)} \right\|_{L^p(0,T)}
$$

\subsection{Bourgain spaces} We use scaled Sobolev spaces
$$
\| f \|_{H^s_\epsilon} = \| \langle \epsilon D \rangle^s f \|_{L^2}
$$
and their associated scaled Bourgain spaces:
\begin{equation} \label{id:bourgainspaces}
\| u \|_{X^{s,b}_\epsilon} = \| e^{-it\Delta} u(t) \|_{H^b_t H^s_{\epsilon,x}} =  \| \langle \epsilon k \rangle^s \langle \tau +|k|^2 \rangle^b \widetilde{u}(\tau,k) \|_{L^2(\mathbb{R} \times \mathbb{Z}^2)}.
\end{equation}
Basic properties of these spaces are given in Appendix \ref{Xsbbasics}.

\subsection{Randomization}We denote $\mathbb{P}$ and $\mathbb{E}$ for the probability of an event and expectation of a random variable. The random variables $(G(k))_{k\in \mathbb{Z}^d}$ are independent centred standard complex Gaussians: they satisfy for all $k,\ell \in \mathbb{Z}^d$.
$$
\mathbb{E} G(k) = 0, \qquad \mathbb{E} [G(k) G(\ell)] = 0, \qquad \mathbb{E} [G(k) \overline{G(\ell)}] = \delta(k-\ell)
$$
and the Wick formula
\begin{equation}
\label{wickformula}
\begin{split}
\mathbb{E} [ G(k_1) \dots G(k_r) \overline{G(k_{r+1}) \dots G(k_{r+s})}] = \# P,
\end{split}
\end{equation}
where $P$ is the set of admissible pairings, that is to say partitions of $\{ 1,\dots,2r \}$ into sets of the form $\{ i,j \}$ (unordered pairs), with $i \leq r$, $j \geq r+1$ and $k_i = k_j$. In particular, the above is zero if $r \neq s$.

\section{Proof of the main theorem}

\label{sectionproofmaintheorem}

{
The aim of this section is to prove the main theorem. This will rely on the combination of results which are proved in the following sections, namely Section~\ref{sectionfeynman} (construction of Feynman diagrams), Section~\ref{sectionapproximate} (bounds on the approximate solution), Section~\ref{sectionlinear} (boundedness in $X^{s,b}$ of the linearized operator around the approximate solution), Section~\ref{sectionnonlinear} (multilinear estimates in $X^{s,b}$). We take the results proved in these later sections for granted, and show how they can be combined to obtain our main result.}

\subsection{The frequency truncation for Wick ordering}
Define the truncation operator
\begin{align*}
& \mathcal{F} [P(a,b,c)](k) = \frac{1}{(2\pi)^d} \sum_{k_1+k_2+k_3=k}  \widehat a(k_1)\widehat b(k_2)\widehat c(k_3)  (1-\delta(k_1+k_2)-\delta(k_2+k_3)).
\end{align*}
This gives the decomposition of the product
$$
abc =P(a,b,c) + \frac{1}{(2\pi)^d} \langle \overline{a},b\rangle c + \frac{1}{(2\pi)^d} a \langle \overline{b},c\rangle.
$$

\subsection{Approximate solution and error }
The approximate solution is defined through the following iterative resolution scheme
\begin{equation}
\label{defun}
u^0 = e^{it\Delta} u_0 \qquad \mbox{and if $n \geq 1$,} 
 \left\{ \begin{array}{l l}\displaystyle i \partial_t u^n + \Delta u^n = \lambda^2 \sum_{ i+j+k=n-1}P (u^i,\overline{u^j},u^k), \\ u^n(0)=0, \end{array} \right..
\end{equation}
Defining further
$$
V^{i,j} =  \langle  u^i , u^{j}\rangle , \qquad V = \sum_{i,j \leq N} V^{i,j}=\left\| \sum_{n=0}^N u^n\right\|_{L^2}^2, \qquad \omega(t) =- \frac{2t}{(2\pi)^d} \| u_0\|^2_{L^2},
$$
our approximate solution will be
$$
u^{app} = \chi \left( \frac{t}{2} \right) e^{i \lambda^2 \omega} v^{app} \qquad \mbox{with} \qquad v^{app} = \sum_{n=0}^N u^n.
$$
Notice how we add a smooth cutoff function in the definition of $u^{app}$. Here, $\chi$ is a function in $\mathcal{C}^\infty_0$ such that $\chi = 1$ on $B(0,1)$ and $\chi = 0$ on $B(0,2)^\complement$; as a result, the cutoff only affects $t>2$, and the equation is unchanged for $t<2$).
Extracting the factor $e^{i \lambda^2 \omega}$ is sometimes called Wick ordering, and is classically used for random data problems, see for instance~\cite{Bourgain3}. 

Since the flow of~\eqref{NLS} preserves the mass $\| u(t) \|_{L^2}=\| u_0\|_{L^2}$, the approximate solution satisfies for $t<1$
\begin{align*}
 i\partial_t u^{app} + \Delta u^{app} - \lambda^2 |u^{app}|^2 u^{app}  &= E^N e^{i\lambda^2 \omega}+\frac{2}{(2\pi)^d} \lambda^2\left(2\mathfrak{Re}\langle u-u^{app},u^{app}\rangle +\left\| u -u^{app}\right\|_{L^2}^2\right)u^{app},
\end{align*}
(using that, for any $x,y \in \mathbb{C}$, $|x|^2 - |y|^2 = 2 \mathfrak{Re}((x-y)\overline{y}) + |x-y|^2$) where the error terms are given by
\begin{align*}
&E^N = \lambda^2 \left[- \sum_{\substack{i,j,k \leq N \\ i+j+k \geq N}} u^i \overline{u^j} u^k + \frac{2}{(2\pi)^d} \sum_{\substack{i,j,k \leq N \\ i+j+k \geq N}} V^{i,j} u^k \right].
\end{align*}
The solution $u$ can now be decomposed into approximation plus error
$$
u = u^{app} + u^{err}, \qquad \mbox{with} \qquad u^{err} = e^{i\lambda^2 \omega} v^{err},
$$
where $v^{err}$ solves
\begin{align*}
i\partial_t v^{err} + \Delta v^{err} + \frac{2 \lambda^2}{(2\pi)^d} \| u\|_{L^2}^2 v^{err} &= \lambda^2 \left( |v^{err} + v^{app}|^2 (v^{err} + v^{app}) - |v^{app}|^2v^{app} \right)  - E^N\\
&- \frac{2\lambda^2}{(2\pi)^d} \left(2\mathfrak{Re}\langle v^{err},v^{app}\rangle +\left\| v^{err}\right\|_{L^2}^2\right)v^{app}
\end{align*}
which can be written, for $t<1$,
\begin{align*}
i\partial_t v^{err} + \Delta v^{err} = \mathcal{L}(v^{err}) + \mathcal{B}(v^{err}) + \mathcal{T}(v^{err}) + \mathcal{E}
\end{align*}
where the linear, bilinear, trilinear, and error terms are given by
\begin{align*}
& \mathcal{L}(w) = \chi \left( t\right) \lambda^2 \left[ 2 |v^{app}|^2 w - 2(2\pi)^{-d} V w-2(2\pi)^{-d}\langle v^{app},w\rangle v^{app} + (v^{app})^2 \overline{w}-2(2\pi)^{-d}\langle w,v^{app}\rangle v^{app} \right] \\
& \mathcal{B}(w) = \chi(t) \lambda^2 \left[ 2 |w|^2 v^{app}-2(2\pi)^{-d}\| w\|_{L^2}^2v^{app}-2(2\pi)^{-d}\langle w,v^{app}\rangle w + w^2 \overline{v^{app}}-2(2\pi)^{-d}\langle v^{app},w\rangle w \right] \\
& \mathcal{T}(w) = \chi(t) \lambda^2 \left(|w|^2 w- 2(2\pi)^{-d} \| w \|_{L^2}^2w\right) \\
& \mathcal{E} = - \chi \left( t\right)   E^N.
\end{align*}
Notice how we once again added smooth cutoff functions in the definitions of the terms above; $\chi$ is still a function in $\mathcal{C}^\infty_0$ such that $\chi = 1$ on $B(0,1)$ and $\chi = 0$ on $B(0,2)^\complement$.

\subsection{Bounds on the expansion}

{Assuming $\epsilon \leq T \leq 1$, Proposition~\ref{pr:LpLpinter} implies the following corollary.}

\begin{corollary}
\label{aigleroyal}
Given $1\geq t,T > \epsilon$, $N \in \mathbb{N}$, $\mu>0$, $s>0$, $p \geq 2$, there exists $b > \frac{1}{2}$ and a set $E=E_{t,T,N,\mu,s,p}$ with probability $\mathbb{P}( E )\geq 1 - \epsilon^{\mu}$ such that on $E$, and if $n \leq N$,
\begin{align*}
& \| u^n(t) \|_{L^2} \lesssim_{N,\mu} \epsilon^{-\mu}t^{\frac 12}\left( \frac{1}{T_{kin}} \right)^{n/2} \\
& \| u^n(t) \|_{L^p_T L^p} \lesssim_{N,\mu,p} \epsilon^{-\mu} (\ep+T^2)^{\frac{1}{p}} \left( \frac{1}{T_{kin}} \right)^{\frac{n}{2}} (T^{1/2} \epsilon^{-1})^{\frac{1}{2} - \frac{1}{p}} \\
& \left\| \chi(t) \int_0^t e^{i(t-s)\Delta} E^N \,ds \right\|_{X^{s,b}_\epsilon} \lesssim_{N,\mu} {\epsilon^{-2\mu}} \epsilon^{- \frac{1}{4} - \frac{d}{4}}  \left( \frac{1}{T_{kin}} \right)^{\frac N2} \\
\end{align*}
\end{corollary}

{\begin{proof}
Considering for instance the $L^2$ bound, it follows from the combination of~\eqref{bd:fourierwnLp} and the Bienaym\'e-Chebyshev inequality
$$
\mathbb{P}( \| u^n(t) \|_{L^2} > \rho) \lesssim \rho^{-2} \mathbb{E} \| u_n(t) \|_{L^2}^2
$$
with $\rho = \epsilon^{- \mu} t^{1/2} T_{kin}^{-n/2}$ 
\end{proof}}

\begin{remark}
Note a loss factor in the $L^p$ estimates for $p>2$ in the above. It could be removed by a further refined analysis, but there is no need for it in the present paper as the error shows an arbitrarily large polynomial gain.
\end{remark}

\subsection{Bounds on the linear, bilinear and trilinear terms}

The following proposition gives a bound on $\mathcal{L}$, if one excludes an exceptional set.
\begin{proposition} \label{prop:lineaire}
If $N\in \mathbb{N}$, $\mu>0$, $s>0$, there exists $b > \frac{1}{2}$ and a set $E_{N,\mu,s}$ of probability $\mathbb{P}(E_{N,\mu,s}) > 1 - \epsilon^{\mu}$ on which the operator norm of $\mathfrak{L}$ can be bounded as follows:
$$
\left\| \chi(t)\int_0^t e^{i(t-s)\Delta} \mathcal{L} \,ds \right\|_{X^{s,b}_\epsilon \to X^{s,b}_\epsilon} \lesssim_{N,\mu}
\epsilon^{-\mu} \sqrt{\frac{1}{T_{kin}}}
$$ 
\end{proposition}

Turning to the nonlinear terms, they will be controled by the two following propositions.

\begin{proposition} \label{propbilinear} If $N \in \mathbb{N}$, $\mu>0$, $s> \frac{d}{2} - 1$, there exists $b>\frac{1}{2}$ and a set $E_{N,\mu,s,b}$ with probability $\mathbb{P} (E_{N,\mu,s,b}) \geq 1 - \epsilon^{\mu}$ such that on $E_{N,\mu,s,b}$,
$$
\left\| \chi(t)\int_0^t e^{i(t-s)\Delta} \mathcal{B}(u) \,ds  \right\|_{X^{s,b}_\epsilon} \lesssim_{N,\mu} \lambda^2  \epsilon^{\frac{1}{2} - \frac{d}{2}-2\mu} \| u \|_{X^{s,b}_\epsilon}^2
$$
\end{proposition}

\begin{proposition}
\label{proptrilinear}
Given $s> \frac{d}{2} - 1$ and $\kappa>0$, there exists $b>\frac{1}{2}$ such that
$$
\left\| \chi(t)\int_0^t e^{i(t-s)\Delta} \mathcal{T}(u) \,ds  \right\|_{X^{s,b}_\epsilon}  \lesssim \lambda^2 \epsilon^{2-d-\kappa} \| u \|_{X^{s,b}_{\epsilon}}^3
$$
\end{proposition}

\subsection{Control of the error, proof of the first part of Theorem \ref{th:main}}  Our aim is to apply the Banach fixed point theorem in $B_{X^{s,b}_\epsilon}(0,\rho)$, where $s> \frac{d}{2}-1$, and $\rho>0$ will be fixed shortly, to the mapping
$$
\Phi: u \mapsto \chi(t) \int_0^t e^{i(t-s) \Delta} [\mathcal{L}(u) + \mathcal{B}(u) + \mathcal{T}(u) + \mathcal{E} ] \,ds 
$$
Note that $\frac{1}{T_{kin}}=\ep^{2-4\gamma}$ with $2-4\gamma>0$. Applying Corollary~\ref{aigleroyal}, the error term can be made $< \epsilon^M$, for any fixed $M$, in $X^{s,b}_{\ep}$ by choosing $N$ sufficiently big. This leads to the choice $\rho = 2 \epsilon^{M}$. Applying Proposition~\ref{prop:lineaire} with $\mu <2\gamma-1$, it appears that the linear term has an operator norm $\ll 1$. Similarly, applying propositions~\eqref{propbilinear} and~\ref{proptrilinear}, one checks easily that the bilinear and cubic term act as contractions on $B(0,\rho)$. Therefore, the Banach fixed point theorem gives a solution $v^{err}$, with norm $\| v^{err}\|_{X^{s,b}_\ep}\lesssim \epsilon^{M}$.
The fact that the solution is in $\mathcal{C}^\infty([0,1] \times \mathbb{R}^d)$ follows by propagation of regularity, since the data is smooth, and the application of the fixed point theorem gives a bound in a supercritical space.

Notice that, in the course of applying Corollary~\ref{aigleroyal}, Proposition~\ref{prop:lineaire} and~\ref{propbilinear}, we had to exclude each time a set $E$ of size $\epsilon^\mu$ in probability. The total set to exclude is thus of size $\leq C \ep^\mu$ and $C$ can be taken to be $1$ from a direct check on the proof.

\subsection{Comparison to the kinetic wave equation, proof of the second part of Theorem \ref{th:main}}
Let $E$ be the exceptional set obtained in the previous subsection. Expanding $u$ gives
$$
u = e^{i\lambda^2 \omega} \left[u^0 + u^1 + u^2 + \sum_{n=3}^N u^n + v^{err}\right],
$$
so that, following the computation in~\cite{BGHS},
\begin{align*}
&\mathbb{E}\left(  \mathbf{1}_E \left[  |\widehat{u}(k)|^2 - \epsilon^d |A(\epsilon k)|^2 \right] \right)=  \underbrace{\mathbb{E} \left( \mathbf{1}_E \left|\widehat{u^1}(k)\right|^2\right) + 2 \mathfrak{Re} \mathbb{E} \left( \mathbf{1}_E \left[\widehat{u^0}(k) \overline{\widehat{u^2}(k)} \right]\right)}_{\displaystyle \mbox{main term}} \\
& \qquad \qquad + \underbrace{\sum_{\substack{i,j \leq N \\ i+j \geq 4} }  \mathbb{E}\left( \mathbf{1}_E \left[\widehat{u^i}(k) \overline{\widehat{u^j}(k)} \right] \right) + 2 \sum_{i \leq N}  \mathfrak{Re} \mathbb{E} \left( \mathbf{1}_E \left[\widehat{u^i}(k) \overline{\widehat{v^{err}}(k)} \right] \right)+\mathbb{E} \left( \mathbf{1}_E|\widehat{v^{err}}(k)|^2 \right)}_{\displaystyle \mbox{higher order}}
\end{align*}
(notice that, due to a cancellation for $i+j = 3$, the first sum of the higher order term only involves $i+j \geq 4$). Using first the Cauchy-Schwarz inequality, and then the bound on the iterates $u^n$ in Corollary~\ref{aigleroyal}, as well as the bound on $v^{err}$ in $X^{s,b}_\epsilon$ (hence in $L^\infty_t L^2_x$) derived in the previous subsection, where we take $M$ sufficiently big,
\begin{align*}
\| \mbox{higher order} \|_{\ell^1_k} & \lesssim \mathbb{E} \left( \mathbf{1}_E \left(  \sum_{i+j \geq 4} \| u^i(t) \|_{L^2} \| u^j(t) \|_{L^2} + \| v^{err}(t) \|_{L^2} \sum_{0 \leq n \leq N} \| u^n(t) \|_{L^2}+ \| v^{err}(t) \|_{L^2}^2 \right) \right) \\
& \lesssim \epsilon^{-\mu} \frac{t}{T_{kin}^2}.
\end{align*}
Forgetting for a moment about $\mathbbm{1}_E$, the main term, following~\cite{BGHS}, can be written
\begin{align*}
\mbox{main term} = \frac{\epsilon^{3d} \lambda^4}{(2\pi)^{2d}} \sum_{k + \ell = m + n} \left| \frac{\sin(\frac{t}{2}\Omega(k,\ell,m,n))}{\Omega(k,\ell,m,n)} \right|^2 & |A(\epsilon k)|^2 |A(\epsilon \ell)|^2 |A(\epsilon m)|^2 |A(\epsilon n)|^2 \\
&  \left[ \frac{1}{|A(\epsilon k)|^2} + \frac{1}{|A(\epsilon \ell)|^2 |} - \frac{1}{|A(\epsilon m)|^2} - \frac{1}{|A(\epsilon n)|^2} \right] .
\end{align*}
where
$$
\Omega(k,\ell,m,n) = |k|^2 + |\ell|^2 - |m|^2 - |n|^2.
$$
Viewing the sum above as a nonlinear Riemann sum, it seems natural to replace the sum by an integral. This is indeed possible: an application of the Hardy-Littlewood-Ramanujan circle method~\cite{Vaughan} gives for $\eta>0$ the existence of $\nu>0$ such that, if $t < \epsilon^{\eta}$,
\begin{align*}
\mbox{main term} = \frac{\lambda^4 \epsilon^d}{(2\pi)^d} \left(  \int_{(\mathbb{R}^d)^3} \right. &  \delta(\epsilon k+\ell-m-n) \left| \frac{\sin(\frac t 2 \epsilon^{-2} \Omega(\epsilon k,\ell,m,n)}{\epsilon^{-2}  \Omega(\epsilon k,\ell,m,n)} \right|^2  |A(\epsilon k)|^2 |A( \ell)|^2 |A( m)|^2 |A( n)|^2 \\
& \qquad \left. \left[ \frac{1}{|A(\epsilon k)|^2} + \frac{1}{|A( \ell)|^2 |} - \frac{1}{|A( m)|^2} - \frac{1}{|A( n)|^2} \right] d\ell\,dm\,dn \right) (1 + O(\epsilon^{\nu})).
\end{align*}

Since $\int \frac{(\sin x)^2}{x^2} \,dx = \pi^2$, there holds, for $f \in \mathcal{C}^\infty_0$ as $\tau \rightarrow \infty$, 
$$
\int \left| \frac{\sin(\tau \Omega)}{\Omega} \right|^2 f(\Omega) \,d\Omega = \pi^2 \tau f(0) + O(1),
$$
and therefore (recalling $T_{kin}=\frac{1}{\lambda^4\ep^2}$)
\begin{align*}
\mbox{main term} &= 2^{-d-1} \pi^{2-d} \epsilon^d \frac{t}{T_{kin}} \int_{(\mathbb{R}^d)^3}  \delta(k+\ell-m-n) \delta(\Omega(k,\ell,m,n))  |A( k)|^2 |A( \ell)|^2 |A( m)|^2 |A( n)|^2 \\
&\quad \quad \quad\left[ \frac{1}{|A( k)|^2} + \frac{1}{|A( \ell)|^2 |} - \frac{1}{|A( m)|^2} - \frac{1}{|A( n)|^2} \right] d\ell\,dm\,dn + O_{\ell^1_k}\left( \frac{\epsilon^\beta t + \epsilon^2}{T_{kin}} \right).
\end{align*}

This concludes the proof of the main theorem, except that we need to put back the characteristic function $\mathbbm{1}_E$. But one can check that the random variable $|\widehat{u^1}_k|^2$ enjoys better integrability properties: this is achieved by raising it to a high power, and taking the expectation. Therefore, the error resulting from $\mathbbm{1}_E$ is at most $O(\epsilon^{c \eta})$.

\section{Encoding correlations by Feynman diagrams}

\label{sectionfeynman}
Our strategy is to relate the computation of the quantities involved in Proposition \ref{pr:LpLpinter} to the computation of integrals with oscillatory phases in high dimension, whose structure can be encoded by Feynman diagrams. We will give all details for the computation of $\| u^n\|_{L^2}$. Other quantities will be also estimated using similar diagrams, and their construction and associated notations will naturally adapt. The notation and graph analysis follow essentially that of \cite{LS2}.

\subsection{Diagrammatic representation}

\label{SubsectionDR}

Recall that $u^n$ is defined recursively by \fref{defun}. To obtain a formula for $u^n$, we use diagrams. We first define so-called interaction diagrams which encode three properties:
\begin{itemize}
\item As $u^n$ is the solution of a forced Schr\"odinger equation, solved via Duhamel formula, the diagram possesses time slices corresponding to a specific choice of ordering of the time variables.
\item As $u^n$ involves a sum over triplets $(i,j,k)$ with $i+j+k=n-1$, the edges and vertices of the diagram correspond to a particular choice of a triplet at each recursive iteration.
\item The sum defining $u^n$ involves complex conjugation: to each edge is associated a sign $\sigma$ which records conjugation.
\end{itemize}

More precisely, following \cite{LS2}, given an integer $n$, we first define the index set $I_n=\{1,2,...,n\}$. A graph with $n$ interaction has the total time $t$ divided into $n+1$ time slices of length $s_i$, $i=0,1,...,n$ whose index label the time ordering: from bottom to top in the graph. Associated with a time slice $i$ there are $1+2(n-i)$ "waves", with three of them merging into a single one for the next time slice. Each wave in each time slice is represented by an \textit{interaction edge} $e_{i,j}$ in the graph, with index set $ \mathcal I_{n}=\{(i,j), \ 0 \leq i \leq n,1\leq j\leq 1+2(n-i)\}$. The \textit{interaction history} is encoded by a vector $\ell = (\ell_1, \dots , \ell_n) \in \mathcal G_n:=I_{2n-1}\times I_{2n-3}\times ...\times I_1$. Edges of the $i$-th time slice are related to edges of the $i+1$-th as follows: an edge with index $(i,k)$ for $k< \ell_{i+1}$ is matched with the edge with index $(i+1,k)$ above it; the three edges with indices $ (i,\ell_{i+1})$, $(i,\ell_{i+1}+1)$ and $(i,\ell_{i+1}+2)$ merge through the vertex $v_{i+1}$ to form the edge with index $(i+1,\ell_i)$; and an edge with index $(i,k)$ for $k>\ell_{i+1}+2$ is matched with the edge with index $(i+1,k-2)$. Complex conjugation is encoded by the \textit{parity} $\sigma_{i,j}\in \{ \pm 1\}$ associated to each edge. Parity is defined recursively from top to bottom: $\sigma_{n,1}=+1$ or $\sigma_{n,1}=-1$ if the graph corresponds to $u^n$ or to $\bar u^n$. Then, parity is kept unchanged from an edge to the one below in absence of merging, and in case of a merging we require that $\sigma_{i,\ell_{i+1}}=-1$, $\sigma_{i,\ell_{i+1}+1}=\sigma_{i+1,\ell_i}$ and $\sigma_{i,\ell_{i+1}+2}=+1$. 

At the initial time slice $s_0$, below each of the edges of index $(0,j)$ for $1\leq j \leq 2n+1$ an \textit{initial vertex} $v_{0,j}$ is placed. At the final time slice $s_n$, a \textit{root vertex} $v_{R}$ is placed. Vertices which are neither initial vertices nor root vertices are called \textit{interaction vertices}.

The graph obtained this way is a tree. Orienting edges from bottom to top, we define the set of \textit{initial vertices below} $v_i$ as $\operatorname{In}(v_i) = \{ j\in (0,2n+1), \, \mbox{there exists an oriented path from}\, v_{0,j} \, \mbox{to} \, v_i\}$. An example is given below:

\vspace*{0.1cm}
\begin{center}
\includegraphics[width=12cm]{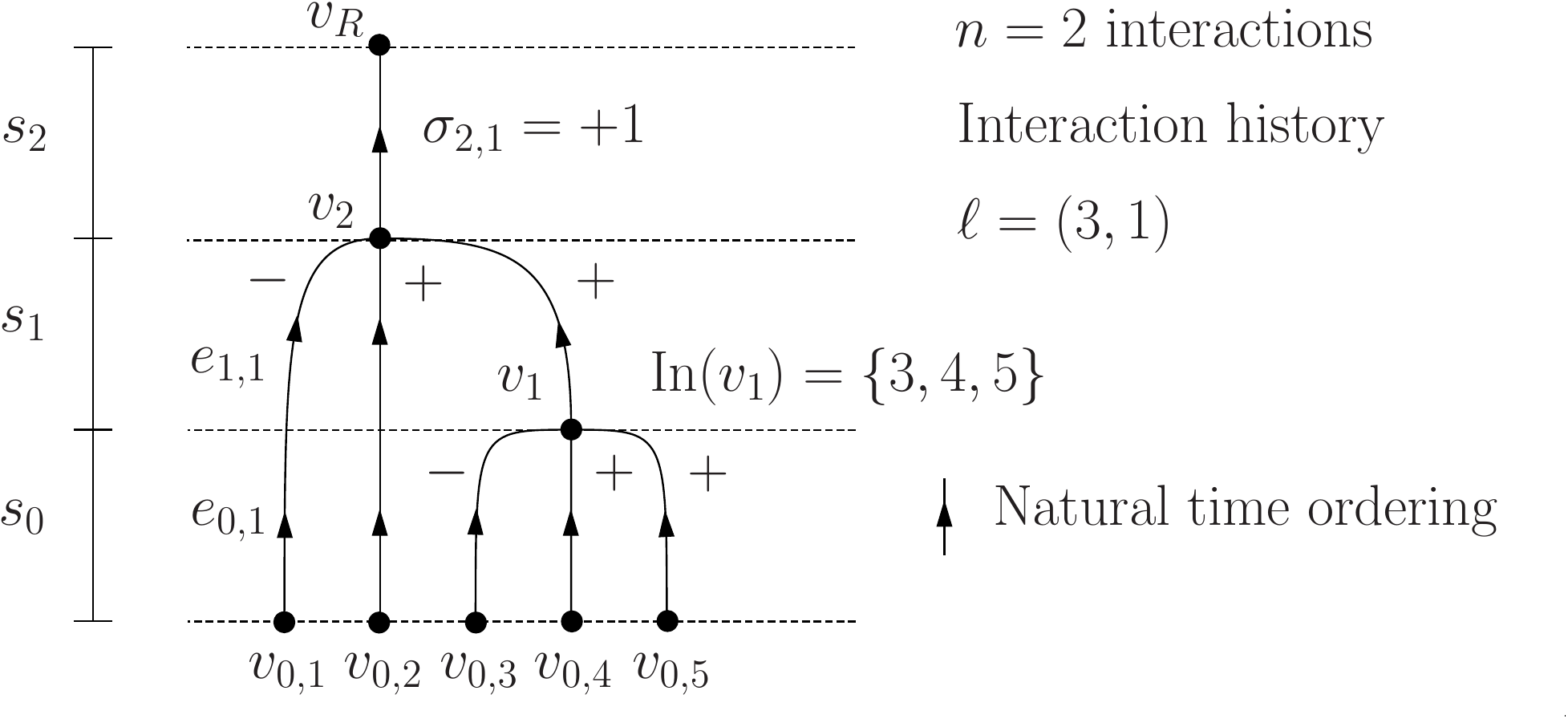}
\end{center}
\vspace*{0.1cm}

We now associate to each edge $e_{i,j}$ in the extended graph a frequency $k_{i,j}$. At each vertex corresponds a $\delta$ function ensuring that the sum of the frequencies associated to the edges below is equal to that of the frequencies for edges above, and that frequencies associated with Wick ordering are removed. These are the Kirchhoff rules for the graph. At the final vertex $v_R$, we impose the Dirac $\delta(k_{n,1}-k_R)$ where $k_R$ denotes the total output frequency. This gives the following formula for $\widehat{u^n} (k)$, where $p$ is the number of vertices $v_i$ whose edge above them carries a $-1$ parity sign
\begin{equation} 
\label{colvert}
\begin{split}
& \widehat{u^n}(t,k_R) = {e^{-it |k_R|^2}  (-1)^p \left(\frac{-i\lambda^2}{(2\pi)^d}\right)^n}\sum_{\ell\in \mathcal G_n} \sum_{\underline{k} \in \mathbb{Z}^{d \# \mathcal I_n}} \\
& \qquad \qquad \qquad  \int_{\mathbb R_+^{n+1}} \prod_{i\in I_{2n+1} }\widehat{u_0}(k_{0,i},\sigma_{0,i})  \prod_{k=1}^{n}e^{-i\Omega_k \sum_{j=0}^{k-1}s_j} \Delta_{\ell}( \underline{k}, k_R)\delta(t-\sum_{i=0}^{n}s_i) d \underline{s}
\end{split}
\end{equation}
where we used the shorthand notations 
\begin{itemize}
\item $\underline{k} =  (k_{i,j})_{(i,j) \in \mathcal{I}_n} \in \mathbb Z^{d \# \mathcal I_n}$
\item $\underline{s} = (s_0,\dots,s_n) \in \mathbb{R}_+^{n+1}$
\item $\widehat{u_0}(k,+1)=\widehat{u_0}(k)$ and $\widehat u_0(k,-1)= \widehat{ \overline{ u_0}}(k) = \overline{\widehat{ u_0}}(-k)$
\item $\Omega_k=|k_{k-1,\ell_k+2}|^2-|k_{k-1,\ell_k}|^2+\sigma_{k,\ell_k}\left( |k_{k-1,\ell_k+1} |^2- |k_{k,\ell_k}|^2\right)$ is the resonance modulus corresponding to the vertex $v_k$
\end{itemize}
and, finally, $ \Delta_{\ell}$ encapsulates the Kirchhoff law and frequency truncation at each vertex:
$$
\Delta_{\ell}(\underline{k},k_R) = \Delta_\ell(\underline{k}) \delta(k_{n,1} - k_R)
$$
with
\begin{align}
\nonumber & \Delta_\ell(\underline{k}) &= \prod_{i=1}^n \Bigl\{ \prod_{j=1}^{\ell_i-1}\delta(k_{i,j}-k_{i-1,j})\left(1-\delta(k_{i-1,\ell_i}+k_{i-1,\ell_i+2})-\delta(k_{i-1,\ell_i+1}+k_{i-1,\ell_i+1-\sigma_{i,\ell_i}})\right)\\
\label{id:defDeltaln} && \delta \left(k_{i,\ell_i}-\sum_{j=0}^2 k_{i-1,\ell_i+j} \right)\prod_{j=\ell_i+1}^{1+2(n-i)} \delta(k_{i,j}-k_{i-1,j+2}) \Bigr\} .
\end{align}
 
\subsection{Pairing two diagrams} 
\label{SubsectionPTD} We now ask how the expectation of $\| u^n \|_{L^2}^2$ can be represented diagrammatically. Consider two summands as in~\eqref{colvert} corresponding to different histories. Each is represented by a tree, which we will call the \textit{left subtree} and the \textit{right subtree}, with histories $\ell$ and $\ell'$ respectively. More generally, we adopt the notation that the frequencies, time variables, etc...  of the right subtree carry primes, and those of the left subtree do not. We will sometimes concatenate the initial frequencies (and parities, etc...) $(k_{0,i})$ and $(k'_{0,i})$ into a single vector as follows
$$
(\widetilde{k}_{0,1}, \dots , \widetilde{k}_{0,4n+2}) = (k_{0,1},\dots,k_{0,2n+1}, k'_{0,1},\dots,k'_{0,2n+1}).
$$
By Wick's formula,
$$
\mathbb E\left( \prod_{i=1}^{4n+2} \widehat {u_0}(\widetilde{k_{0,i}},\sigma_{0,i}) \right) =\ep^{d(2n+1)}\sum_{P\in \mathcal P(\ell,\ell')} 
\prod_{\{i,j\}\in P} |A(\epsilon \widetilde k_{0,i})|^2,
$$
where the \textit{pairing} $P$ is a partitition of $I_{4n+2}$ into pairs satisfying $\sigma_{0,i}\sigma_{0,j}=-1$ if $\{i,j\}\in P$, and $\mathcal P(\ell,\ell')$ is the set of such pairings.

We adopt the following diagrammatic representation of pairings. The top parities of the left and right subtrees are set to $\sigma_{n,1}=1$ and $\sigma_{n,1}' = -1$ respectively. The root vertices $v_R$ and $v_R'$ of the left and right subtrees respectively, are merged into a single vertex $v_R$. To this \textit{root vertex} $v_R$ is attached an edge $e_R$, with frequency $\xi_R = 0$, and we demand that Kirchhoff's rule applies at $v_R$, namely: $k_{n,1} + k'_{n,1} = 0$. If $\{i,j\} \in P$, we create a vertex $v_{\{i,j\}}$; then we add $e_{-1,i}$ and $e_{-1,j}$ between $v_{0,i}$ or $v_{0,j}$ respectively, and $v_{\{i,j\}}$. These new edges $e_{-1,i}$ are called \textit{upper pairing edges}. We attach an edge $e_{\{i,j\}}$  to $v_{\{i,j\}}$, and set the frequency $k_{\{i,j\}} = 0$. These new edges $e_{\{i,j\}}$ are called root pairing edges. Finally, we demand that Kirchhoff's rule applies at $v_{\{i,j\}}$ and $v_{0,i}$. In other words, $k_{0,i} + k_{0,j} = 0$ if $\{i,j\} \in P$, and $k_{-1,i} = k_{0,i}$ for all $i$. This construction corresponds to the following picture:

\begin{center} 
\includegraphics[width=14cm]{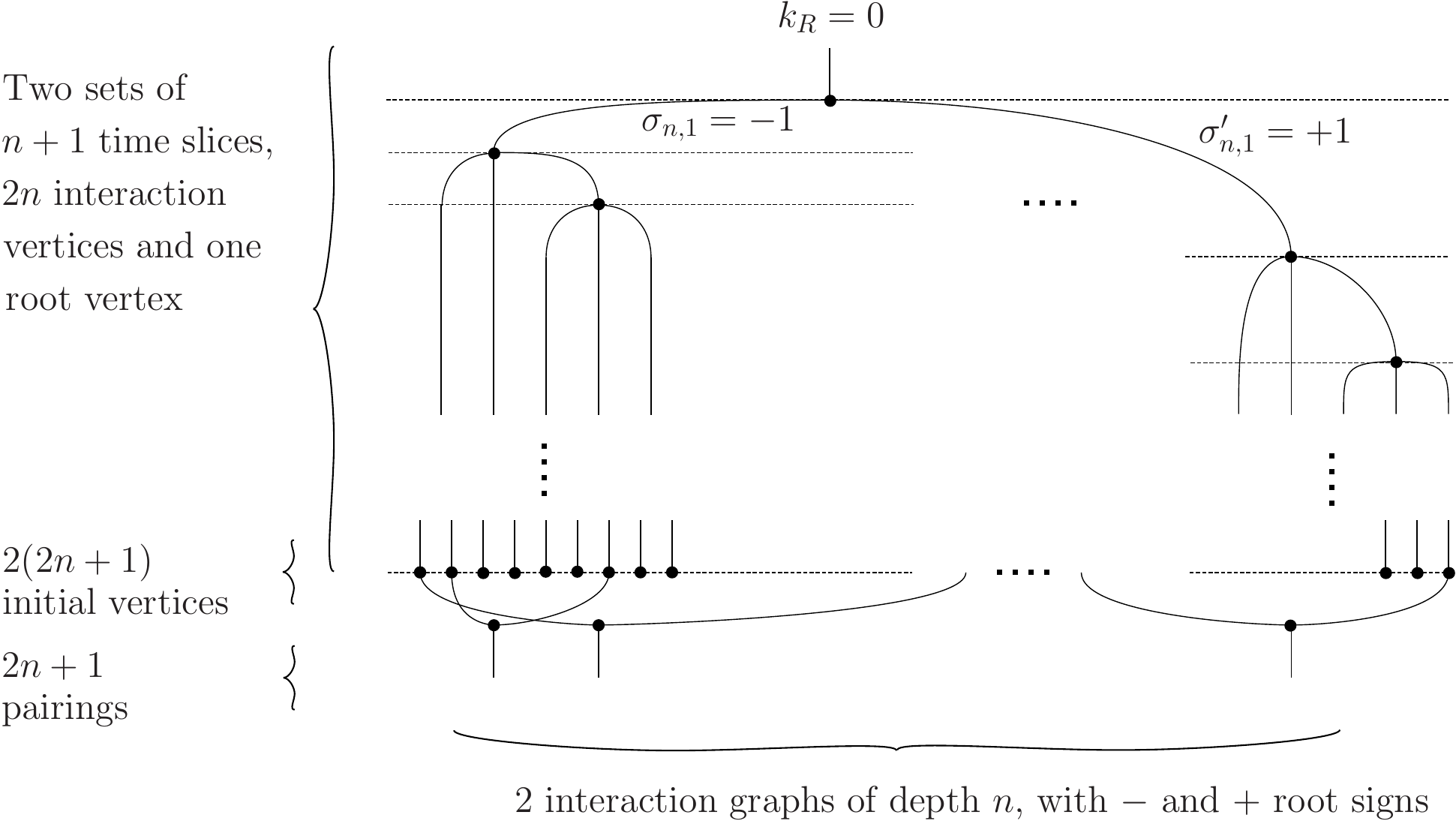}
\end{center}
\vspace*{0.5cm}

In more analytical terms, this corresponds to the formula
\be \label{LpLpexpression}
\mathbb E \| u^n \|_{L^2}^2 =\sum_{G,P} \mathcal F(G,P),
\ee
where $G=G_{\ell,\ell'}$ gives the left and right subtrees, $P \in \mathcal{P}(\ell,\ell')$ is the pairing which is considered, and
\begin{align}
\nonumber \mathcal F(G,P)&= \frac{(-1)^{p+p'} \lambda^{4n}\ep^{d(2n+1)}}{(2\pi)^{2dn}} \sum_{\substack{\underline{k},\underline{k'} \in \mathbb{Z}^{d(2 \# \mathcal{I}_n)}\\ \underline{k_{-1}} \in \mathbb{Z}^{d(4n+2)}}} \int_{\mathbb R_+^{n+1}\times \mathbb R_+^{n+1}} \prod_{k=1}^{n}e^{-i\Omega_k \sum_{j=0}^{k-1}s_j}\prod_{k=1}^{n}e^{i\Omega_k' \sum_{j=0}^{k-1}s_j'} \\
\label{id:formulamathcalFGPraw}& \qquad \qquad  \qquad  \qquad  \qquad \prod_{\{i,j\}\in P} |A(\epsilon \widetilde k_{0,i})|^2   \Delta_{\ell,\ell',P}(\underline{k},\underline{k}',0)\delta \left(t-\sum_{i=0}^{n}s_i \right)\delta \left(t-\sum_{i=0}^{n}s_i' \right) \, d\underline{s} \, d \underline{s}' 
\end{align}
where 
\begin{itemize}
\item $\underline{k} =  (k_{i,j})_{(i,j) \in \mathcal{I}_n}  \in \mathbb{Z}^{d \# \mathcal I_n }$ and similarly for $\underline{k'}$.
\item $\underline{k_{-1}} =  (k_{-1,i})_{i \in \{1,\dots,4n+2\}}$.
\item $\underline{s} = (s_0,\dots,s_n) \in \mathbb{R}_+^{n+1}$ and similarly for $\underline{s'}$,
\item $\Omega_k$ was defined previously below \fref{colvert}, and $\Omega_k'$ is defined in an analogous fashion,
\end{itemize}
and finally
\begin{align*}
\Delta_{\ell,\ell',P}(\underline{k},\underline{k}',\underline{k_{-1}},k_R) = \Delta_{\ell}(\underline{k}) \Delta_{\ell'}(\underline{k'}) \delta(k_R - k_{n,1} - k'_{n,1}) \Delta_P(\underline{k},\underline{k'},\underline{k_{-1}})
\end{align*}
with
$$
\Delta_P(\underline{k},\underline{k'},\underline{k_{-1}}) = \prod_{\{i,j\}\in P}\delta(\widetilde k_{0,i}- k_{-1,i}) \delta(\widetilde k_{0,j}- k_{-1,j}) \delta(k_{-1,i}+k_{-1,j}).
$$

We now come to the question of degenerate pairings: we say that a pairing $P$ has a degeneracy of type $(i,\{j,k\})$, with $j,k \in \{0,1,2\}$, if $\sigma_{i-1,\ell_i+j} \sigma_{i-1,\ell_i+k}=-1$ and if the vertices of $\operatorname{In}(e_{i-1,\ell_i+j})\cup \operatorname{In} (e_{i-1,\ell_i+k})$ are all paired together by $P$. We say a pairing $P$ is degenerate if it has a degeneracy. Degenerate pairings are responsible for the phase modulation $e^{i\lambda^2\omega}$ of our approximate solution. Wick renormalization, that cancels out this phase, is encoded by the frequency truncation $1-\delta(k_{i-1,\ell_i}+k_{i-1,\ell_i+2})-\delta(k_{i-1,\ell_i+1}+k_{i-1,\ell_i+1+\sigma_{i,\ell_i}})$ in the Kirchhoff law \fref{id:defDeltaln} above. The cancellation is not exact, however, but it gives a strong constraint at the degenerate vertex, as appears in the following lemma (which will be used in Lemma \ref{macareux2}).

\begin{lemma} \label{lem:degeneracy}
Given a positive integer $n$, interaction histories $\ell$ and $\ell'$, and a pairing $P$, assume it has a degeneracy of type $(i,\{j,k\})$. Then
$$
|\Delta_{\ell,\ell',P}(\underline{k},\underline{k}',\underline{k_{-1}},k_R)|\lesssim \delta \left(k_{i-1,\ell_i+1}+ k_{i-1,\ell_i+1+\sigma_{i,\ell_i}}\right)\delta \left(k_{i-1,\ell_i}+k_{i-1,\ell_i+2}\right).
$$
\end{lemma}

\begin{proof} Let $\ell$ be such that $\{ j,k,\ell \} = \{ 0,1,2 \}$, and assume without loss of generality that $j$ and $\ell$ have the same parity. Then the factor corresponding to Wick renormalization at the vertex $v_i$ is
$$
(1 - \delta(k_{i-1,\ell_i+j} + k_{i-1,\ell_i+k}) - \delta(k_{i-1,\ell_i+\ell} + k_{i-1,\ell_i+k}))
$$
Since the inital vertices below the edges $e_{i-1,\ell_i+j}$ and $v_{i-1,\ell_i+k}$ are all paired together, it is easy to see that $k_{i-1,\ell_i+j} + k_{i-1,\ell_i+k} = 0$. But then, for the above factor to be nonzero, there has to hold $k_{i-1,\ell_i+\ell} + k_{i-1,\ell_i+k} = 0$, which proves the lemma.
\end{proof}

Below is an example of a degenerate pairing, which is treated by the above Lemma:
\vspace*{0.5cm}
\begin{center} 
\includegraphics[width=10cm]{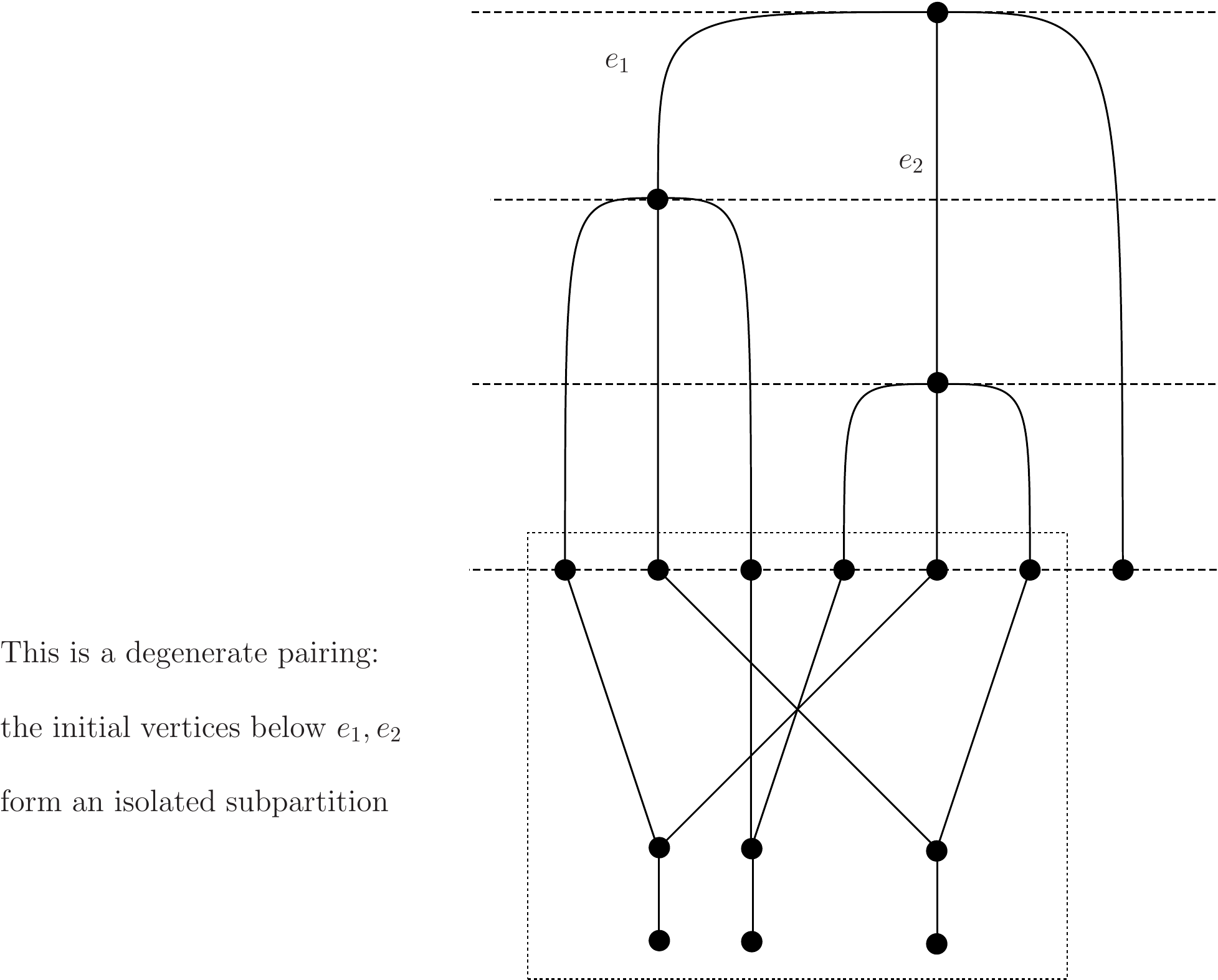}
\end{center}

\subsection{The resolvent identity}

\begin{lemma} \label{lem:resolventemoduli}
Let $m \in \mathbb{N}$, $e_1,\dots,e_m \in \mathbb{R}$, and $\eta>0$. Then
$$
\int_{\mathbb R_+^{m}} \prod_{k=1}^{m} e^{-i s_k e_k} \delta \left( \sum_{k=1}^{m} s_k - t \right) ds_1 \dots ds_m =\frac{e^{\eta t}}{2\pi}\int_{\mathbb R}e^{-i\alpha t}\prod_{k=1}^{m} \frac{i}{\alpha-e_k+i\eta} d\alpha.
$$
\end{lemma}

\begin{proof}
 We use the identity $\delta(x) = \frac{1}{2\pi} \int e^{i \alpha x} \, d\alpha$, together with the fact that $e^{\eta(t - \sum s_k)} = 1$ on the support of the integral on the left-hand side, to write
 \begin{align*}
 \int_{\mathbb{R}_+^m}  \prod_{k=1}^{m} e^{-i s_k e_k} \delta \left( \sum_{k=1}^{m} s_k - t \right) ds_1 \dots ds_m = \int_{\mathbb{R}_+^m} e^{\eta(t - \sum s_k)} \prod_{k=1}^m e^{-i s_k e_k} \frac{1}{2\pi} \int e^{i\alpha(\sum s_k - t)} \,d\alpha \,ds_1 \dots ds_m
 \end{align*}
Switching the order of integration, this is
\begin{align*}
\dots = \frac{e^{\eta t}}{2 \pi} \int e^{-i \alpha t} \prod_{k=1}^m \int_0^\infty e^{is_k(\alpha - e_k + i \eta)} \,ds_k \,d\alpha = \frac{e^{\eta t}}{2\pi}\int_{\mathbb R}e^{-i\alpha t}\prod_{k=1}^{m} \frac{i}{\alpha-e_k+i\eta} d\alpha.
\end{align*}
\end{proof}

Choosing $\eta=1/t$ in the previous lemma, the identity \fref{id:formulamathcalFGPraw} is transformed into:
\begin{equation}
\label{pivert}
\begin{split}
& \mathcal F(G,P)=\frac{(-1)^{p+p'} e^2 \lambda^{4n}\ep^{d(2n+1)}}{(2\pi)^{2dn+2}}  \sum_{\underline{k},\underline{k'},\underline{k_{-1}}}  \int_{\mathbb R \times \mathbb R}  \prod_{\{i,j\}\in P} |A(\epsilon \widetilde k_{0,i})|^2  \\
&  e^{-i(\alpha+\alpha') t}\prod_{k=1}^{n}\frac{1}{\alpha -\sum_{j=k}^{n}\Omega_j+\frac{i}{t}}\prod_{k=1}^{n}\frac{1}{\alpha' -\sum_{j=k}^{n}\Omega_j'+\frac{i}{t}} \frac{1}{\alpha + \frac{i}{t}}\frac{1}{\alpha' + \frac{i}{t}}  \Delta_{\ell,\ell',P}(\underline{k},\underline{k}',0) \, d\alpha\, d\alpha'.
\end{split}
\end{equation}

\subsection{Construction of the spanning tree}

In the formula \fref{pivert}, the variables $k_{i,j}$ are related by Kirchhoff's law at each vertex, as appearing in the first factor in the second line of \fref{id:defDeltaln}. We aim at finding a minimal collection of edges from which, given their associated wave number $k_{i,j}$, one can retrieve the values of all other wave numbers. The first trivial simplification is to identify two edges when they are continued unchanged (no merging) from one slice to another: $(e_{i,k},k_{i,k})$ is identified with $(e_{i+1,k},k_{i+1,k})$ if $k<\ell_{i+1}$, and $(e_{i,k},k_{i,k})$ is identified with $(e_{i+1,k},k_{i+1,k-2})$ if $k>\ell_{i+1}+2$.
The vector space given by the Kirchhoff rules for the $k_{i,j}$ has dimension $2n+1$: indeed, there are $2(2n+1)$ initial frequencies, and $2n+1$ pairings (the condition that the frequencies add up to zero being induced by the pairing). Therefore, we will choose $2n+1$ \textit{free edges}, from whose frequencies all other frequencies can be reconstructed; these free edges will be determined by a spanning tree which will be constructed shortly.
Edges that are not free will be called \textit{integrated}. 

Prior to stating our main result, we define a natural ordering for vertices as follows: the right graph before the left one, and within each graph, from top to bottom; in other words,  $v_1'<\dots < v_n' < v_1 < \dots < v_n$.
So as to be able to compare edges in paired diagrams, we extend this order to interaction edges as follows: if $(G,P)$ is a as in Section~\ref{SubsectionPTD}; for two interaction edges $e$, $e'$, we write $e\leq e'$ if:
\begin{itemize}
\item either $e$ belongs to the interaction graph on the right and $e'$ to that on the left.
\item or $e$ and $e'$ belong to the same interaction graph (left or right), and end at vertices, $v_i$ and $v_j$ or $v_i'$ and $v_j'$, with $i\leq j$.
\end{itemize}
For this order, edges which lie below the same vertex are considered equal. Finally, we write $e>f$ if $e \geq f$ but $f \ngeq e$. This order is the natural time ordering inside each graph (left or right), and corresponds to the fact that we will integrate over the graph according to the natural time ordering of the time slices, first over the graph on the right, then over the graph on the left.

\begin{theorem} \label{th:spanning}

Consider a frequency graph $G$ with pairing $P$. There exists a complete integration of the frequency constraints \fref{id:defDeltaln} in the following sense. There exists a subset of free interaction edges:
$$
\mathcal E^f =\{e^f_1,...,e^f_{2n+1} \},
$$
with associated frequencies $(k^f_i)_{1\leq i \leq 2n+1}$, satisfying the following properties.
\begin{itemize}
\item \emph{Basis property}: On the vectorial subspace of $\mathbb R^{2 \# \mathcal I_n+4n+2}$ determined by the Kirchhoff rules encoded in $\Delta_{\ell,\ell'}$, the family $(k^f_i)_{1\leq i \leq 2n + 1} $ is a basis in the following sense: the map $(k^f_i)_{1 \leq i \leq 2n+1} \to (\underline{k},\underline{k'})$ is a linear bijection.
\item \emph{Time ordering for the spanning}: if $1\leq i\leq j\leq 2n+1$ then $e^f_i\leq e^f_j$ for the natural time ordering of the graph. Furthermore, if $e$ is an integrated interaction edge, then its associated wave number is given by:
$$
k_e=\sum_{1\leq i \leq 2n+1} c_{i,e} k^f_i \quad \mbox{with} \quad c_{i,e}\in \{-1,0,1\},
$$
and $c_{i,e}=0$ whenever $e^f_i < e$ for the natural time ordering of the diagram.
\end{itemize}

\end{theorem}
\begin{proof} This proof follows closely the work of Lukkarinen and Spohn~\cite{LS2}. The idea is to construct iteratively a \textit{spanning tree}, which will determine which edges are free, and which are integrated. The algorithm is as follows: first, add all upper pairing edges $(e_{-1,i})_{1\leq i \leq 2+4n}$ to the spanning graph. Then, consider interaction vertices one by one, according to the natural ordering of the graph: first $v_1', \dots ,v_n'$, then $v_1, \dots ,v_n$. 

\vspace*{0.5cm}
\begin{center}
\includegraphics[width=15cm]{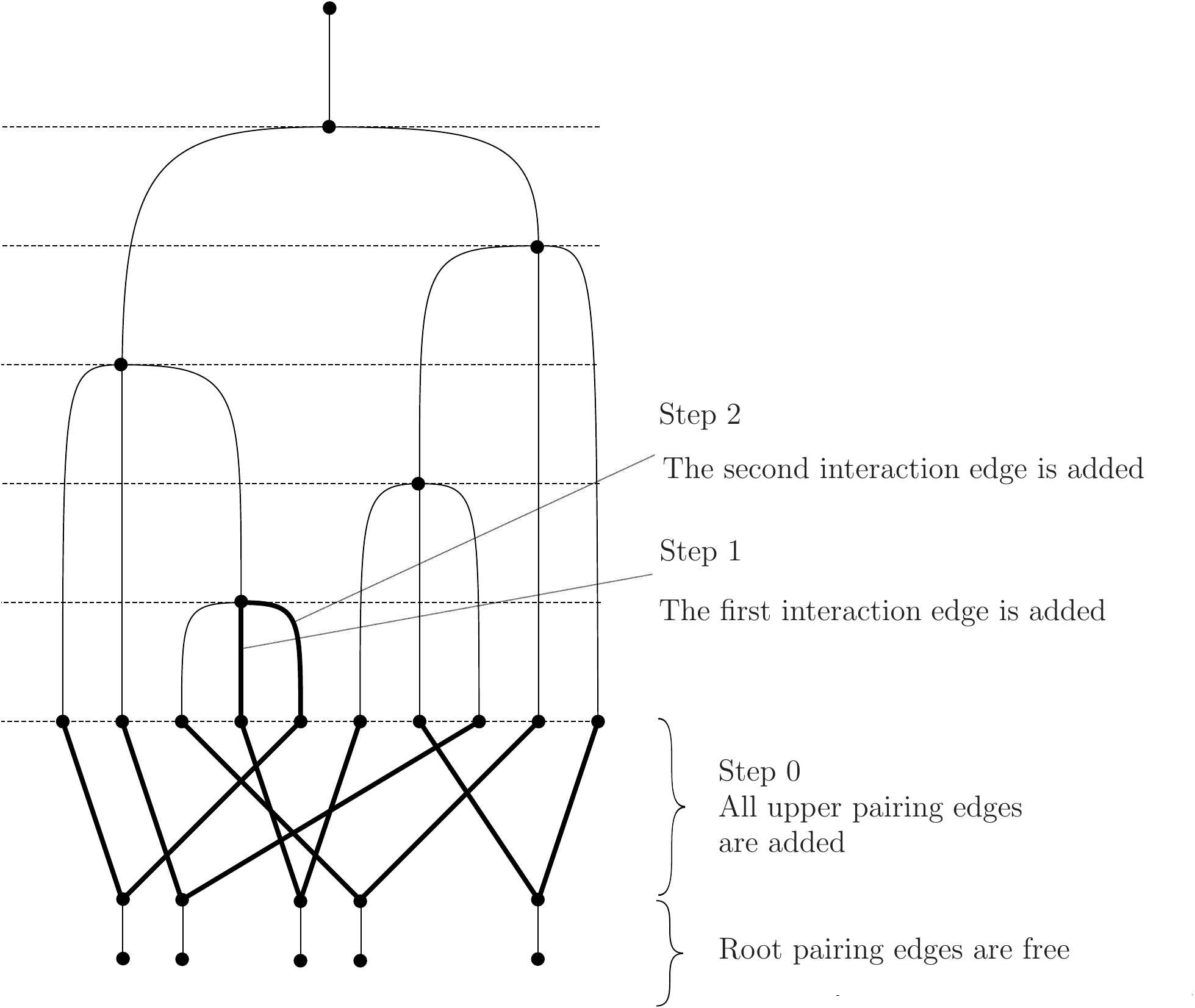}
\end{center}
\vspace*{0.5cm}

In this iterative procedure, suppose we reached $v_k'$ (the procedure is identical for $v_k$). First we add the edge on the right below it. Next, we consider the middle edge:  if adding it does not create a loop with the spanning tree that has been constructed so far, we add it to the spanning tree, and if it does we leave it free. We do the same for the left edge, and then move on to the next vertex.
After all interaction vertices have been treated in this manner, we add the edge between $v_n'$ and $v_R$, and finally the root edge on top of the diagram.

\vspace*{0.5cm}
\begin{center}
\includegraphics[width=15cm]{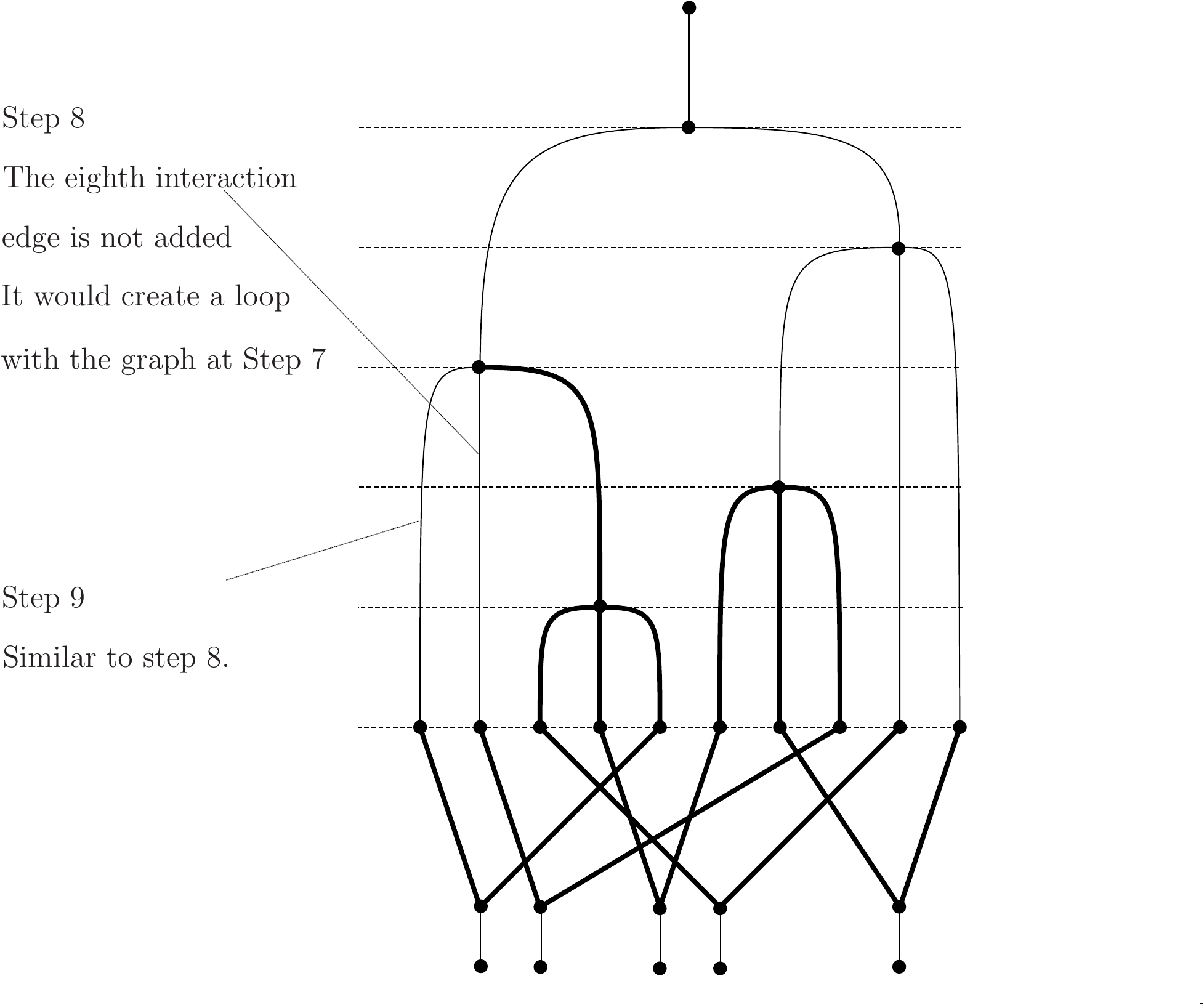}
\end{center}
\vspace*{0.5cm}
After the procedure explained above is completed, we declare all the edges which have been added to the graph \textit{integrated}, and the ones which have not \textit{free}.

The graph obtained at the end of the algorithm is a tree, as by construction it has no loop. In addition, each vertex is connected to the root vertex by a unique path. Indeed, let us prove the following property: in constructing the spanning tree, after the step at which a vertex $v$ is considered, given any initial vertex $v_0\in \operatorname{In}(v)$ below $v$, there exists a path between $v_0$ and $v$ in the spanning tree under construction. To prove the property, assume we reach $v$ in the iterative construction, $e=\{v',v\}$ is an edge below $v$, $v_0\in \operatorname{In}(v')$ an initial vertex below $v'$, and that there is a path from $v_0$ to $v'$. If $e$ is added, there is thus a path from $v_0$ to $v$. If $e$ is not added, this is because this would create a loop, so that there already existed a path between $v'$ and $v$, and thus there is also a path from $v_0$ to $v$. This shows the property by induction. Notice that at least one initial vertex of the left graph is paired with one of the right graph, as there is an odd number $2n+1$ of them. Given any interaction vertex $v$, by a suitable application of the property to $v$, $v_n'$ (which is connected to the root), and to $v_n$ (if $v$ is in the left graph), we thus obtain that there is a path between $v$ and the root.

 We call this tree the \textit{spanning tree}. It carries an orientation, defined as follows: an integrated edge $e=\{v,v' \}$ goes from $v$ to $v'$ if $v'$ belongs to the path from $v$ to the root vertex. This also defines a partial order: we say that $u \preceq w$ if $w$ belongs to the path from $u$ to the root vertex; in particular, $u \preceq u$. We denote by $\mathcal P(u)=\{w, \ w\preceq u\}$ the set of vertices $w$ such that $u$ belongs to the path from $w$ to the origin. Given an integrated edge $e$ going from $v$ to $v'$, the number of edges on the path from $v'$ to the root is the \textit{distance} to the root. An integrated edge $e$ going from $v$ to $v'$ is a leaf of the spanning tree if $\mathcal P(v)=\emptyset$.
\vspace*{0.5cm}
\begin{center} 
\includegraphics[width=15cm]{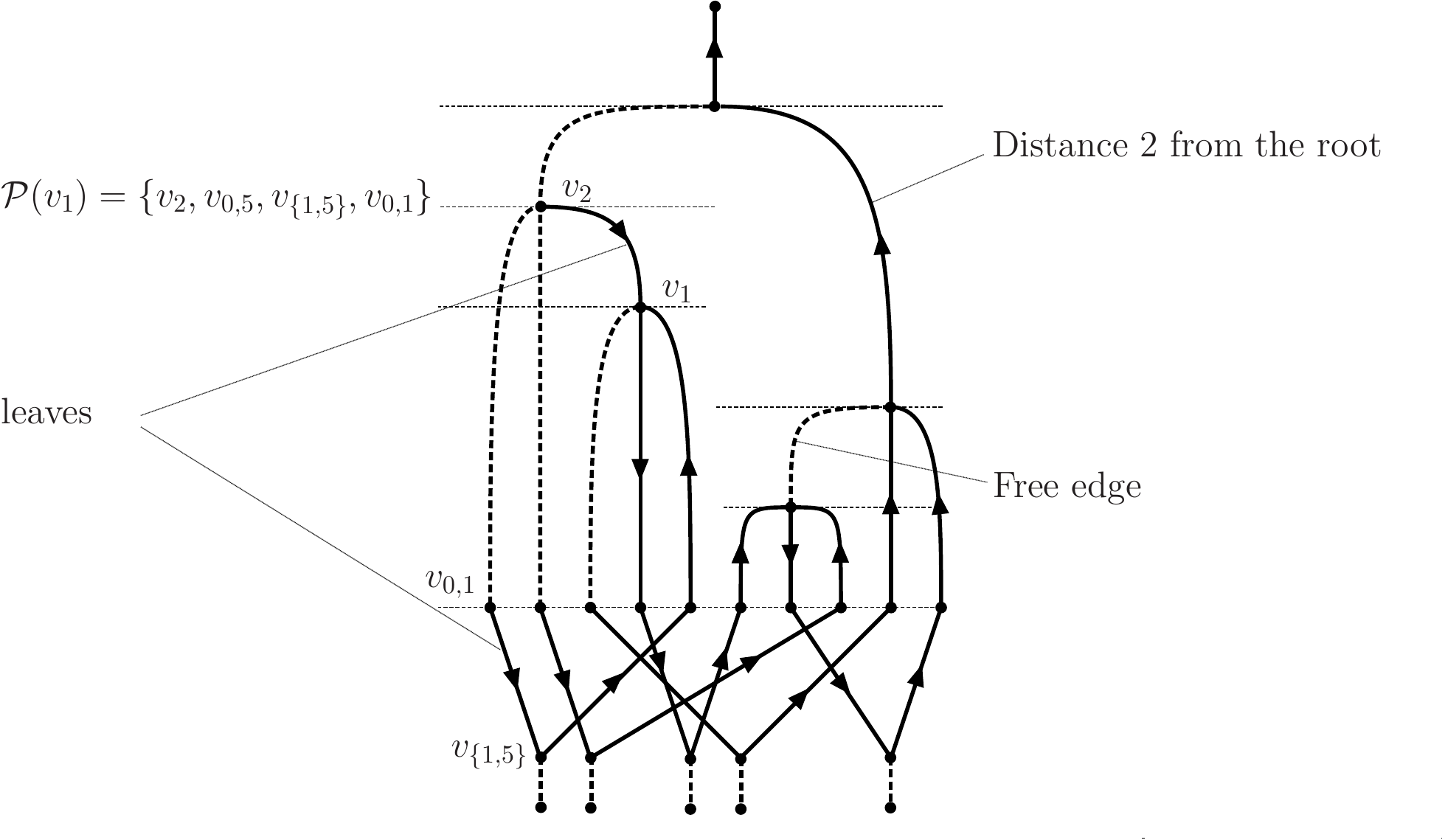}
\end{center}
\vspace*{0.5cm}
Frequencies of the edges belonging to the spanning tree are expressed in function of free frequencies. Given an edge $e$, and $v$ one of its vertices we define the parity of the edge with respect to the vertex as:
$$
\sigma_v(e)=\left\{\begin{array}{l} +1 \quad \mbox{if $e$ is one of the edges above $v$ for the natural ordering}, \\-1\quad \mbox{if $e$ is one of the edges below $v$ for the natural ordering}.\end{array} \right.
$$
Given a vertex $v$, $\mathcal F(v)$ denotes the set of free edges $f$ that have one extremity at $v$. Given $e=\{v,v'\}$ an integrated edge going from $v$ to $v'$, the formula for its associated frequency is then
\be \label{id:integrationmomenta}
k_e= -\sigma_{v}(e) \sum_{w\in \mathcal P(v) ,\;f\in \mathcal F(w)} \sigma_w(f)k_f.
\ee

This formula can be proven by induction on the distance to the root vertex.
\begin{itemize}
\item Elements whose distance is maximal are leaves, so let us assume first that $e$ is a leaf, $e=\{v,v'\}$ with $e$ going from $v$ to $v'$. All other edges having $v$ as an extremity are free, and $\mathcal P(v)=\{v\}$. The Kirchhoff law at $v$ is then
$$
\sigma_v(e)k_e+\sum_{f\in \mathcal F(v)}\sigma_v (f)k_f=0,
$$ 
which implies \fref{id:integrationmomenta} for the leaf $e$
\item Next, assume that the formula holds for edges whose distance is $\geq D$, and consider an edge $e=\{v,v'\}$ going from $v$ to $v'$, at distance $D-1$ from the root. Denoting $\mathcal I(v)$ the integrated edges ending at $v$ (for the orientation of the spanning tree), the Kirchhoff law at $v$ gives
$$
\sigma_v(e)k_e+\sum_{f\in \mathcal F(v)}\sigma_v (f)k_f+\sum_{e'\in \mathcal I(v)}\sigma_v(e')k_{e'}=0,
$$
from which the desired formula for $e$ follows after substituting to $k_{e'}$ its expression in terms of free frequencies.
\end{itemize}

The iterative construction which was described above ensures that Kirchhoff's law holds at every vertex. Indeed, for each vertex $v$, there is a single attached integrated edge $e$ which is oriented away from $v$. When the edge $e$ is considered in the above construction, the corresponding frequency $k_e$ is chosen so that Kirchhoff's law holds at $v$. When the construction ends, Kirchhoff's law holds at all interaction vertices. There remains $v_R$, for which this follows from the fact that all initial frequencies add up to zero.

Overall, this proves that, for any choice of the free frequencies, there is a unique choice of the integrated frequencies so that Kirchoff's law holds throughout the graph. Denoting $\mathcal{E}$ the vector space of all frequencies $\underline{k}, \underline{k'}$ which satisfy Kirchhoff's law at all vertices, this has the following implication. A basis of $\mathcal{E}$ can be written under the form $k_f^i = \delta{ij}$, with the integrated frequencies being determined by~\eqref{id:integrationmomenta}. Another basis of $\mathcal{E}$ can be constructed by selecting the initial frequencies $\widetilde{k}_{0,i}$. Since both bases have the same cardinal, this implies that $n_f=2n+1$.

To finish the proof of Theorem \ref{th:spanning}, we need to show that if $e$ is an integrated edge, then it is only a linear combination of free frequencies appearing after $e$ for the natural time ordering of the diagram. Assume $f=\{ v',v\}$ is a free edge, with $v'$ before $v$ for the natural time ordering. This means that during the construction of the spanning tree, at the step where the vertex $v$ is considered, $f$ is not added as this would create a loop in the spanning tree in construction. At that step, all edges in the spanning tree are before $f$ for the natural time ordering. Hence there exists a path $\widetilde p$ from $v$ to $v'$, and all its edges are before $f$ for the natural time ordering. Also, there exist unique paths $p$ and $p'$ going from $v$ to the root and from $v'$ to the root respectively. These paths intersect at a vertex $v_0$. By their uniqueness, $v_0$ has to belong to $\widetilde p$. Consider now the formula above: $k_f$ can only appear in the integrated frequencies on the paths from $v$ and $v'$ to the root. Moreover, after the vertex $v_0$, the two contributions from $v$ and $v'$ in this formula cancel. Hence $k_f$ can only appear in the integrated frequencies on the path from $v$ to $v_0$, and in the integrated frequencies on the path from $v'$ to $v_0$. These belong to $\widetilde p$ hence are indeed before $k_f$ for the natural time ordering. This also shows that $c_{i,e} \in \{ -1,0,1 \}$.
\end{proof}

\subsection{Degrees of vertices} \label{subsec:kirchhoff}
For any interaction vertex, we consider the three edges below it and define its \textit{degree} to be the number of free edges amongst them.

Any interaction vertex is necessarily of degree $0$, $1$ or $2$. Indeed, in the determination of the free edges in Theorem \ref{th:spanning}, by construction, the edge on the right below each vertex always belongs to the minimal spanning tree and hence is not free. The following Lemma describes how the frequencies associated to free edges below one interaction vertex appear in the decomposition of the frequencies associated to the other integrated edges below this vertex.

\begin{lemma} \label{lem:Lpinter}
The following holds true.
\begin{itemize}
\item \emph{Degree $1$ vertex:} Assume $v$ is a degree one vertex. Then the edges below it are always of the form $\{f,e,e'\}$ (unordered list), where $f$ is the free edge, and where the formulas giving $k_e$ and $k_{e'}$ in terms of the free edges are $k_e=-k_f+G$ and $k_{e'}=G'$, where $G$ and $G'$ are independent of $k_f$.
\item \emph{Degree $2$ vertex:} Assume $v$ is a degree two vertex. Then the edges below it are always of the form $\{f,f',e\}$ (unordered list), where $f$ and $f'$ are the free edges, and where the formula giving $k_e$ in terms of the free edges is $k_e=-k_f-k_{f'}+G$, where $G$ is independent of $k_f$ and $k_{f'}$.
\end{itemize}
\end{lemma}

\begin{proof}
For a degree one vertex, recall that $k_f + k_e + k_{e'}$ can only depend on later frequencies (for the natural order); furthermore, both $k_e$ and $k_{e'}$ can be written as one of $G+k_f$, $G-k_f$, or $G$, where $G$ only depends only on later frequencies. The desired conclusion follows immediately.

For a degree two vertex, it suffices to observe that $k_f + k_{f'} + k_{e}$ only depends on  later frequencies.
\end{proof}

\begin{lemma} \label{lem:comptage}
Denoting by $n_i$ the number of interaction vertices of degree $i$ for $i=0,1,2$, one has the following relations:
\be \lab{sumni}
n_0+n_1+n_2=2n
\ee
\be \lab{sumni4}
n_1+2n_2=2n
\ee
\end{lemma}

\begin{proof}

The first relation comes from the fact that there are $2n$ interaction vertices, $n$ in each subgraph below the root vertex. For the second, recall that there are $2n+1$ free interaction edges (excluding the pairing edges). Moreover, there is always one edge below the root vertex that is free, and one that is not. Hence, there are $2n+1-1=2n$ free variables below the interaction edges, which on the other hand equals $n_1+2n_2$ by definition of the degree. This proves \fref{sumni4}.

\end{proof}

\begin{lemma} \label{lem:degreezerofirst}
In the construction of the spanning tree in the proof of Theorem \ref{th:spanning}, the first interaction vertex $v_1$ that is considered, in the graph on the right, is always of degree zero or one. Moreover, if it is of degree one, then the graph has a degeneracy at $v_1$ of type $(1,\{j,k \})$ for some $\{j,k\}\subset \{0,1,2\}$ and there holds:
\be \label{degeneracyleftvertex}
|\Delta_{\ell,\ell',P}(\underline{k},\underline{k}',\underline{k_{-1}},k_R)|\lesssim \delta \left(k_{0,\ell_1+1}+ k_{0,\ell_1+1+\sigma_{1,\ell_1}}\right)\delta \left(k_{0,\ell_1}+k_{0,\ell_1+2}\right).
\ee
\end{lemma}

\begin{proof}
Let us for simplicity call $v$ the first interaction vertex in the right graph, $e_1=e_{0,\ell_1}$, $e_2=e_{0,\ell_1+1}$ and $e_3=e_{0,\ell_1+2}$ the three edges below it (left, center, right), with initial vertices $v_{0,\ell_1}$, $v_{0,\ell_1+1}$ and $v_{0,\ell_1+2}$ respectively. We recall that in the construction of the spanning tree in Theorem \ref{th:spanning}, first all upper pairing edges are added, giving a graph $\mathcal G^0$. Then, the edge $e_3$ is added to $\mathcal G^0$, and we denote this graph by $\mathcal G^1$. 

First, if $v_{0,\ell_1+1}$ is not paired with $v_{0,\ell_1+2}$, then $e_2$ is an integrated edge that is added to $\mathcal G^1$. Next, if $v_{0,\ell_1}$ is not paired with $v_{0,\ell_1+1}$ or $v_{0,\ell_1+2}$, then $e_1$ is also added to the spanning tree and $v$ is of degree zero. If $v_{0,\ell_1}$ is paired with $v_{0,\ell_1+j}$ for $j\in \{1,2\}$, then $e_1$ is not added, $v$ is of degree one and has a degeneracy of type $(1,\{0,j\})$.

Second, if $v_{0,\ell_1+1}$ is paired with $v_{0,\ell_1+2}$, then $e_2$ is kept free and is not added to $\mathcal G^1$. Next, $v_{0,\ell_1}$ cannot be paired with either $v_{0,\ell_1+1}$ or $v_{0,\ell_1+2}$, so $e_1$ is added to the spanning tree. $v$ is of degree one and has a degeneracy of type $(1,\{1,2\})$.

In either cases, \fref{degeneracyleftvertex} holds true from Lemma \ref{lem:degeneracy}.
\end{proof}

\label{sectionchainladder}

\section{Bounds on the approximate solution}

\label{sectionapproximate}

\subsection{Main result}

\begin{proposition} \label{pr:LpLpinter}
For any $p\in \mathbb N$, under the hypothesis~\eqref{epsilont},
\be \label{bd:fourierwnLp}
\mathbb E \| u^n(t) \|_{L^{2p}}^{2p} \lesssim_{n,p,\kappa} 
\epsilon^{-\kappa} \left\{ \begin{array}{l l l} (\lambda^2 t)^{2pn} & \mbox{for } t\leq \epsilon,\\ \displaystyle \left( \frac{1}{T_{kin}}\right)^{pn} t(t^{\frac 12}\epsilon^{-1})^{p-1} & \mbox{for }  \epsilon \leq t \lesssim 1. \end{array} \right.
\ee
Therefore,
\be \label{bd:fourierwnLpLp}
 \mathbb E \| u^n \|_{L^{2p}_TL^{2p}}^{2p} \lesssim_{n,p,\kappa} 
  \epsilon^{-\kappa} \left\{ \begin{array}{l l l} T(\lambda^2 T)^{2pn} & \mbox{for } T \leq \epsilon ,\\  \displaystyle \left( \frac{1}{T_{kin}}\right)^{pn}(\ep+T^2)  (T^{\frac 12}\epsilon^{-1})^{p-1}& \mbox{for }  \epsilon \leq T \lesssim 1.\end{array} \right.
\ee
Furthermore,
 if $\kappa>0$, $s\in \mathbb{N}$, there exists $b>\frac{1}{2}$ such that
\be \label{bd:fourierwnXsb}
\mathbb E \left\| \chi \left(\frac t T\right)u^n \right\|_{X^{s,b}_\epsilon}^2 \lesssim_{n,s,b,\kappa} \epsilon^{-\kappa}  \displaystyle \left( \frac{1}{T_{kin}} \right)^{n} \quad \quad \mbox{for }  \epsilon \lesssim T \lesssim 1.
\ee
Finally,
$$
\mathbb{E} \left\| \chi(t)\int_0^t e^{i(t-s)\Delta}\chi(s) E^N \,ds \right\|_{X^{s,b}_\epsilon} \lesssim \epsilon^{- \frac{1}{4} - \frac{d}{4}-\kappa}  \left( \frac{1}{T_{kin}} \right)^{\frac N4}.
$$
\end{proposition}

\subsection{Proof of the $L^2$ bound}

\label{subsecproofl2}

 \underline{The trivial bound on the time integral.} For $t\leq \epsilon$, we use the identity \fref{id:formulamathcalFGPraw} and use the rough estimates $|e^{-i\Omega_k \sum_{j=0}^{k-1}s_j}|\leq 1$:
\begin{align*}
& |\mathcal F(G,P)| \\
& \quad \lesssim \lambda^{4n}\ep^{d(2n+1)} \sum_{\underline{k},\underline{k}'}  \int_{\mathbb R_+^{n+1}\times \mathbb R_+^{n+1}} \mathbbm{1}(|\underline{k}|,|\underline{k}'| \lesssim \epsilon^{-1}) \Delta_{\ell,\ell',P}(\underline{k},\underline{k}',0)\delta\left(t-\sum_{i=0}^{n}s_i\right)\delta\left(t-\sum_{i=0}^{n}s_i'\right) \, d\underline{s} \, d \underline{s}' 
\end{align*}
After resolution of the momenta constraints, Theorem \ref{th:spanning}, we compute the above integral by integrating over the free variables $(k^f_i)_{1\leq i \leq 2n+1}$ given by this Theorem, giving a factor $\epsilon^{-d(2n+1)}$.Then we integrate over the temporal variables, giving a $t^{2n}$ factor, so that:
$$
|\mathcal F(G,S,P)|\lesssim \lambda^{4n}\ep^{d(2n+1)} \epsilon^{-d(2n+1)}t^{2n}=\lambda^{4n}t^{2n}.
$$

\bigskip

\noindent \underline{Splitting the resolvent integral}. We use the identity \fref{pivert} for the oscillatory factors, instead of \fref{id:formulamathcalFGPraw} as in Step 0. We resolve the momenta constraints using Theorem \ref{th:spanning}, so that the sum over all frequencies reduces to the sum over free frequencies $(k_f^i)_{1\leq i \leq 2n+1}$. This gives
\begin{align}
\nonumber \left| \mathcal F(G,P)\right|&  \lesssim \lambda^{4n}\ep^{d(2n+1)} \sum_{\underline{k_f}} \mathbbm{1}(|\underline{k}|,|\underline{k}'| \lesssim \epsilon^{-1})  \Delta_{\ell,\ell',P} (\underline{k},\underline{k'},0) \\
\nonumber & \qquad \int_{\mathbb R^2} \prod_{k=1}^{n}\frac{1}{\left|\alpha -\sum_{j=k}^{n}\Omega_j+\frac{i}{t}\right|}\prod_{k=1}^{n}\frac{1}{\left|\alpha' -\sum_{j=k}^{n}\Omega_j'+\frac{i}{t}\right|} \frac{d\alpha}{|\alpha+\frac i t|}\frac{d\alpha'}{|\alpha'+\frac i t|}\\
\label{id:integralsplit} & = \underbrace{\lambda^{4n}\ep^{d(2n+1)} \int_{|(\alpha,\alpha')|\leq \epsilon^{-K}}[...]}_{\displaystyle \mathcal F_1(G,P)} + \underbrace{\lambda^{4n}\ep^{d(2n+1)} \int_{|(\alpha,\alpha')|> \ep^{-K}}[...]}_{\displaystyle \mathcal F_2(G,P)},
\end{align}
where $K \gg 1$ remains to be fixed.

\bigskip

\noindent \underline{The bound for $(\alpha,\alpha')$ small: $\mathcal F_1(G,P)$.} In order to bound this term, we integrate over the free variables in the following order: we consider each interaction vertex iteratively according to the natural time ordering of the graph, and then integrate over the free variables below it; see below for the details of this operation. This results in an integration over the variables $(k^i_f)_{1\leq i \leq 2n+1}$. Finally, we integrate over the $\alpha$ and $\alpha'$ variables which contribute a subpolynomial factor.

When integrating at each edge we obtain the following bounds. Below we treat the case for which the edge belongs to the left graph for simplicity.
\begin{itemize}
\item If $v_k$ is of degree $0$, then we use the rough bound $\displaystyle
\left| \frac{1}{\alpha -\sum_{j=k}^{n}\Omega_j+\frac{i}{t}} \right|\leq t.$
\item If $v_k$ is of degree $1$, we write:
\begin{align*}
\alpha -\sum_{j=k}^{n}\Omega_j & = -\Omega_k+\alpha-\sum_{j=k+1}^{n}\Omega_j\\
&=|k_{k-1,\ell_k}|^2-|k_{k-1,\ell_k+2}|^2-\sigma_{k,\ell_k}\left( |k_{k-1,\ell_k+1} |^2- |k_{k,\ell_k}|^2\right)+\alpha-\sum_{j=k+1}^{n}\Omega_j.
\end{align*}
There exists one free variable among $k_{k-1,\ell_k}$ and $k_{k-1,\ell_k+1}$, that we denote by $k^f_{i_k}$ (i.e. this is the $i_k$-th free variable given by Theorem \ref{th:spanning}). From the time ordering property of Theorem \ref{th:spanning}, one notices that the quantities $|k_{k,\ell_k}|^2$ and $\alpha-\sum_{j=k+1}^{n}\Omega_j$ are independent of $k^f_{i_k}$. By Lemma \ref{lem:Lpinter} and Lemma~\ref{macareux2}, since $\ep \lesssim t$,
$$
\left| \sum_{k_{i_k}^f}\mathbf{1}(|\underline{k}| \lesssim \epsilon^{-1})\Delta_{\ell,\ell',P} (\underline{k},\underline{k'},0)   \frac{1}{\alpha -\sum_{j=k}^{n}\Omega_j+\frac{i}{t}} \right|\lesssim \ep^{1-d-\kappa} .
$$

\item If $v_k$ is of degree $2$ we perform a similar analysis. We integrate over its associated free variables (which we denote here $k^f_{i_k}$ and $k^f_{i_k'}$) and estimate by Lemma \ref{lemmavertextwo}:
$$
\left| \sum_{k^f_{i_k},k^f_{i_k'}} \mathbf{1}(|\underline{k}| \lesssim \epsilon^{-1}) \Delta_{\ell,\ell',P} (\underline{k},\underline{k'},0) \frac{1}{\alpha -\sum_{j=k}^{n}\Omega_j+\frac{i}{t}} \right|\lesssim \ep^{2-2d-\kappa}.
$$
\end{itemize}
Recall the notation $n_i$ from Lemma \ref{lem:comptage}. Once the above integration procedure is completed, we integrate over the last free variable ending at the root vertex using solely that it is restricted to a ball of radius $\lesssim \epsilon^{-1}$, yielding a factor $\epsilon^{-d}$. We thus obtain the following bound for the first part of the integral in \fref{id:integralsplit}:
\begin{align*}
\left| \mathcal F_1(G,P)\right| & \lesssim_{n,\kappa} \lambda^{4n} \epsilon^{-\kappa} \ep^{d(2n+1)} t^{n_0}\left(\ep^{1-d}\right)^{n_1}\left(\ep^{2-2d}\right)^{n_2}\epsilon^{-d} \int_{|\alpha,\alpha'|\leq \ep^{-K}}\frac{d\alpha}{|\alpha+\frac i t|}\frac{d\alpha'}{|\alpha'+\frac it|}  \\
 & \lesssim \epsilon^{-\kappa} \lambda^{4n}\ep^{2nd}t^{n_0}\ep^{(1-d)n_1} \ep^{(2-2d)n_2}.
\end{align*}
Using that $n_1 + 2n_2 = 2n$ and $n_0+\frac{n_1}{2} = n$, as follows from \eqref{sumni} and \eqref{sumni4}, this is
\begin{align}
\dots & = \epsilon^{-\kappa} \lambda^{4n}\ep^{2n}t^{n -\frac{n_1}{2}} \lesssim  \epsilon^{-\kappa} \label{bd:mathcalF1L2} \lambda^{4n}\epsilon^{2n}t,
\end{align}
where, in order to obtain the last inequality, we used that $t \leq 1$, and that the worst case is $n_1=2n-2$. If $n_1=2n$ then the first interaction vertex is degenerate, which produces an additional $\ep^{d-1}\leq t$ factor thanks to \fref{bd:degree1estimate} and the bound is actually better. The case $n_1=2n-1$ is ruled out by \fref{sumni} and \fref{sumni4}.

\bigskip

\noindent \underline{The bound for $(\alpha,\alpha')$ large: $\mathcal F_2(G,P)$.} Let us for simplicity only consider the case where $|\alpha'| > \ep^{-K}$, $|\alpha| <\ep^{-K}$ (by symmetry, the only other case to consider is $|\alpha|, |\alpha'| >\ep^{-K} $, which is simpler). Noticing that $|\Omega_k|\lesssim \ep^{-2}$ for the interactions we consider, there holds for $|\alpha'|>\ep^{-K}$, if $K>2$,
$$
\frac{1}{|\alpha' -\sum_{j=k}^{2n}\Omega_j'+\frac{i}{t}|}\lesssim \frac{1}{|\alpha'|}.
$$
Therefore, the integral over $\alpha'$ can be bounded by $\int_{\ep^{-K}} \frac{d\alpha}{|\alpha|^{n+1}}$, which, by taking $K$ sufficiently big, can gain an arbitrarily large power of $\epsilon$; this makes the estimate trivial.

\subsection{Proof of the $L^p$ bound}

We use the relation:
$$
\| u^n \|_{L^{2p}}^{2p}= \mathcal F \left((u^n)^p (\overline{ u^n})^p\right)(0)
$$
and follow the same proof as for $p=1$. What changes is that the identity corresponding to \fref{LpLpexpression} is now:
$$
\mathbb E \| w^n \|_{L^{2p}}^{2p} =\sum_{G,P} \mathcal F(G,P)
$$
with the following analogue of \fref{pivert}, where the notation $'$ to distinguish between the left and right graph is now replaced by the superscript $m\in \{1,...,2p\}$ to distinguish between the $2p$ different graphs,
\begin{align*}
\mathcal F(G,P)= & \frac{(-1)^{\sum_{m=1}^{2p}p^m} e^{2p} \lambda^{4pn}\ep^{dp(2n+1)}}{(2\pi)^{2p(nd+1)}} \sum_{\underline{k}}\prod_{\{i,j\}\in P} |A(\epsilon k_{0,i})|^2  \Delta_{\underline{\ell},P}(\underline{k},0)  \\
&  \qquad  \int_{\mathbb R^{2p}} e^{-i\sum_{m=1}^{2p}\alpha^m t}\prod_{m=1}^{2p} \prod_{k=1}^{n}\frac{1}{\alpha^m -\sum_{j=k}^{2n}\Omega_j^m+\frac{i}{t}} \prod_{m=1}^{2p} e^{-it\sigma^m_{n,1}|k^m_{n,1}|^2} \prod_{m=1}^{2p} \frac{1}{\alpha^{m} + \frac{i}{t}}  \,  d\underline{\alpha},
\end{align*}
with an an obvious definition for $\Delta_{\underline{\ell},P}(\underline{k},k_R) $ etc... This can be represented by the following graph:
\begin{center}
\includegraphics[width=15cm]{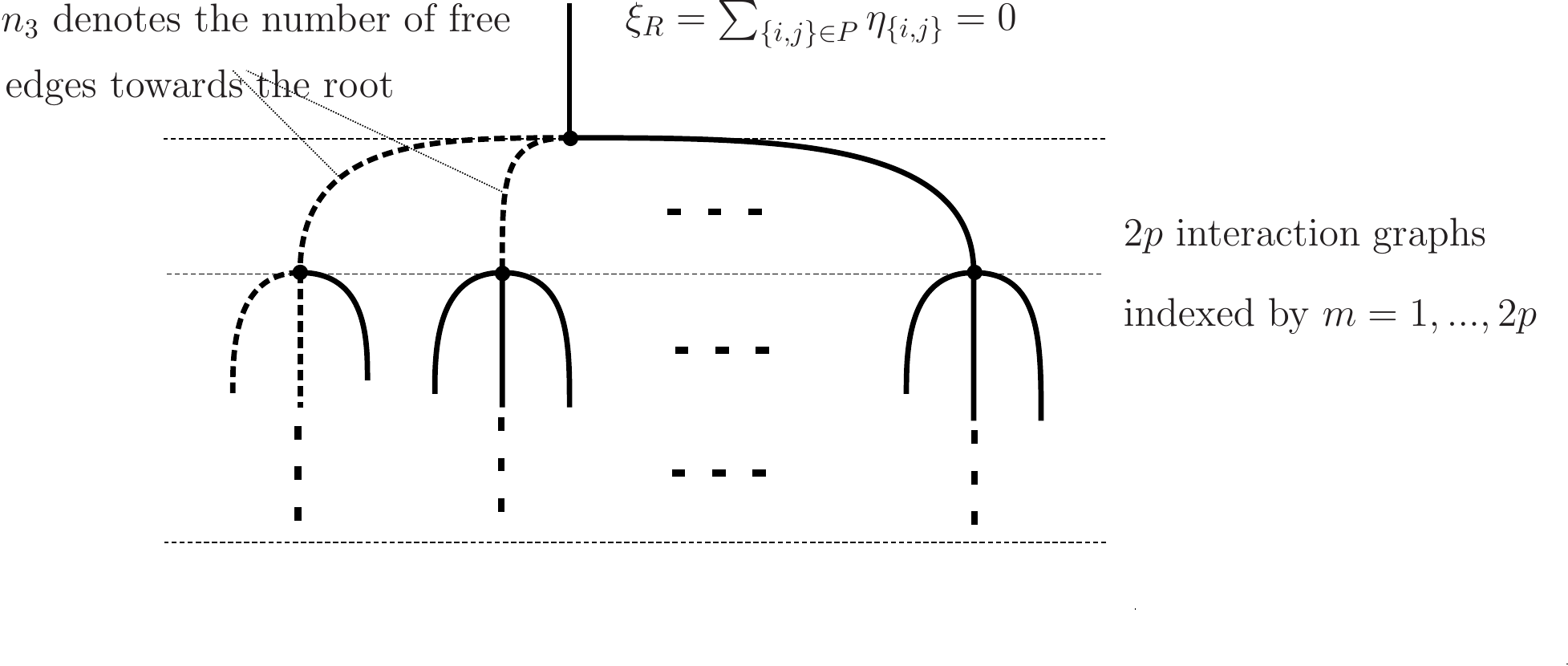}
\end{center}

We introduce the following new notation: $n_3\in \{1,...,2p-1\}$ denotes the number of free edges joining the last interaction vertices to the root. Lemma \ref{lem:comptage} naturally adapts and we get the following relations:
\be \lab{sumLpLp}
n_0+n_1+n_2=2pn, \quad \mbox{and} \quad n_1+2n_2+n_3=p(2n+1), \quad 1\leq n_3\leq 2p-1.
\ee
The first equality corresponds to the classification of $2pn$ interaction according to their degree, the second one reflects the total number of free variables, and the last one expresses the fact that at least one of the last edges reaching the root is free, and that at least one is not. We perform the very same strategy as in the proof of the bound for $\mathcal F(G,P)$ when $p=1$. There is one exception: after integrating all free variables below interaction edges, we integrate over the remaining last $n_3$ free variables and estimate using the trivial support estimate $|k^m_{n,1}|\lesssim \e^{-1}$. This produces:
\begin{align*}
 \left| \mathcal F(G,P)\right|  & \lesssim \epsilon^{-\kappa} \lambda^{4pn}\ep^{dp(2n+1)}t^{n_0}\ep^{(1-d)n_1}\ep^{(2-2d) n_2}\epsilon^{-dn_3}.
\end{align*}
Since $2n_2 = p(2n+1) - n_1 -n_3$ and $n_0+\frac{n_1}{2}=pn+\frac{n_3-p}{2}$, this is
\begin{align*}
\dots &=\epsilon^{-\kappa} \lambda^{4pn}\ep^{2pn}t^{pn} t^{-\frac{n_1}{2}}(t^{\frac 12}\ep^{-1})^{n_3-p}\\
&\lesssim \left\{\begin{array}{l l} \epsilon^{-\kappa} \lambda^{4pn}\epsilon^{2pn}t(t^{\frac 12}\ep^{-1})^{p-1} \quad \mbox{if }n_1\leq 2pn-2, \\
 \epsilon^{-\kappa} \lambda^{4pn}\epsilon^{2pn}t^{\frac 12}(t^{\frac 12}\ep^{-1})^{2-p}  \quad \mbox{if }n_1= 2pn-1, \\
 \ep^{d-1}\lambda^{4pn}\ep^{2pn}(t^{\frac 12}\ep^{-1})^{1-p} \quad \mbox{if } n_1=2pn,\\
 \end{array} \right.\\
 &\lesssim  \epsilon^{-\kappa} \lambda^{4pn}\epsilon^{2pn}t(t^{\frac 12}\ep^{-1})^{p-1}
\end{align*}
Above, we used that, for $n_1\leq 2pn-2$, if $t\leq 1$ then $t^{-\frac{n_1}{2}}\leq t^{1-pn}$, and then that the worst possible contribution of $(t^{\frac 12}\ep^{-1})^{n_3-p}$ occurs for $n_3=2p-1$ as $t\geq \epsilon$. The case $n_1=2pn-1$ forces $n_3=2$ by \fref{sumLpLp} and we then use $\ep\lesssim t$ and $p\geq 2$. If $n_1=2pn$ then the first interaction vertex is degenerate, which produces an additional $\ep^{d-1}\leq t$ factor thanks to \fref{bd:degree1estimate}.

\subsection{Proof of the $X^{s,b}$ bound} \label{proofxsb}
The proof follows the same strategy as that of the $L^2$ norm, we will simply highlight what are the necessary modifications. From \fref{colvert} and the resolvent identity Lemma \ref{lem:resolventemoduli} with $\eta=1/T$, one obtains the following expression for the spacetime Fourier transform of $u^n$ (where the $e^{\frac tT}$ factor has been absorbed in the cut-off $\chi(t/T)$ in the right hand side to simplify notations):
\begin{align*}
& \mathcal F\left(\chi\left(\frac t T\right)u^n \right)(\tau,k_R) =  \frac{i^n (-1)^p (-i\lambda^2)^n}{(2\pi)^{dn+1}} \sum_{\ell} \sum_{\underline{k}} \prod_{i}\widehat u_0(k_{0,i},\sigma_{0,i})  \\
& \qquad\quad \qquad  \int_{\mathbb{R}} T\widehat \chi(T(\tau-\tau_1))  \frac{1}{-\tau_1- |k_R|^2+\frac iT}\prod_{k=1}^n \frac{1}{-\tau_1- |k_R|^2 -\sum_{j=k}^{n}\Omega_j+\frac iT} \Delta_{\ell}( \underline{k},k_R)\, d\tau_1.
\end{align*}
The identity corresponding to \fref{LpLpexpression} is now
$$
\mathbb E \left\| \chi\left(\frac tT\right)u^n \right\|_{X^{s,b}}^{2} =\sum_{G,P} \mathcal F(G,P)
$$
with the following analogue of \fref{pivert}:
\begin{align*}
& \mathcal F(G,P)= \frac{T^2 \lambda^{4n}\ep^{d(2n+1)}}{(2\pi)^{2dn+2}} \sum_{\underline{k},\underline{k'}}  \langle \epsilon k \rangle^{2s} \prod_{\{i,j\}\in P} |A(\epsilon k_{0,i})|^2  \Delta_{\ell,\ell',P}(\underline{k},\underline{k}',0) \\
&  \int_{\mathbb R \times \mathbb R\times \mathbb R} \frac{1}{-\tau_1-|k_{n,1}|^2+\frac iT}\prod_{k=1}^n \frac{1}{-\tau_1-|k_{n,1}|^2-\sum_{j=k}^{n}\Omega_j+\frac iT}\frac{1}{-\tau_2-|k_{n,1}|^2+\frac iT}\prod_{k=1}^n \frac{1}{-\tau_2-|k_{n,1}|^2-\sum_{j=k}^{n}\Omega_j'+\frac iT} \\
& \qquad \qquad \qquad \qquad \langle \tau +|k_R|^2 \rangle^{2b}  \widehat \chi(T(\tau-\tau_1))\widehat \chi(T(\tau-\tau_2))\, d\tau_1 \, d\tau_2 \, d\tau \\
& = \underbrace{T^2 \lambda^{4n}\ep^{d(2n+1)} \int_{|\tau|\leq \ep^{-K}}[...]}_{\displaystyle \mathcal F_1(G,P)} + \underbrace{T^2 \lambda^{4n}\ep^{d(2n+1)} \int_{|\tau|> \ep^{-K}}[...]}_{\displaystyle \mathcal F_2(G,P)}.
\end{align*}
The second term above, $\mathcal{F}_2$, gains arbitrarily large powers of $\epsilon$ by choosing $K$ sufficiently big, making estimates easy. As for the first term $\mathcal F_1$, one can proceed identically to what was done in Subsection \ref{subsecproofl2}, with one exception: the bound $\sum_{|k_{n,1}| \lesssim \epsilon^{-1}} \int_{|\alpha,\alpha'| \leq \epsilon^{-K}} \frac{1}{|\alpha + \frac{i}{t}|} \frac{1}{|\alpha' + \frac{i}{t}|} \,d\alpha \,d\alpha' \lesssim \epsilon^{-d-\kappa}$ should be replaced by an estimate on
\begin{align*}
& \sum_{|k_{n,1}| \lesssim \epsilon^{-1}} \iiint \frac{1}{|-\tau_1- |k_{n,1}|^2 +\frac iT|} \frac{1}{|-\tau_2- |k_{n,1}|^2 +\frac iT|}\langle \tau + |k_{n,1}|^2 \rangle^{2b}\\
&  \qquad\qquad\qquad\qquad\qquad\qquad \widehat T^2\widehat \chi (T(\tau-\tau_1)) \chi (T(\tau-\tau_2)) \mathbbm{1}(\tau < K\ep^{-2}) \, d\tau_1 \,d \tau_2 \, d\tau.
\end{align*}
Since the convolution of $\frac{1}{|-\tau + C + \frac iT|}$ with $T\widehat{\chi}(T \tau)$ is bounded by a multiple of $\frac{1}{|-\tau + C + \frac iT|}$, the above can be bounded by
$$
\sum_{|k_{n,1}| \lesssim \epsilon^{-1}} \int_{|\tau| < \ep^{-K}} \frac{1}{|-\tau- |k_{n,1}|^2 +\frac iT|^2} \langle \tau + |k_{n,1}|^2 \rangle^{2b}\,d\tau  \lesssim \epsilon^{-d} \int_{|\tau| < \ep^{-K}}  \frac{1}{|-\tau +\frac iT|^2} \langle \tau \rangle^{2b} \,d\tau \lesssim \epsilon^{-d -K(2b-1)}.
$$
Given $K$, it now suffices to choose $b > \frac{1}{2}$ such that $K(2b-1)$ is as small as desired.

\subsection{Bound on the error} First, notice that the Fourier support of the approximate solution makes the choice of $s$ irrelevant in our scaled Sobolev and Bourgain spaces. Second, we discard the second summand in the definition of $E^N$, since it is easier to estimate. Next, by~\eqref{mesangebleue0},
$$
\left\| \chi(t) \int_0^t e^{i(t-s) \Delta} \chi(s) E^N \,ds \right\|_{X^{s,b}_\epsilon} \lesssim \left\|  \chi (t)\sum_{\substack{i,j,k\leq N \\ i+j+k \geq N}} u^i u^j u^k \right\|_{X^{s,b-1}_\epsilon}.
$$
We can now proceed as usual by interpolating $X^{s,b-1}_\epsilon$ between $X^{s,b'}_\epsilon$, with $b' < - \frac{1}{2}$, and a trivial but lossy bound in $X^{s,0}$. Omitting for the latter, we focus on the former and choose $b'<-\frac{1}{2}$. Then, by~\eqref{mesangebleue2},
\begin{align*}
 \left\|  \chi \left( t  \right) \sum_{\substack{i,j,k \leq N \\ i+j+k\geq N}} u^i u^j u^k \right\|_{X^{s,b'}_\epsilon}&  \lesssim \epsilon^{\frac{1}{2} - \frac{d}{4} - \kappa} \left\| \chi (t)\sum_{\substack{i,j,k \leq N \\ i+j+k\geq N}}u^i u^j u^k \right\|_{L^{4/3} L^{4/3}} \\
& \lesssim \epsilon^{\frac{1}{2} - \frac{d}{4}-\kappa} \sum_{\substack{i,j,k \leq N \\ i+j+k\geq N}} \| u^i \|_{L^4_{2} L^4} \| u^j \|_{L^4_{2} L^4} \| u^k \|_{L^4_{2} L^4}.
\end{align*}
By H\"older's inequality and the $L^p$ estimates of Proposition \ref{pr:LpLpinter},
\begin{align*}
& \mathbb{E}  \left\|  \chi (t) \sum_{\substack{i,j,k \leq N \\ i+j+k\geq N}} u^i u^j u^k \right\|_{X^{s,b'}_\epsilon} \\
& \qquad \qquad \qquad \qquad \qquad \lesssim  \epsilon^{\frac{1}{2} - \frac{d}{4}-\kappa} \sum_{\substack{i,j,k \leq N \\ i+j+k\geq N}} \left[ \mathbb{E} \| u^i \|_{L^4_{2} L^4}^4 \right]^{1/4}   \left[ \mathbb{E}\| u^j \|_{L^4_{2} L^4} ^4 \right]^{1/4}   \left[ \mathbb{E}\| u^k \|_{L^4_{2} L^4}^4 \right]^{1/4} \\
& \qquad \qquad \qquad \qquad \qquad \lesssim  \epsilon^{- \frac{1}{4} - \frac{d}{4}-\kappa}  \left( \frac{1}{T_{kin}} \right)^{\frac N2}.
\end{align*}

\section{Control of the linearization around $u^{app}$: proof of Proposition~\ref{prop:lineaire}} \label{sectionlinear}

We consider here the operator
$$
\mathcal{L} = 2\mathfrak L +\mathfrak L'
$$
where
$$
\mathfrak{L}: f \mapsto \lambda^2\chi\left( t \right) \left( f|v^{app}|^2 -fV -\langle v^{app},f\rangle v^{app}\right)
$$
and
$$
\mathfrak{L}': f \mapsto \lambda^2\chi\left( t \right) \left((v^{app})^2 \overline{f}-2\langle f,v^{app}\rangle v^{app}\right)
$$
and aim at proving Proposition~\ref{prop:lineaire}. We only prove the corresponding bound for the operator $\mathfrak L$. Indeed, the proof for $\mathfrak L'$ is verbatim the same.

\subsection{Reduction to elementary operators} The operator $\mathfrak{L}$ will be decomposed into
$$
\mathfrak L f=\sum_{i,j=1}^N \mathfrak{L}_{i,j}f, \quad \mathfrak{L}_{i,j}f= \lambda^2\chi\left( t \right)\left(   fu^i \overline{u^j}-f\langle u^j, u^i \rangle -\langle u^j,f\rangle u^i\right).
$$
and each $\mathfrak{L}_{i,j}$ is localised in frequency by letting
$$
\mathfrak{L}_{i,j,n}=\mathfrak{L}_{i,j} Q^n_{\ep,1}.
$$
The following Lemma gives an upper bound for these operators.

\begin{lemma} \label{lem:estimationmathfrakLijn}
For any $0<\kappa\ll 1$, for $\ep$ small enough, there exists a set $E_\kappa$ of measure greater than $1-\ep^\kappa$ such that on this set, for all $n\in \mathbb Z^d$ and $(i,j)\in \{0,N\}^2$ we have the following estimates for the operator norms:
$$
\| \mathfrak{L}_{i,j,1} \|_{X^{0,\frac{1}{2}} \to X^{0,-\frac{1}{2}}} \lesssim \left(\frac{1}{T_{kin}}\right)^{\frac{i+j+1}{2}} \ep^{-\kappa}.
$$
\end{lemma}

With the help of the above Lemma, we are able to prove Proposition \ref{prop:lineaire}.

\begin{proof}[Proof of Proposition \ref{prop:lineaire} from Lemma~\ref{lem:estimationmathfrakLijn}]

\underline{Almost locality.} We decompose the input and output function in frequency cubes by writing
$$
\mathfrak L_{i,j} u = \sum_{n,n'\in \mathbb Z^d} Q^{n'}_{1,N}\mathfrak L_{i,j}Q^{n}_{1,N}u.
$$
Since $\mathfrak L_{i,j}$ corresponds to convolution in frequency with a kernel localized in a frequency ball of size $C\ep^{-1}$, the operator is almost local: for $|n-n'|>R$ (a fixed constant),
$$
Q^{n'}_{\ep,1}\mathfrak L_{i,j}Q^{n}_{\ep,1}u=0.
$$
As a first consequence, we claim that, for any $n$
$$
\| \mathcal{L}_{i,j,n} \|_{X^{0,\frac{1}{2}} \to X^{0,-\frac{1}{2}}} \sim \| \mathcal{L}_{i,j,1} \|_{X^{0,\frac{1}{2}} \to X^{0,-\frac{1}{2}}}.
$$
Indeed, the main part of $\mathcal{L}_{i,j}$ is simply a convolution operator in space (and time) frequency. Thus, the boundedness of $\mathcal{L}_{i,j,1}$ would imply that of $\mathcal{L}_{i,j,n}$ were it not for the weights in $X^{0,\frac{1}{2}}$ and $X^{0,-\frac{1}{2}}$. Using almost locality, these weights cancel since
$$
\sup_{k,\ell} \langle \epsilon k \rangle^{1/2}  \langle \epsilon \ell \rangle^{-1/2} \mathbf{1}_{C^n_{\epsilon,1}} (k) \mathbf{1}_{C^{n'}_{\epsilon,1}}(\ell) \sim 1 \qquad \mbox{if $n-n' \leq R$},
$$
which gives the desired result.

\bigskip

\noindent
\underline{Bound from $X^{s,\frac{1}{2}}_\epsilon$ to $X^{s,-\frac{1}{2}}_\epsilon$.} By almost locality,
$$
\| \mathfrak L_{i,j}Q^{n}_{\ep,1} \|_{X^{s,\frac{1}{2}}_\epsilon \to X^{s,-\frac{1}{2}}_\epsilon} \sim \| \mathfrak L_{i,j}Q^{n}_{\ep,1} \|_{X^{0,\frac{1}{2}}_\epsilon \to X^{0,-\frac{1}{2}}_\epsilon}
$$
and, by almost orthogonality,
\begin{align*}
\|\mathfrak L_{i,j} u\|_{X^{s,-\frac 12}_\epsilon}& \lesssim \left[ \sum_{n\in \mathbb Z^d} \| \mathfrak L_{i,j}Q^{n}_{\ep,1}u \|_{X^{s,-\frac 12}_\epsilon}^2 \right]^{1/2} \lesssim \left(\sup_{n\in \mathbb Z^d} \| \mathfrak L_{i,j,n} \|_{X^{s,\frac 12}_\epsilon \to X^{s,-\frac 12}_\epsilon}\right) \left[ \sum_{n\in \mathbb Z^d} \| Q^{n}_{\ep,1}u \|_{X^{s,\frac 12}_\epsilon}^2 \right]^{1/2}\\
&\lesssim \left(\sup_{n\in \mathbb Z^d} \| \mathfrak L_{i,j,n} \|_{X^{0,\frac 12}\to X^{0,-\frac 12}}\right) \| u \|_{X^{s,\frac{1}{2}}} \lesssim \left(\frac{1}{T_{kin}}\right)^{\frac{i+j+1}{2}} \ep^{-\kappa} \| u \|_{X^{s,\frac{1}{2}}_\epsilon}.
\end{align*}
on the set $E_\kappa$.

\bigskip
\noindent \underline{Bound from $X^{s,0}_\epsilon$ to $X^{s,0}_\epsilon$ and interpolation} Since $X^{s,0}_\epsilon$ is simply $L^2_t H^s_\epsilon$, and since $u_i$ and $u_j$ are localized in frequency in a ball of radius $C\epsilon^{-1}$, the operator norm of $\mathfrak{L}_{i,j}$ is less than $\| u^i \|_{L^\infty} \| u^j \|_{L^\infty}\lesssim \ep^{-d} \| u_i \|_{X^{s,b}_\ep}\| u_j \|_{X^{s,b}_\ep}$. From \fref{bd:fourierwnXsb} and Bienaym\'e-Tchebychev, for $\kappa$ small enough depending on $\gamma$, a set $E$ with $\mathbb P(E)\geq 1-\ep^\kappa$ exists on which $\| u_i \|_{X^{s,b}_\ep}\lesssim 1$. Hence the operator norm from $X^{s,0}_\epsilon$ to $X^{s,0}_\epsilon$ can be bounded by $\ep^{-d}$.

Interpolating between this very rough bound and the $X^{s,\frac{1}{2}} \to X^{s,-\frac{1}{2}}$ bound, we obtain a bound from $X^{s,\frac{1}{2}-\delta}$ to $X^{s,-\frac{1}{2}+\delta}$ with a loss $\epsilon^{-\kappa}$, where $\kappa$ can be made arbitrarily small by choosing $\delta$ sufficiently small. There remains to choose $b>\frac{1}{2}$ such that $b-1<-\frac{1}{2}+\delta$
\end{proof}

\subsection{The trace of the iterated operator and its diagrammatic representation}

{The key idea will be to use the inequality $\| \mathfrak T \|\lesssim \| (\mathfrak T^*\mathfrak T)^{N}\|^{1/2N}$ to relate the control of a random operator $\mathfrak T$ to that of a nonnegative self-adjoint operator, and then to use the control of the operator norm by the trace $\| (\mathfrak T^*\mathfrak T)^{N}\|\leq \operatorname{Tr}(\mathfrak T^*\mathfrak T)^N$, together with $N\rightarrow \infty$.}
By transfering the weight from the function space to the operator, and getting rid of irrelevant constants, the operator norm of $\mathfrak{L}_{i,j,1}$ from $X^{0,\frac{1}{2}}_\epsilon$ to $X^{0,-\frac{1}{2}}_\epsilon$ equals that of the operator $\mathfrak{K}$
\begin{align*}
\mathfrak{K}: & L^2(\mathbb{R} \times \mathbb{Z}^d) \to L^2(\mathbb{R} \times \mathbb{Z}^d) \\
 & f(\tau_0,k_0) \mapsto \sum_{k_0} \int K(\tau_3,\tau_0,k_3,k_0) f(\tau_0,k_0) \,d\tau_0
 \end{align*}
with kernel
 \begin{align*}
 &K(\tau_3,\tau_0,k_3,k_0)  = \lambda^2\langle \tau_0 + |k_0|^2 \rangle^{-\frac{1}{2}}  \langle \tau_3 + |k_3|^2 \rangle^{-\frac{1}{2}} \sum_{\substack{ k_0+k_1 + k_2 = k_3 }} \mathbf{1}_{C^n_{\ep,1}}(k_0) \\ 
& \qquad \qquad \iint \widetilde{u^i}(k_1,\tau_1)\widetilde{\overline{ u^j}}(k_2,\tau_2) \left(1-\delta(k_1+k_2)-\delta(k_2+k_0)\right) \widehat{\chi}(\tau_3 - \tau_0 -\tau_1 -\tau_2) \, d\tau_1 \, d\tau_2.
\end{align*}
To compute the adjoint kernel, we change variables by setting $(k_0',k_1',k_2',k_3') = (k_3,-k_2,-k_1,k_0)$ and $(\tau_0',\tau_1',\tau_2',\tau_3') = (\tau_3,-\tau_2,-\tau_1,\tau_0)$. Getting rid of primes gives the following formula for the adjoint kernel:
\begin{align*}
&K^*(\tau_3,\tau_0,k_3,k_0)  = \lambda^2\langle \tau_0+|k_0|^2 \rangle^{-\frac{1}{2}}  \langle \tau_3+|k_3|^2 \rangle^{-\frac{1}{2}} \sum_{\substack{ k_0+k_1 + k_2 = k_3 }} \mathbf{1}_{C^n_{\ep,1}}(k_3) \\ 
& \qquad \qquad \iint \widetilde{u^j}(k_1,\tau_1)\widetilde{ \overline{ u^i}}(k_2,\tau_2) \left(1-\delta(k_1+k_2)-\delta(k_2+k_0)\right)  \widehat{\chi}(\tau_3 - \tau_0 -\tau_1 -\tau_2)\,d\tau_1\,d\tau_2
\end{align*}
(here we are using that $\chi$ is even, which can be assumed without loss of generality). Composing and iterating, we see that the operator $(\mathfrak{M})^N = ((\mathfrak{K})^* \mathfrak{K})^N$ has kernel
\begin{align*}
& M^{N}(\tau_{6N},\tau_0,k_{6N},k_0) \\
&\quad = \lambda^{4N} \sum_{\substack{k_1, \dots, k_{6N-1} }}\int_{\mathbb R^{6N-1}} \langle \tau_{6N}+|k_{6N}|^2 \rangle^{-1/2}  \langle \tau_0+|k_0|^2 \rangle^{-1/2} \Delta(\underline{k})  \prod_{m=0}^{N}\mathbf{1}_{C^n_{\ep,1}}(k_{6m})\\
&\qquad  \qquad \prod_{m=0}^{N-1} \widetilde{ u^i}(k_{6m+1},\tau_{6m+1})\widetilde{ \overline{u^j}}(k_{6m+2},\tau_{6m+2})\widetilde{  u^j} (k_{6m+4},\tau_{6m+4})\widetilde{ \overline{ u^i}}(k_{6m+5},\tau_{6m+5})  \\
&\qquad \qquad  \prod_{m=1}^{2N-1} \langle \tau_{3m}+|k_{3m}|^2 \rangle^{-1} \prod_{m=0}^{2N-1} \widehat{\chi}(\tau_{3m+3}-\tau_{3m} -\tau_{3m+1} -\tau_{3m+2}) \, d\tau_1\dots d\tau_{6N-1}
\end{align*}
where $\underline{k} = (k_0,\dots,k_{6N})$ and
\begin{align*}
& \Delta (\underline{k}) = \prod_{m=0}^{2N-1} \delta(k_{3m+3}-k_{3m} - k_{3m+1} - k_{3m+2} )\left(1-\delta(k_{3m+1}+k_{3m+2})-\delta(k_{3m+1}+k_{3m})\right) 
\end{align*}
Setting
\begin{align*}
\omega_{3m} & = \tau_{3m} + |k_{3m}|^2, \quad \omega_{3m+1}= \tau_{3m+1} + |k_{3m+1}|^2 ,\quad \omega_{3m+2}=\tau_{3m+2} - |k_{3m+2}|^2  \\
\Omega_m & = \tau_{3m+3}-\tau_{3m} -\tau_{3m+1} -\tau_{3m+2} \\
& = -|k_{3m+3}|^2 + |k_{3m}|^2+|k_{3m+1}|^2-|k_{3m+2}|^2+\omega_{3m+3}-\omega_{3m}-\omega_{3m+1}-\omega_{3m+2},
\end{align*}
this becomes
\begin{align*}
& M^{N}(\tau_{6N},\tau_0,k_{6N},k_0) =  \lambda^{4N} \sum_{\substack{k_1, \dots, k_{6N-1} }}  \langle \omega_{6N}\rangle^{-\frac 12}  \langle \omega_{0}\rangle^{-\frac 12} \Delta(\underline{k})  \prod_{m=0}^{N}\mathbf{1}_{C^n_{\ep,1}}(k_{6m})\\
& \qquad \qquad \int_{\mathbb R^{6N-1}} \prod_{m=0}^{N-1}  \widetilde{ u^i}(k_{6m+1},\tau_{6m+1})\widetilde{ \overline{ u^j}}(k_{6m+2},\tau_{6m+2}) \widetilde{  u^j}(k_{6m+4},\tau_{6m+4})\widetilde{ \overline{u^i}}(k_{6m+5},\tau_{6m+5})   \\
& \qquad \qquad \qquad \qquad \qquad \qquad  \prod_{m=1}^{2N-1} \langle \omega_{3m} \rangle^{-1} \prod_{m=0}^{2N-1} \widehat{\chi}(\Omega_m) \, d\omega_1\dots d\omega_{6N-1}.
\end{align*}
Taking the trace gives 
\begin{align*}
&\operatorname{Tr} (\mathfrak{M})^N  = \lambda^{4N} \sum_{\underline{k}}\int_{\mathbb R^{6N+1}}  \Delta(\underline{k}) \delta(\omega_0-\omega_{6N})\delta(k_0-k_{6N}) \prod_{m=0}^{N}\mathbf{1}_{C^n_{\ep,1}}(k_{6m})\\
& \prod_{m=0}^{N-1} \widetilde{ u^i} (k_{6m+1},\tau_{6m+1})\widetilde{ \overline{ u^j}}(k_{6m+2},\tau_{6m+2}) \widetilde{  u^j}(k_{6m+4},\tau_{6m+4})\widetilde{ \overline{ u^i}} (k_{6m+5},\tau_{6m+5})  \prod_{m=0}^{2N-1} \langle \omega_{3m} \rangle^{-1}  \widehat{\chi}( \Omega_m) \, d\underline{\omega} 
\end{align*}
This can be represented by the following interaction diagram, in which the input and output frequencies are $k_0$ and $k_{6N}$ respectively, and are equal since the trace was taken.
\vspace*{0.2cm}
\begin{center}
\includegraphics[width=14cm]{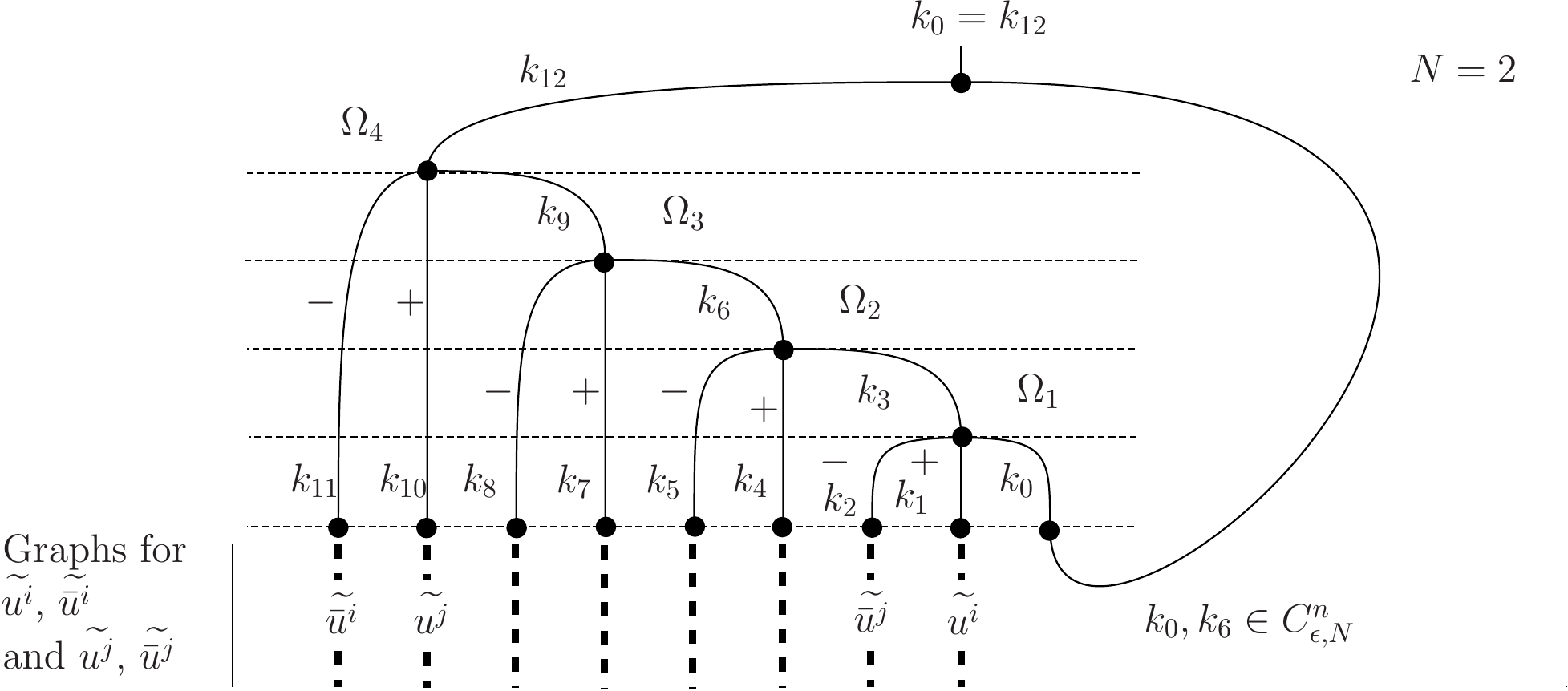}
\end{center}
\vspace*{0.2cm}
We can now take the expectation, and use the computations in Section~\ref{proofxsb} (after multiplying by a harmless cutoff function $\chi (2t)$ all $u^i$'s and $u^j$'s). This gives the bound
\begin{equation}\label{id:traceMNsharpijlineaire}
\mathbb{E} \operatorname{Tr} \mathfrak{M}^N   =   \lambda^{4N(i+j+1)}\ep^{2Nd(i+j+1)} \sum_{\vec \ell=(\ell_1,\ell_2,...,\ell_{4N})\in \mathcal G_i\times \mathcal G_j\times...\times \mathcal G_j}\sum_P \mathcal F(\vec l,P),
\end{equation}
where $\vec \ell=(\ell_1,...,\ell_{4N})$ represents all possible interaction histories for the $u^i$ and $u^j$ terms, $P$ is summed over all possible pairings of the initial vertices, and where 
\begin{align}
& \nonumber \mathcal F(\vec l,P)  =\sum_{\underline k}\int d\underline{\tau} \Delta_{\ell_1,...,\ell_{6N-1},P}(\underline k)\delta(\omega_0-\omega_{6N})\delta(k_0-k_{6N}) \prod_{\{i,j\}\in P} \mathbf 1(|\underline{k}|\lesssim \ep^{-1})\\
\label{id:mathcalFlineaire}& \prod_{m=0}^{2N-1} \prod_{n=1}^i \frac{1}{|\omega_{3m+1}-\widetilde \omega_{3m+1}-\sum_{r= n}^i \Omega_r^{3m+1}+i|}\prod_{n=1}^j \frac{1}{|\omega_{3m+2}-\widetilde \omega_{3m+2}-\sum_{r= n}^j \Omega_r^{3m+2}+i|}  \\
\nonumber & \prod_{m=0}^{2N-1} \widehat \chi (\Omega_m)\langle \omega_{3m}\rangle^{-1}\frac{\widehat \chi (\widetilde \omega_{3m+1})}{|\omega_{3m+1}-\widetilde \omega_{3m+1}+i |}\frac{\widehat \chi (\widetilde \omega_{3m+2})}{|\omega_{3m+2}-\widetilde \omega_{3m+2}+i |}d\omega_{3m}d\omega_{3m+1}d\omega_{3m+2}d\widetilde \omega_{3m+1}d\widetilde \omega_{3m+2}
\end{align}
where for $n=1,2$:
$$
\widetilde \omega_{3m+n}=\tau_{3m+n}-\widetilde \tau_{3m+n},
$$
and where $\Delta_{\ell_1,...,\ell_{6N-1},P}(\underline k)$ records all Kirchhoff laws of the graph. Once a pairing has been fixed, we can represent formula \fref{id:mathcalFlineaire} as an interaction diagram in which initial vertices are paired. This is done the same way as in Subsection \ref{sectionfeynman}. Then we construct a spanning tree for the graph the same way as in Theorem \ref{th:spanning}. For this construction, we treat first each subgraph of $u^i$ and $u^j$, from right to left; then we turn to the vertices corresponding to the resonance moduli $\Omega_m$, in increasing order of $m$.

We denote by $\widetilde n$ the number of free edges ending to or starting from the top vertices. We denote by $\widetilde k_{1}^f,...,\widetilde k_{\widetilde n}^f$ the corresponding free frequencies, where the order is from right to left in the graph. The spanning tree then yields a collection of free edges $(k^f_i)_{i\in \{1,...,2N(i+j+1)+1-\widetilde n\}}$ and $(\widetilde k_{i}^f)_{1\leq i \leq \widetilde n}$ from which all other edges are recovered in the graph using the Kirchhoff laws. An example for free edges for the $\Omega_m$'s vertices is as follows:

\begin{center}
\includegraphics[width=14cm]{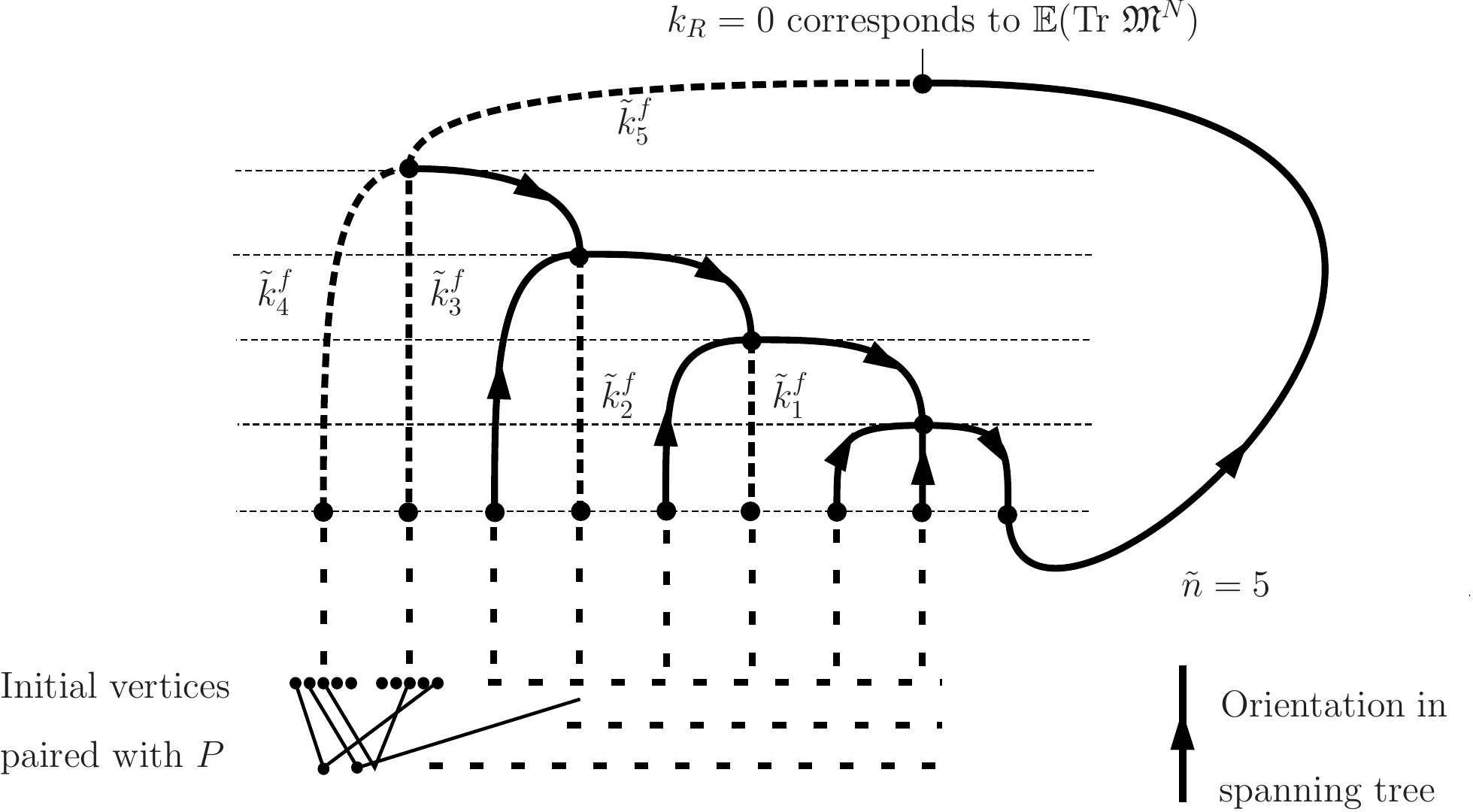}
\end{center}

Not that, since in the spanning tree algorithm top vertices are considered after the vertices associated to the graphs of $u^i$ and $u^j$, therefore frequencies below top vertices are only sums of free frequencies associated to top vertices. Namely, for all $m\in \{0,...,6N\}$, if $k_m$ is not a free frequency, then $k_m$ is a sum of free frequencies among $\widetilde k_1^f,...,\widetilde k^f_{\widetilde n}$ only, that moreover appear after $k_m$ for the natural ordering of the graph.\\

We define degree zero, one and two vertices as in Subsection \ref{subsec:kirchhoff}, and Lemma \ref{lem:Lpinter} holds true. We denote by $n_0$, $n_1$ and $n_2$ the number of degree zero, one and two interaction vertices \emph{inside} a graph for either $u^i$ or $u^j$. The other vertices are those associated with the phases $\Omega_1,...,\Omega_{2N}$ on top of the graph and we denote by $\widetilde n_0$, $\widetilde n_1$ and $\widetilde n_2$ the number of degree zero, one and two vertices among these vertices. In particular we get:
\be \label{id:relationntildenlineaire}
n_0+\widetilde n_0+n_1+\widetilde n_1+n_2+\widetilde n_2=2N(i+j+1), \quad \mbox{and} \quad n_1+\widetilde n_1+2n_2+2\widetilde n_2=2N(i+j+1).
\ee
$$
\widetilde n_1+2\widetilde n_2=\widetilde n-1.
$$

\subsection{Bounding the trace}
We are now ready to estimate \fref{id:mathcalFlineaire}. First, we notice that as in the proof for the $X^{s,b}$ estimate for $u^i$, if one takes $|\omega_{3m+1}|\geq \ep^{-K}$ or $|\omega_{3m+2}|\geq \ep^{-K}$ then the contribution of the third line in \fref{id:mathcalFlineaire}, once integrated, is negligible. Hence, we focus on the case for which $|\omega_{3m+1}|,|\omega_{3m+2}|\leq \ep^{-K}$ for $m=0,...,2N-1$. We first integrate over all vertices inside the graphs for $u^i$ and $u^j$, following the natural time ordering of the graph, from right to left. Using Lemmas \ref{macareux1} and \ref{macareux2} this produces
\begin{align*}
\mathcal F(\vec l,P) & \lesssim \epsilon^{-\kappa} (\ep^{-2(d-1)})^{n_2} \ep^{-(d-1)n_1} \sum_{\underline{\widetilde k^f}}\int \Delta(\underline k)\delta(\tau_0-\tau_{6N})\delta(k_0-k_{6N})\\
\nonumber & \prod_{m=0}^{2N-1} \widehat \chi (\Omega_m)\langle \omega_{3m}\rangle^{-1}\frac{\widehat \chi (\widetilde \omega_{3m+1})}{|\omega_{3m+1}-\widetilde \omega_{3m+1}+i |}\frac{\widehat \chi (\widetilde \omega_{3m+2})}{|\omega_{3m+2}-\widetilde \omega_{3m+2}+i |}\\
&d\omega_{3m}d\omega_{3m+1}d\omega_{3m+2}d\widetilde \omega_{3m+1}d\widetilde \omega_{3m+2}\prod_{m=0}^{2N}\mathbf 1(|\omega_{3m+1}|,|\omega_{3m+2}|\leq \ep^{-K})  \mathbf 1(|\underline{\widetilde k^f}|\lesssim \ep^{-1})+O(\ep^K) \\
&\lesssim  \epsilon^{-\kappa}(\ep^{-2(d-1)})^{n_2} \ep^{-(d-1)n_1} \left(\underbrace{ \sum_{\underline{\widetilde k^f}}\int_{\forall m, \ |\omega_{3m}|\leq \ep^{-\widetilde K}}[...]}_{\displaystyle \mathcal F_1(G,P)} + \underbrace{\sum_{\underline{\widetilde k^f}}\ \int_{\exists m, \ |\omega_{3m}|\geq \ep^{-{\widetilde K}}}[...]}_{\displaystyle \mathcal F_2(G,P)}
\right)
\end{align*}
where $\widetilde K \gg K$ is another very large constant. The second part $\mathcal F_2$ gives an irrelevant contribution. Indeed, we write:
$$
\mathcal F_2(G,P)\leq  \sum_{m=0}^{2N-1} \int_{S_m}[...]
$$
where we define $S_m=\{|\omega_{3m}|\geq \ep^{-K}\mbox{ and } |\omega_{3m'}|\leq |\omega_{3m}| \ \forall m'\}$. Notice that on $S_m$ that we have $\Omega_m=\omega_{3m+3}-\omega_{3m}+O(C\ep^{-K})$ so that if all momenta are kept fixed:
$$
\int_{|\omega_{3m+3}|\leq |\omega_{3m}|} \langle \omega_{3m+3}\rangle^{-1} \widehat \chi (\Omega_{3m})d\omega_{3m+3}\lesssim \langle \omega_{3m} \rangle^{-1}.
$$
Hence, on $S_m$, integrating first with respect to $\omega_{3m+3}$ using the above estimate, then integrating over all $\omega_{3m'}$ for $m'\neq m$ producing a $\langle \ln \ep \rangle$ factor, then integrating over all $\widetilde k^f_i$ producing a $\ep^{-d(\widetilde n_1+2\widetilde n_2)}$ factor, and over $k_{6N}$ producing a $\ep^{-d}$ factor, we arrive at:
$$
 \int_{S_m}[...] \lesssim \epsilon^{-\kappa} \ep^{-O(N)} \int_{|\omega_{3m}|\geq \ep^{-\widetilde K}} \langle \omega_{3m} \rangle^{-2}d\omega_{3m}\lesssim \ep^{\frac{\widetilde K}{2}}=O(\ep^K).
$$
This shows that $\mathcal F_2$ gives an irrelevant contribution:
$$
\mathcal F_2(G,P)=O(\ep^{K}).
$$
We now claim that for any $P$ and $\vec l$,
$$
\mathcal F_1(\vec l,P) \lesssim   \epsilon^{-\kappa} \ep^{-2Nd(i+j+1)+2N(i+j+1)-d}
$$
As in earlier estimates, we integrate iteratively at interaction vertices from left to right, each time over the free variables below it. When the vertex is of degree zero we get a factor $1$, when it is of degree one we use Lemma \ref{colombe}, and when it is of degree 2 we use Lemma \ref{lemmavertextwo}.

 At the end of the integration process, we integrate over the last free variable $k_{6N}$ which produces a factor $\ep^{-d}$, and integrate $\langle \omega_i\rangle^{-1}d\omega_i$ and $\widehat \chi ( \widetilde \omega_i)d\widetilde \omega_i$ over the $\omega_i$ and $\widetilde \omega_i$'s variables which, from the constraint $|\omega_i|\leq \ep^{-K}$ for all $i$, produces a factor $\langle \ln \ep \rangle^{2N}$. This produces the following estimate:
\begin{align}
\nonumber \mathcal F(\vec l,P) &\lesssim \epsilon^{-\kappa}(\ep^{-2(d-1)})^{n_2} \ep^{-(d-1)n_1} \ep^{-2(d-1)\widetilde n_2}\ep^{-(d-1)\widetilde n_1}\ep^{-d}\\
\label{bd:estimationgraphlineairesharpij} &\lesssim \epsilon^{-\kappa} \ep^{-2Nd(i+j+1)+2N(i+j+1)-d},
\end{align}
thanks to \fref{id:relationntildenlineaire}.

\subsection{End of the proof of Lemma~\ref{lem:estimationmathfrakLijn}}

We can now end the proof of the Lemma. From the identity \fref{id:traceMNsharpijlineaire} and \fref{bd:estimationgraphlineairesharpij}, we obtain that the trace of $\mathfrak M^N$ is independent of $n$ and satisfies:
\begin{align*}
\mathbb{E} \operatorname{Tr} \mathfrak{M}^N& \lesssim  \lambda^{4N(i+j+1)}\ep^{2Nd(i+j+1)} \epsilon^{-\kappa} \ep^{-2Nd(i+j+1)+2N(i+j+1)-d}\\
&= \left(\frac{1}{T_{kin}}\right)^{N(i+j+1)} \ep^{-d}\epsilon^{-\kappa}.
\end{align*}
Hence, via Bienaym\'e-Tchebychev, for any $\kappa$, there exists a set of measure greater than $1-\ep^\kappa$ such that
$$
\operatorname{Tr} \mathfrak{M}^N \lesssim \left(\frac{1}{T_{kin}}\right)^{N(i+j+1)} \ep^{-d- 2\kappa}.
$$
On this set,
$$
\| \mathfrak L_{i,j,n}\|_{X^{0,\frac 12}\rightarrow X^{0,-\frac 12}}\leq \left(\operatorname{Tr} \mathfrak{M}^N\right)^{\frac{1}{2N}}\lesssim \left(\frac{1}{T_{kin}}\right)^{\frac{i+j+1}{2}} \ep^{-\frac{d+2\kappa}{2N}}\epsilon^{-\kappa} \lesssim  \left(\frac{1}{T_{kin}}\right)^{\frac{i+j+1}{2}} \ep^{-\kappa}
$$
choosing $N$ large enough.

\section{The nonlinear terms: proof of propositions~\ref{propbilinear} and~\ref{proptrilinear}}

\label{sectionnonlinear}

We will consider in this section that
\begin{align*}
& \mathcal{B}(u) = \lambda^2 \left[ 2 |u|^2 v^{app} + u^2 \overline{v^{app}} \right] \\
& \mathcal{T}(u) = \lambda^2  |u|^2 u.
\end{align*}
Indeed, the additional terms in the definition of $\mathcal{B}$ and $\mathcal{T}$ in Section~\ref{sectionproofmaintheorem} can be treated similarly (actually, even more simply).

\subsection{A simplified case: restricting to low frequencies}
We add here a projector on frequencies $\lesssim \frac{1}{\epsilon}$, which is the range of interest physically. It greatly simplifies nonlinear estimates and in particular
$$
\| P_{\epsilon,1} f \|_{X^{s,b}_\epsilon} \sim \| f \|_{X^{0,b}},
$$
where the projection operator $P_{\epsilon,1}$ is defined in Section~\ref{notations}, and we simply denote $X^{0,b}$ instead of $X^{0,b}_\epsilon$. We will use the shorthand
$$
u_\epsilon = P_{\epsilon,1}u.
$$

\begin{proposition} For any $\kappa>0$, there exists $b>\frac{1}{2}$ such that
$$
\left\| \chi(t) \int_0^t  e^{i(t-s) \Delta} \chi(s) |u_\epsilon|^2 u_\epsilon \, ds \right\|_{X^{0,b}} \lesssim \epsilon^{2-d-\kappa} \| u_\epsilon \|_{X^{0,b}}^3
$$
\end{proposition}

\begin{proof} First, by~\eqref{mesangebleue0}, for $b > \frac{1}{2}$,
$$
\left\| \chi(t) \int_0^t e^{i(t-s) \Delta} \chi(s) |u_\epsilon|^2 u_\epsilon \, ds \right\|_{X^{0,b}} \lesssim_b \left\| \chi(t) |u_\epsilon|^2 u_\epsilon \right\|_{X^{0,b-1}}.
$$

\bigskip \noindent \underline{Estimate $(X^{0,b})^3 \to X^{0,b'}$, with $b>\frac{1}{2}$, $b' < - \frac{1}{2}$} Applying successively~\eqref{mesangebleue2}, H\"older's inequality, and~\eqref{mesangebleue1}, we get
\begin{align*}
\left\| \chi(t)  |u_\epsilon|^2 u_\epsilon \right\|_{X^{0,b'}} & \lesssim_{b',\kappa} \epsilon^{\frac{1}{2} - \frac{d}{4} -\kappa} \|   |u_\epsilon|^2 u_\epsilon \|_{L^{4/3}_2 L^{4/3}} \\
& \lesssim  \epsilon^{\frac{1}{2} - \frac{d}{4}-\kappa} \| u_\epsilon \|_{L^4_2 L^4}^3 \\
& \lesssim_{b,\kappa} \epsilon^{2 - d - 4 \kappa} \| u_\epsilon \|_{X^{0,b}}^3
\end{align*}

\bigskip \noindent \underline{Estimate $(X^{0,b})^3 \to X^{0,0}$ for $b>\frac{1}{2}$} Using successively the definition of $X^{s,b}$, H\"older's inequality, the Sobolev embedding theorem and~\eqref{cormorant},
\begin{align*}
\left\| \chi(t) |u_\epsilon|^2 u_\epsilon \right\|_{X^{0,0}} & \lesssim  \left\| |u_\epsilon|^2 u_\epsilon \right\|_{L^2_2 L^2} \leq \| u_\epsilon \|_{L^\infty_2 L^6}^3 \\
& \lesssim \epsilon^{-d} \| u_\epsilon \|_{L^\infty L^2}^3 \\
& \lesssim_b  \epsilon^{-d} \| u_\epsilon \|_{X^{0,b}}^3.
\end{align*}

\bigskip \noindent \underline{Interpolation.} Interpolating between the two estimates above gives that for any $\kappa>0$, there exists $b>\frac{1}{2}$ such that
$$
\left\| \chi(t) |u_\epsilon|^2 u_\epsilon \right\|_{X^{0,b-1}} \lesssim_{b,\kappa} \epsilon^{2 - d-\kappa} \left\| u_\epsilon \right\|_{X^{0,b}}^3.
$$
\end{proof}

\begin{proposition} For any $\kappa>0$, if $\mu$ in Corollary \ref{aigleroyal} is chosen sufficiently small, there exists $b> \frac{1}{2}$ such that
$$
\left\| \chi(t) \int_0^t  e^{i(t-s) \Delta} \chi(s) |u_\epsilon|^2 v^{app} \, ds \right\|_{X^{0,b}} \lesssim \epsilon^{{-2\mu}}  \epsilon^{\frac{1}{2} - \frac{d}{2}}  \| u_\epsilon \|_{X^{0,b}}^2.
$$
The same result holds if $|u_\epsilon|^2 v^{app}$ is replaced by $(u_\epsilon)^2 \overline{v^{app}}$.
\end{proposition}

\begin{proof}
First, by~\eqref{mesangebleue0}, for $b > \frac{1}{2}$,
$$
\left\| \chi(t) \int_0^t e^{i(t-s) \Delta} \chi(s)  |u_\epsilon|^2 v^{app} \, ds \right\|_{X^{0,b}} \lesssim_b \left\| \chi(t) |u_\epsilon|^2 v^{app} \right\|_{X^{0,b-1}}.
$$

\bigskip \noindent \underline{Estimate $(X^{0,b})^3 \to X^{0,b'}$, with $b>\frac{1}{2}$, $b' < - \frac{1}{2}$} Applying successively the Sobolev embedding theorem, H\"older's inequality, and Corollary \ref{aigleroyal}, we get for $p \geq q \geq \frac{1}{2}$
\begin{align*}
\| v^{app} \|_{L^q_2 L^\infty} & \lesssim \epsilon^{-\frac{d}{p}}\| v^{app} \|_{L^q_2 L^p} \lesssim \epsilon^{-\frac{d}{p}}  \| v^{app} \|_{L^p_2 L^p}\\
&  \lesssim \epsilon^{-\mu} \epsilon^{-\frac dp}  (\epsilon^{-1})^{\frac{1}{2} - \frac{1}{p}} =  \epsilon^{-\frac{1}{2} + \frac{1-d}{p}-\mu}.
\end{align*}
As a consequence, for any $\kappa>0$ for $p$ large enough:
\begin{equation}
\label{mesangecharbonniere}
\| v^{app} \|_{L^q_2 L^\infty} \lesssim \epsilon^{-\frac{1}{2} - \mu-\kappa}.
\end{equation}
Therefore, taking $q=2$, by~\eqref{mesangebleue3}, H\"older's inequality, and~\eqref{mesangebleue1}:
\begin{eqnarray*}
\| \chi(t) |u_\epsilon|^2 v^{app} \|_{X^{0,b'}} &\lesssim& \| |u_\epsilon|^2v^{app} \|_{L^1_2 L^2} \lesssim \| v^{app} \|_{L^2 L^\infty}  \| u_\epsilon \|_{L^4_2 L^4}^2 \lesssim \epsilon^{-\frac{1}{2} -\mu-\kappa} \epsilon^{1 - \frac{d}{2} - \kappa} \| u_\epsilon \|_{X^{0,b}}^2\\
&\lesssim & \ep^{\frac 12 -\frac d2{-\mu-\kappa}}\| u_\epsilon \|_{X^{0,b}}^2.
\end{eqnarray*}

\bigskip \noindent \underline{Estimate $(X^{0,b})^3 \to X^{0,0}$ for $b>\frac{1}{2}$} By definition of $X^{s,b}$, H\"older's inequality, the Sobolev embedding theorem and~\eqref{cormorant},
\begin{align*}
\| \chi(t) |u_\epsilon|^2 v^{app} \|_{X^{0,0}} & \lesssim \| |u_\epsilon|^2 v^{app} \|_{L^2_2 L^2} \lesssim \| v^{app} \|_{L^\infty L^6} \|u_\epsilon\|_{L^\infty L^6}^2 \\
& \lesssim \epsilon^{-2d/3} \| u_\epsilon \|_{L^\infty L^2}^2 \leq  \epsilon^{-2d/3} \| u_\epsilon \|_{X^{0,b}}^2.
\end{align*}

\bigskip \noindent \underline{Interpolation.} Interpolating between the two estimates above and taking $\kappa$ small enough depending on $\mu$ gives the desired result.
\end{proof}

\subsection{Proof of Proposition~\ref{proptrilinear}: the trilinear bound} We aim at proving that if $s > \frac{d}{2} - 1$, for any $\kappa>0$ there exists $b>\frac{1}{2}$ such that
$$
\left\| \chi(t) \int_0^t e^{i(t-s) \Delta} \chi(s) |u|^2 u \, ds \right\|_{X^{s,b}_\epsilon} \lesssim \epsilon^{2-d-\kappa} \| u \|_{X^{s,b}_\epsilon}^3.
$$
The starting point is~\eqref{mesangebleue0}, which gives, for $b > \frac{1}{2}$,
$$
\left\| \chi(t) \int_0^t e^{i(t-s) \Delta} \chi(s) |u|^2 u \, ds \right\|_{X^{s,b}_\epsilon} \lesssim_b \left\| \chi(t) |u|^2 u \right\|_{X^{s,b-1}_\epsilon}.
$$
By duality, it will therefore suffice to show that
\begin{equation}
\label{goldfinch}
\sup_{\| v \|_{X^{-s,1-b}_\epsilon }\leq 1} \iint \chi(t) |u|^2 u \overline{v} \,dx\,dt \lesssim \epsilon^{2-d-\kappa} \| u \|_{X^{s,b}_\epsilon}^3.
\end{equation}
 As a preparation for this estimate, we will use the following lemma. Recall that the projection operators $P_{\epsilon,N}$ and $Q_{\epsilon,N}^n$ are defined in Section~\ref{notations}.
 
 \begin{lemma} \label{robin} If $N_1 \leq N_2 \in 2^{\mathbb{N}_0}$, for any $\kappa>0$ there exists $b_0 < \frac{1}{2}$ such that
 $$
 \| P_{\epsilon, N_1} u P_{\epsilon, N_2} v \|_{L^2_2 L^2} \lesssim N_1^{\frac{d}{2}-1 + \kappa} \epsilon^{-\frac{d}{2}+1-\kappa} \| P_{\epsilon, N_1} u \|_{X^{0,b_0}_{\epsilon}} \| P_{\epsilon, N_2} v \|_{X^{0,b_0}_{\epsilon}}
$$
The same holds if $u$ or $v$ are replaced by their complex conjugates.
\end{lemma}
\begin{proof}  \underline{Step 1: the estimate $(X^{0,\frac{1}{2}+}_\epsilon)^2 \to L^2_2 L^2$}. 
By almost orthogonality followed by H\"older's inequality,
$$
\| P_{\epsilon, N_1} u P_{\epsilon, N_2} v \|_{L^2_2 L^2}^2 \lesssim \sum_{n \in \mathbb{Z}^d} \|  P_{\epsilon, N_1} u Q_{\epsilon, N_1}^n P_{\epsilon, N_2} v \|_{L^2_2 L^2}^2 \lesssim \sum_n \| P_{\epsilon, N_1} u \|_{L^4_2 L^4}^2  \|Q_{\epsilon, N_1}^n P_{\epsilon, N_2} v \|_{L^4_2 L^4}^2 .
$$
Applying now the Strichartz estimates, followed once again by almost orthogonality, we get for $b>\frac{1}{2}$
$$
\dots \lesssim_b N_1^{\frac{d}{2}-1+\kappa} \epsilon^{-\frac{d}{2}+1-\kappa} \sum_n \| P_{\epsilon, N_1} u \|_{X^{0,b}_\epsilon}^2  \|Q_{\epsilon, N_1}^n P_{\epsilon, N_2} v \|_{X^{0,b}_{\epsilon}}^2 \lesssim  N_1^{\frac{d}{2}-1+\kappa} \epsilon^{-\frac{d}{2}+1-\kappa}   \| P_{\epsilon, N_1} u \|_{X^{0,b}_{\epsilon}}^2 \| P_{\epsilon, N_2} v \|_{X^{0,b}_{\epsilon}}^2
$$

\bigskip

\noindent
\underline{Step 2: the estimate $(X^{0,\frac{1}{4}+}_\epsilon)^2 \to L^2_2 L^2$} Interpolating between the inequalities $\displaystyle \| v \|_{L^\infty L^2} \lesssim \| v \|_{X^{0,b}}$ if $b>\frac{1}{2}$, and $\displaystyle \| v \|_{L^2 L^2} = \| v \|_{X^{0,0}}$ gives for any $\delta>0$
$$
\| v \|_{L^4 L^2} \lesssim_\delta \| v \|_{X^{0,\frac{1}{4}+\delta}}.
$$ 

Applying H\"older's inequality, followed by Sobolev embedding and the above inequality,
\begin{align*}
\| P_{\epsilon, N_1} u P_{\epsilon, N_2} v \|_{L^2_2 L^2} & \lesssim \| P_{\epsilon,N_1} u \|_{L^4 L^\infty} \| P_{\epsilon,N_2} u \|_{L^4 L^2} \\
& \lesssim \epsilon^{-d/2} N_1^{d/2} \| P_{\epsilon,N_1} u \|_{L^4 L^2} \| P_{\epsilon,N_2} v \|_{L^4 L^2} \\
& \lesssim \epsilon^{-d/2}  N_1^{d/2} \| P_{\epsilon,N_1} u \|_{X^{0,\frac{1}{4}+\delta}_{\epsilon}}  \| P_{\epsilon,N_2} v \|_{X^{0,\frac{1}{4}+\delta}_{\epsilon}} 
\end{align*}

\bigskip

\noindent
\underline{Step 3: interpolation} Interpolating between the results of Step 1 and Step 2 gives the desired result.
\end{proof}

We can now proceed with the proof of~\eqref{goldfinch}. Omitting complex conjugation in the notations for simplicity, we split all the functions into dyadic frequency projections to obtain
$$
 \iint \chi(t) |u|^2 u v\,dx\,dt =  \sum_{N_1,N_2,N_3,N_4 \in 2^{\mathbb{N}_0}} \iint  \chi(t) P_{\epsilon,N_1} u \, P_{\epsilon,N_2} u  \, P_{\epsilon,N_3} u \,  P_{\epsilon,N_4} v  \, dx\,dt.
$$
Without loss of generality, we can assume that $N_1 \leq N_2 \leq N_3$. We distinguish two cases.

\bigskip

\noindent
\underline{Case 1: $N_3 \sim N_4$} By Cauchy-Schwarz' inequality followed by Lemma~\ref{robin},
\begin{align*}
& \left|  \iint   \chi(t)  P_{\epsilon,N_1} u \, P_{\epsilon,N_2} u  \, P_{\epsilon,N_3} u \,  P_{\epsilon,N_4} v  \, dx\,dt \right| \leq \|P_{\epsilon,N_1} u \, P_{\epsilon,N_3} u\|_{L^2_2 L^2} \| P_{\epsilon,N_2} u \,  P_{\epsilon,N_4} v \|_{L^2_2 L^2} \\
& \quad \leq \epsilon^{- d+2-\kappa}  N_1^{\frac{d}{2}-1 + \kappa} N_2^{\frac{d}{2}-1 + \kappa} \|P_{\epsilon,N_1} u \|_{X_{\ep}^{0,b_0}}  \|P_{\epsilon,N_2} u \|_{X_{\ep}^{0,b_0}}  \|P_{\epsilon,N_3} u \|_{X_{\ep}^{0,b_0}}  \|P_{\epsilon,N_4} v \|_{X_{\ep}^{0,b_0}} \\
& \quad \leq \epsilon^{- d+2-\kappa}  N_1^{\frac{d}{2}-1 + \kappa - s} N_2^{\frac{d}{2}-1 + \kappa - s} \underbrace{ N_3^{-s} N_4^{s}}_{\displaystyle \sim 1}  \|P_{\epsilon,N_1} u \|_{X_{\ep}^{s,b_0}}  \|P_{\epsilon,N_2} u \|_{X_{\ep}^{s,b_0}}  \|P_{\epsilon,N_3} u \|_{X_{\ep}^{s,b_0}}  \|P_{\epsilon,N_4} v \|_{X_{\ep}^{-s,b_0}} .
\end{align*}
It is now easy to conclude that
$$
\sum_{\substack{N_1 \leq N_2 \leq N_3 \\ N_3 \sim N_4}}  \left|  \iint \chi(t) P_{\epsilon,N_1} u \, P_{\epsilon,N_2} u  \, P_{\epsilon,N_3} u \,  P_{\epsilon,N_4} v  \, dx\,dt \right| \lesssim \epsilon^{- d+2-\kappa} \| u \|_{X_{\ep}^{s,b_0}}^3 \| v \|_{X_{\ep}^{-s,b_0}}.
$$
Indeed, the variables $N_1$ and $N_2$ simply contribute a geometric series as $s>d/2-1$, while the sum over $N_3 \sim N_4$ can be bounded by the Cauchy-Schwarz inequality.

\bigskip

\noindent
\underline{Case 2: $N_4 \ll N_3$} By the same arguments as in Case 1,
\begin{align*}
& \left|  \iint \chi(t) P_{\epsilon,N_1} u \, P_{\epsilon,N_2} u  \, P_{\epsilon,N_3} u \,  P_{\epsilon,N_4} v  \, dx\,dt \right| \leq \|P_{\epsilon,N_1} u \, P_{\epsilon,N_3} u\|_{L^2_2 L^2} \| P_{\epsilon,N_2} u \,  P_{\epsilon,N_4} v \|_{L^2_2 L^2} \\
& \qquad  \leq \epsilon^{- d+2-\kappa}  N_1^{\frac{d}{2}-1 + \kappa} N_4^{\frac{d}{2}-1 + \kappa} \|P_{\epsilon,N_1} u \|_{X_{\ep}^{0,b_0}}  \|P_{\epsilon,N_2} u \|_{X_{\ep}^{0,b_0}}  \|P_{\epsilon,N_3} u \|_{X_{\ep}^{0,b_0}}  \|P_{\epsilon,N_4} v \|_{X_{\ep}^{0,b_0}} \\
& \qquad  \leq \epsilon^{- d+2-\kappa} N_1^{\frac{d}{2}-1 + \kappa - s} N_4^{\frac{d}{2}-1 + \kappa + s} N_2^{-s} N_3^{-s}  \|P_{\epsilon,N_1} u \|_{X_{\ep}^{s,b_0}}  \|P_{\epsilon,N_2} u \|_{X_{\ep}^{s,b_0}}  \|P_{\epsilon,N_3} u \|_{X_{\ep}^{s,b_0}}  \|P_{\epsilon,N_4} v \|_{X_{\ep}^{-s,b_0}}.
\end{align*}
The condition $N_4 \ll N_3$ and quasi-orthogonality implies $N_2 \sim N_3$, and
$$
\sum_{\substack{N_1 \leq N_2 \leq N_3 \\ N_4 \ll N_3 \sim N_2}} N_1^{\frac{d}{2}-1 + \kappa - s} N_4^{\frac{d}{2}-1 + \kappa + s} N_2^{-s} N_3^{-s} < \infty.
$$
As a consequence,
$$
\sum_{\substack{N_1 \leq N_2 \leq N_3 \\ N_4 \ll N_3}} \left|  \iint \chi(t) P_{\epsilon,N_1} u \, P_{\epsilon,N_2} u  \, P_{\epsilon,N_3} u \,  P_{\epsilon,N_4} v  \, dx\,dt \right| \lesssim \epsilon^{- d+2-\kappa} \| u \|_{X_{\ep}^{s,b_0}}^3  \| v \|_{X_{\ep}^{-s,b_0}}.
$$
The estimate~\eqref{goldfinch} now follows by combining Case 1 and Case 2, and by choosing $b$ such that $ \frac{1}{2} < b < 1-b_0$.

\subsection{Proof of Proposition~\ref{propbilinear}: the bilinear bound}
Recall the identity $X^{0,b}=X^{0,b}_{\ep}$. First note that, by interpolating between $\| v \|_{L^\infty L^2} \lesssim \| v \|_{X^{0,b}}$, for $b> \frac{1}{2}$, and $\| v \|_{L^2 L^2} = \| v \|_{X^{0,0}}$, one obtains that for any $r<\infty$, there exists $\widetilde{b} < \frac 12$ such that
\begin{equation}
\label{alouette}
\| v \|_{L^r L^2} \lesssim \| v \|_{X^{0,\widetilde b}}.
\end{equation}
Proceeding as in the previous subsection, we are to bound
\begin{align*}
\sup_{\| v\|_{X^{-s,1-b}} \leq 1}
\iint \chi(t) v^{app} |u|^2 \overline{v} \,dx \,dt.
\end{align*}
Next, we omit complex conjugation signs for simplicity, and localize the functions appearing in the above integral dyadically in frequency, at $N_1$, $N_2$, and $N_3$ for $u$, $u$ and $v$ respectively. By symmetry, we can assume that $N_1 \leq N_2$; and since $v^{app}$ is localized in Fourier on $B(0,C \epsilon^{-1})$, we can assume that $N_3 \lesssim N_2$. Therefore, it suffices to bound
$$
\sum_{\substack{ N_1 \leq N_2 \\ N_3 \lesssim N_2}} \left| \iint \chi(t) v^{app} P_{N_1} u P_{N_2} u P_{N_3} v \,dx \,dt \right|.
$$
Applying successively H\"older's inequality with $\frac 1r +\frac 1q=\frac 12$, Lemma~\ref{robin} and inequalities~\eqref{mesangecharbonniere} and~\eqref{alouette}, this is (for $\kappa$ small enough):
\begin{align*}
\dots & \lesssim \| v^{app} \|_{L^q_T L^\infty} \sum_{\substack{ N_1 \leq N_2 \\ N_3 \lesssim N_2}} \left\| P_{N_1} u P_{N_2} u\right\|_{L^2_2 L^2} \| P_{N_3} v \|_{L^r L^2} \\
& \lesssim  \epsilon^{-\frac{d}{2} + \frac{1}{2} -\mu- \kappa} \sum_{\substack{ N_1 \leq N_2 \\ N_3 \lesssim N_2}}N_1^{\frac{d}{2} - 1 + \kappa}  \| P_{N_1} u \|_{X^{0,b_0}} \| P_{N_2} u \|_{X^{0,b_0}} \| P_{N_3} v \|_{X^{0,\widetilde b}} \\
& \lesssim  \epsilon^{-\frac{d}{2} + \frac{1}{2} -2\mu} \sum_{\substack{ N_1 \leq N_2 \\ N_3 \lesssim N_2}} N_1^{\frac{d}{2} - 1 + \kappa - s} N_2^{-s} N_3^{s}  \| P_{N_1} u \|_{X_\epsilon^{s,b_0}} \| P_{N_2} u \|_{X_\epsilon^{s,b_0}} \| P_{N_3} v \|_{X_\epsilon^{s,\widetilde b}}.
\end{align*}
Summing the geometric series in $N_1$, and applying Cauchy-Schwarz in $N_2$ and $N_3$, this is
$$
\dots \lesssim \epsilon^{-\frac{d}{2} + \frac{1}{2} -2\mu} \| u \|_{X^{s,b_0}_\epsilon}^2 \| v \|_{X^{s,\widetilde b}_\epsilon}.
$$
There remains to fix $q$ as close to $2$ as desired, choose $\widetilde{b}<\frac{1}{2}$ which allows for~\eqref{alouette}, and finally choose $b>\frac{1}{2}$ such that $1-b > \widetilde b$. 

\appendix

\section{Basics of $X^{s,b}$ spaces}

\label{Xsbbasics}

These spaces were introduced in~\cite{Bourgain}. We quickly review their properties, refering the reader to~\cite{Tao}, Section 2.6, for details. 

\bigskip

\noindent \underline{Definition} Let
$$
\| f \|_{H^s_\epsilon} = \| \langle \epsilon D \rangle^s f \|_{L^2}
$$
and
$$
\| u \|_{X^{s,b}_\epsilon} = \| e^{-it\Delta} u(t) \|_{L^2 H^s_\epsilon} =  \| \langle \epsilon k \rangle^s \langle \tau +|k|^2 \rangle^b \widetilde{u}(\tau,k) \|_{L^2(\mathbb{R} \times \mathbb{Z}^2)}
$$

\bigskip

\noindent \underline{Time continuity} For $b > \frac{1}{2}$,
\begin{equation}
\label{cormorant}
\| u \|_{\mathcal{C} H^s_\epsilon} \lesssim \| u \|_{X^{s,b}_{\epsilon}}.
\end{equation}

\bigskip

\noindent \underline{Inverting the linear Schr\"odinger equation} Assume that $u$ solves
$$
\left\{
\begin{array}{l}
i \partial_t u - \Delta u = F \\ u(t=0) = 0
\end{array}
\right.
$$
Then, denoting $\chi$ for a smooth cutoff function, supported on $B(0,2)$, and equal to $1$ on $B(0,1)$,
\begin{equation}
\label{mesangebleue0}
\| \chi(t) u \|_{X_\epsilon^{s,b-1}} \lesssim \| F \|_{X_\epsilon^{s,b}}.
\end{equation}

\bigskip 

\noindent
\underline{From group to $X^{s,b}$ estimates} Assume that, uniformly in $\tau_0 \in \mathbb{R}$,
$$
\| e^{it\tau_0} e^{it \frac \Delta 2} f \|_{Y} \leq C_0(\epsilon) \| \langle \epsilon D \rangle^s f \|_{L^2}
$$
Then, if $b > \frac{1}{2}$,
$$
\| u \|_{Y} \lesssim_b C_0(\epsilon) \| u \|_{X^{s,b}_\epsilon}
$$

\bigskip

\noindent
\underline{Strichartz estimates} We want to apply the previous statement to Strichartz estimates: it was proved in~\cite{Bourgain} that, for $s > 0$ and $\kappa>0$,
\begin{align*}
& \| e^{it \Delta } f \|_{L^4_{1} L^4} \lesssim_{\kappa,s} \epsilon^{-\kappa} \| \langle \epsilon D \rangle^s f \|_{L^2} \qquad \mbox{if $d=2$} \\
& \| e^{it \Delta } f \|_{L^4_{1} L^4} \lesssim_{\kappa,s} \epsilon^{\frac{1}{2} - \frac{d}{4}} \| \langle \epsilon D \rangle^{\frac{1}{2} - \frac{d}{4}} f \|_{L^2} \qquad \mbox{if $d\geq 3$}.
\end{align*}
As a consequence, if $s> 0$, $\kappa>0$, $b > \frac{1}{2}$,
\begin{equation}
\label{mesangebleue1}
\begin{split}
& \| u \|_{L^{4}_{1} L^{4}} \lesssim_{\kappa,s,b} \epsilon^{-\kappa} \| u  \|_{X^{s,b}_\epsilon} \qquad \mbox{if $d=2$} \\
& \| u \|_{L^{4}_{1} L^{4}} \lesssim_{b} \epsilon^{\frac{1}{2} - \frac{d}{4}} \| u \|_{X_\epsilon^{\frac{d}{4} - \frac{1}{2},b}} \qquad \mbox{if $d\geq 3$}.
\end{split}
\end{equation}
It will also be useful to localize Strichartz estimates through frequency projectors: if $d \geq 3$,
\begin{align*}
& \| P_{\epsilon,N} e^{it\Delta} f \|_{L^4 L^4} \lesssim \left( \frac{N}{\epsilon} \right)^{\frac{d}{4}-\frac{1}{2}} \| P_{\epsilon,N} f \|_{L^2} \\
& \| P_{\epsilon,N} u \|_{L^4 L^4} \lesssim \left( \frac{N}{\epsilon} \right)^{\frac{d}{4}-\frac{1}{2}} \| u \|_{X^{0,b}},
\end{align*}
with an additional $\kappa$ loss if $d=2$, and identical statements if $P_{\epsilon,N}$ is replaced by $Q_{\epsilon,N}^n$.

\bigskip

\noindent \underline{Duality} The dual of $X^{s,b}_\epsilon$ is $X^{-s,-b}_\epsilon$. Therefore, the previous inequalities implies that, if $s'<0$, $\kappa>0$, $b' <- \frac{1}{2}$,
\begin{equation}
\label{mesangebleue2}
\begin{split}
& \| \chi(t) u \|_{X^{s',b'}_\epsilon} \lesssim_{\kappa,s',b'}  \epsilon^{-\kappa} \| u \|_{L^{4/3}_{1} L^{4/3}}  \qquad \mbox{if $d=2$} \\
& \| \chi(t) u \|_{X_\epsilon^{-\frac{d}{4} + \frac{1}{2},b'}} \lesssim_{b'}  \epsilon^{\frac{1}{2} - \frac{d}{4}} \| u \|_{L^{4/3}_{1} L^{4/3}}\qquad \mbox{if $d\geq 3$}.
\end{split}
\end{equation}
Similarly, the dual of the inequality~\eqref{cormorant} is, for any $b' < \frac{1}{2}$,
\begin{equation}
\label{mesangebleue3}
\| u \|_{X^{s,b'}_{\epsilon}} \lesssim_{b'} \| u \|_{L^1 H^s_\epsilon}
\end{equation}

\bigskip

\noindent \underline{Interpolation} If $0 \leq \theta \leq 1$, $s = \theta s_0 + (1-\theta) s_1$ and $b = \theta b_0 + (1-\theta) b_1$,
$$
\| u \|_{X^{s,b}} \leq \| u \|_{X^{s_0,b_0}}^\theta \| u \|_{X^{s_1,b_1}}^{1-\theta}. 
$$

\section{Summing at a vertex}

\begin{lemma}[Degree two vertex] \label{lemmavertextwo} \label{macareux1} For any $\epsilon \leq 1$, $k_0 \in \mathbb{Z}^d$ with $|k_0| \leq \epsilon^{-1}$, $\alpha \in \mathbb{R}$, $\beta \geq 1$, $\sigma, \sigma' \in \{ \pm 1 \}$ with $(\sigma, \sigma')\neq (1,1)$, and $\kappa>0$,
$$
\sum_{\substack{k,k' \in \mathbb{Z}^d \\ |k|,|k'| < \epsilon^{-1}} } \frac{1}{||k|^2 + \sigma' |k'|^2 + \sigma |k_0 - k - k'|^2 - \alpha + i \beta|} \lesssim_\kappa \epsilon^{2-2d-\kappa}
$$
In particular,
$$
\# \{ (k,k') \; \mbox{such that} \; |k| , |k'| < \epsilon^{-1} \; \mbox{and} \; ||k|^2 + \sigma' |k'|^2 + \sigma |k_0 - k - k'|^2 - \alpha + i \beta| \leq \beta \} \lesssim_\kappa \beta  \epsilon^{2-2d-\kappa}.
$$
\end{lemma}

\begin{proof} Denoting $Q(k,k') = |k|^2 + \sigma' |k'|^2 + \sigma |k_0 - k - k'|^2$, the above left-hand side can be bounded by
$$
\dots \lesssim \sum_{\substack{j \in \mathbb{N} \\ \beta < 2^j \beta \lesssim  \epsilon^{-2}}} (2^{j} \beta)^{-1} \# \left\{ (k,k') \; \mbox{such that} \; |k| , |k'| < \epsilon^{-1} \; \mbox{and} \; |Q(k,k') - \alpha + i \beta| \leq 2^{j} \beta \right\}.
$$
Therefore, it suffices to show that
$$
 \# \left\{ (k,k') \; \mbox{such that} \; |k| , |k'| < \epsilon^{-1} \; \mbox{and} \; |Q(k,k') - \alpha| \leq 2^{j}\beta \right\} \lesssim_\kappa (2^j \beta)  \epsilon^{2-2d-\kappa},
$$
which follows from
$$
 \# \left\{ (k,k') \; \mbox{such that} \; |k| , |k'| < \epsilon^{-1} \; \mbox{and} \; Q(k,k') = m \right\} \lesssim_\kappa \epsilon^{2-2d-\kappa}
$$
(where $m \in \mathbb{Z}$).
We now examine all possible values of $\sigma$ and $\sigma'$.

\bigskip

\noindent \underline{Case 1: $\sigma = 1,\sigma' = -1$.} In this case, $Q(k,k') = 2 (k-k_0) \cdot (k+k') + |k_0|^2$. To show that $Q(k,k') = m$ has $\lesssim_\kappa \epsilon^{2-2d-\kappa}$ solutions, one proceeds as in Case 1: one picks $k_1,\dots,k_{d-1},k'_{1},\dots,k'_{d-1}$, which gives $\epsilon^{2-2d}$ possibilities. There remains to pick $k_d$ and $k'_d$ such that their product is a fixed number; by the divisor bound, this contributes a further $\epsilon^{-\kappa}$.

\bigskip

\noindent \underline{Cases 2 and 3: $\sigma = -1,\sigma' = 1$ and $\sigma = - 1,\sigma' = -1$.} In the former case, $Q(k,k') = 2(k-k_0) \cdot (k_0 - k') + |k_0|^2$, and in the latter, $Q(k,k') = - |k_0|^2 + 2(k+k') \cdot (k_0 - k')$. Either way, the proof of Case 2 applies.
\end{proof}

\begin{lemma}[Degree one vertex] \label{colombe} \label{macareux2} For any $\epsilon \leq 1$, $k_0 \in \mathbb{Z}^d$ with $ |k_0| \leq \epsilon^{-1}$, $\alpha \in \mathbb{R}$, $\beta \geq 1$, and $\kappa>0$,
$$
\sum_{\substack{k \in \mathbb{Z}^d \\ |k| < \epsilon^{-1}} } \frac{1}{||k|^2 + |k_0 - k|^2 - \alpha + i \beta|} \lesssim_\kappa \epsilon^{1-d-\kappa}.
$$
If moreover $k_1\in \mathbb Z^d$ is such that $|k_1|\lesssim \ep^{-1}$:
\be \label{bd:degree1estimate}
\sum_{\substack{k \in \mathbb{Z}^d \\ |k| < \epsilon^{-1}} } \frac{\left|1-\delta (k_0)-\delta(k_0+k_1-k)\right|}{||k|^2 - |k_0 - k|^2 - \alpha + i \beta|} \lesssim_\kappa \left\{ \begin{array}{l l} \epsilon^{1-d-\kappa} \quad \mbox{if }k_0\neq 0,\\ 1 \quad \mbox{if }k_0= 0 \end{array} \right. \lesssim \epsilon^{1-d-\kappa}
\ee
In particular, for any $\sigma\in \{\pm 1\}$, assuming $k_0\neq 0$ for $\sigma=-1$:
$$
\# \{ k \; \mbox{such that} \; |k| < \epsilon^{-1} \; \mbox{and} \; ||k|^2 + \sigma |k_0 - k |^2 - \alpha + i \beta| \leq \beta \} \lesssim_\kappa \beta \epsilon^{1-d-\kappa}.
$$
\end{lemma}

\begin{proof} The first estimate can be dealt with as in Case 1 of the proof of Lemma~\ref{lemmavertextwo}, and gives a bound $\epsilon^{2-d-\kappa} \lesssim \epsilon^{1-d+\kappa}$ on the right-hand side. We then turn to the second estimate to be proved, in which case $|k|^2 - |k_0 - k|^2 = k_0 \cdot (2k - k_0)$, which we will denote $Q(k)$.\\

\noindent \underline{Case 1: $k_0\neq 0$} We first assume $k_0\neq 0$. In order to obtain the desired bound, it suffices to show that, for $j \in \mathbb{N}$,
$$
 \# \left\{ k \; \mbox{such that} \; |k| < \epsilon^{-1} \; \mbox{and} \; |Q(k) - \alpha| \leq 2^{j}\beta \right\} \lesssim_\epsilon (2^j \beta)  \epsilon^{1-d}.
$$
But an elementary argument shows that the number of solutions of $|k_0 \cdot k - \alpha| < 2^j \beta$ is $\lesssim \left(  \frac{2^j \beta}{|k_0|} + 1 \right) \epsilon^{1-d} \lesssim  \left( 2^j \beta  \right)\epsilon^{1-d}$.\\

\noindent \underline{Case 2: $k_0= 0$} We now assume $k_0= 0$. In that case, notice that the numerator forces:
$$
\left|1-\delta (k_0)-\delta(k_0+k_1-k)\right|=\delta (k_0+k_1-k).
$$
Therefore, the sum is trivial equal to $\frac{1}{|-\alpha + i\beta|}$ as it is made of the only element $k_0+k_1$. The desired bound holds as $\frac{1}{\beta} \lesssim 1 \lesssim \epsilon^{1-d-\kappa} $ due to the dimensional assumption $d\geq 2$.

\end{proof}

\end{document}